\documentclass{amsart}
\usepackage{color}

\definecolor{light-gray}{gray}{0.9}
\usepackage{ifxetex,ifluatex}

\IfFileExists{microtype.sty}{\usepackage{microtype}}{}
\ifxetex
  \usepackage[setpagesize=false, 
              unicode=false, 
              xetex]{hyperref}
\else
  \usepackage[unicode=true]{hyperref}
 \hypersetup{breaklinks=true,
            bookmarks=true,
            pdfauthor={},
            pdftitle={},
            colorlinks=true,
            citecolor=blue,
            urlcolor=blue,
            linkcolor=magenta,
            pdfborder={0 0 0}}
\fi

\usepackage{amsmath}
\usepackage{a4wide}
\usepackage{url}
\usepackage{graphicx}
\usepackage{xy}

\usepackage{exscale}

\xyoption{all}

\font\myfont=cmr6
\def\st#1 {  {\hbox{\myfont \begin{tabular}{c} {\, #1}  \end{tabular}  }} }
\def\state#1#2 {  {\hbox{\myfont \begin{tabular}{c}\noalign{\vskip -0.1cm} {\; #1}
\\[-0.5mm] \noalign{\vskip -0.1cm} {\; #2} \end{tabular}  }} }
\def\statev#1#2#3 {  {\hbox{\myfont \begin{tabular}{c}\noalign{\vskip -0.0cm} {\; #1}
\\[-2.0mm] \noalign{\vskip -0.0cm} {\; #2} \\[-2.0mm] \noalign{\vskip -0.0cm} {\; #3} \end{tabular}  }} }

\font\myfonta=cmr10
\def\vect#1 {  {\hbox{\myfonta \begin{tabular}{c} {\, #1}  \end{tabular}  }} }

\font\myfontb=cmr10
\def\vectv#1#2#3 {  {\hbox{\myfontb \begin{tabular}{c}\noalign{\vskip 0.1cm} {\; #1}
\\[-1.1mm] \noalign{\vskip -0.0cm} {\; #2} \\[-1.1mm] \noalign{\vskip -0.0cm} {\; #3} \end{tabular}  }} }

\newtheorem{theorem}{Theorem}[section]
\newtheorem{proposition}[theorem]{Proposition}
\newtheorem{corollary}[theorem]{Corollary}
\newtheorem{lemma}[theorem]{Lemma}

\newtheorem{conjecture}{Conjecture}
\theoremstyle{definition}
\newtheorem{definition}[theorem]{Definition}
\theoremstyle{remark}
\newtheorem{example}[theorem]{Example}
\newtheorem{remark}[theorem]{Remark}

\makeatletter
\let\c@equation\c@theorem
\makeatother
\numberwithin{equation}{section}

\newcommand{\eff}{q}

\newcommand{\complex}[1]{{\mathcal{K}}{(#1)}}
\newcommand{\complexzero}[1]{{\mathcal{K}}{(#1)}_0}

\newcommand{\faces}[2]{\mathcal{F}_{{#2}}({#1})}

\newcommand{\geometric}[1]{{\overline{{#1}}}}

\newcommand{\cell}[1]{{\langle{{#1}}\rangle}}
\newcommand{\cellstar}[1]{{\cell{#1}}_\start}

\newcommand{\cellstark}{{\cell{S}}_{{\start}_{k}}}

\newcommand{\boundarystarstar}[2]{\partial\operatorname{-star}_{#2}({#1})}

\newcommand{\inv}{^{-1}}
\newcommand{\digits}{\mathcal{D}}
\newcommand{\zn}{\mathbb{Z}^n}
\newcommand{\rn}{\mathbb{R}^n}

\newcommand{\canonical}{{Q}}

\newcommand{\repmap}[1]{{{#1}_T}}
\newcommand{\repmapf}[1]{{{#1}}_\start}

\newcommand{\stage}[2]{{#1}_{#2}}

\newcommand{\stagestar}[2]{{#1}_{M,{#2}}}
\newcommand{\limpt}[1]{{#1}_T}
\newcommand{\limptstar}[1]{{{{#1}}_\start}}

\newcommand{\Bee}{F}
\newcommand{\start}{M}

\title{Self-affine Manifolds}

\author{Gregory R.\ Conner}
\address{Department of Mathematics, Brigham Young University, Provo, UT 84602, USA}

\author[J\"org M. Thuswaldner]{J\"org M. Thuswaldner$^\ast$}
\address{Department of Mathematics and Information Technology, University of Leoben, Franz-Josef-Strasse 18, A-8700 Leoben, Austria}

\thanks{The first author was supported by the Simons Foundation Collaboration Grant \#246221}

\thanks{The second author was supported by project I1136 funded by the Agence National de la Recherche and the Austrian Science Fund, and by project W1230  funded by the Austrian Science Fund}

\thanks{$ ^\ast$ Corresponding author. Tel.: +43 3842 402 3805. Fax: +43 3842 402 3802. \\  {\em E-mail address:} joerg.thuswaldner@unileoben.ac.at}



\date{\today}
\begin{document}

\begin{abstract}
This paper studies closed 3-manifolds which are the attractors of a system of finitely many affine contractions that tile $\mathbb{R}^3$.  Such attractors are called self-affine tiles. Effective characterization and recognition theorems for these 3-manifolds as well as theoretical generalizations of these results to higher dimensions are established. The methods developed build a bridge linking geometric topology with iterated function systems and their attractors.

A method to model self-affine tiles by simple iterative systems is developed in order to study their topology.  The model is functorial in the sense that there is an easily computable map that induces isomorphisms between the natural subdivisions of the attractor of the model and the self-affine tile.  It has many beneficial qualities including ease of computation allowing one to determine topological properties of the attractor of the model such as connectedness and whether it is a manifold.  The induced map between the attractor of the model and the self-affine tile is a quotient map and can be checked in certain cases to be monotone or cell-like. Deep theorems from geometric topology are applied to characterize and develop algorithms to recognize when a self-affine tile is a topological or generalized manifold in all dimensions. These new tools are used to check that several self-affine tiles in the literature are 3-balls.  An example of a wild 3-dimensional self-affine tile is given whose boundary is a topological 2-sphere but which is not itself a 3-ball.  The paper describes how any 3-dimensional handlebody can be given the structure of a self-affine 3-manifold.  It is conjectured that every self-affine tile which is a manifold is a handlebody.
\end{abstract}

\maketitle

\setcounter{tocdepth}{1}
\tableofcontents

\section{Introduction}
A great deal of work in the literature has concentrated on tilings of
$\mathbb{R}^n$ whose tiles are defined by a finite collection of contractions.  One of
the most prevalent examples are tilings by \emph{self-affine tiles}
where the contractions are affine translates of a single linear
contraction. A long-standing open question is whether there exists a closed 3-manifold which is a nontrivial self-affine tile, a so-called {\em self-affine 3-manifold}. To settle this question in the affirmative the current paper effectively characterizes and recognizes self-affine 3-manifolds and gives theoretical generalizations of these results to higher dimensions. The methods developed in this paper build a bridge linking two previously unrelated areas of mathematics: geometric topology on the one side and iterated function systems and their attractors on the other side.

Much research is devoted to how a subset of the Euclidean space can admit a tiling by self-affine tiles. In the planar case, the topology of these tiles has been studied thoroughly. Much less is known about the topology of self-affine tiles of Euclidean 3-space.  In particular it has been an open question as to which (if any) 3-manifolds admit a nontrivial self-affine tiling of $\mathbb{R}^3$.  A number of examples have appeared in the literature which were conjectured to be self-affine tilings of $\mathbb{R}^3$ by 3-balls.  In the current paper we address these questions by describing an often effective method of determining that a given 3-dimensional self-affine tile is a tamely embedded 3-manifold.  The method gives affirmative answers for the previously conjectured examples, and is also used to give examples of 3-dimensional self-affine tiles which are handlebodies of higher genus.  Examples are also given of self-affine tiles in $\mathbb{R}^3$ whose boundaries are wildly embedded surfaces and thus are not 3-manifolds. Our method also has potential to allow effective computations in higher dimensions.  The proofs of our results require a careful formulation of the problem in terms of a certain type of algebro-geometric complexes which are used to approximate the tile and allow for arbitrarily fine computations due to their recursive structure.  Deep tools of geometric topology developed by Cannon and Edwards are then used to determine the homeomorphism type of the boundary and check if it is a tamely embedded surface in the case of dimension 3.  We offer Conjectures \ref{3dimconj} and \ref{ndimconj} stating that every self-affine manifold is homeomorphic to a handlebody.
  
The study of self-affine tiles and their tiling properties goes back to the work of Thurs\-ton~\cite{Thurston:89} and Kenyon~\cite{Kenyon:92}.  In the 1990s Lagarias and Wang~\cite{LW:96,LW:96a,LW:97} proved fundamental properties of self-affine tiles. Wang~\cite{Wang:99} surveys these early results on self-affine tiles like tiling properties, Hausdorff dimension of the boundary and relations to wavelets. By now there exists a vast literature on self-affine tiles. The topics of research include their geometric, topological, and fractal properties, characterization problems, relations to number systems, and wavelet theory (see {\it e.g.}~\cite{AL:11,BBLT:06,Curry:06,GY:06,GM:92,Lai-Lau-Rao:10,SW:99}). The topology of (mainly planar) self-affine tiles is the topic of considerable study ({\it cf.} for instance~\cite{BG:94,GH:94,HSV:94,KL:00,LW:96a}). In particular, the case where the self-affine tile is a 2-manifold ({\it i.e.}, a closed disk) has been well understood from early on ({\it cf.}~\cite{Akiyama-Thuswaldner:05,BG:94,BW:01,LL:07,LRT:02,Luo-Zhou:04}). The question of determining the topology of higher dimensional self-affine tiles arose naturally; in particular, if they could be manifolds in a nontrivial way and if there is a method to recognize whether a self-affine tile (or its boundary) is a manifold (see  Gelbrich~\cite{Gelbrich:96} who first raised this question in 1996, and more recent work in~\cite{Bandt:12,BM:09,DJN:12,LL:07}). We will call a 3-dimensional self-affine tile which is a topological 3-manifold with boundary a {\em self-affine 3-manifold}. Numerous specific examples of nontrivial self-affine tiles were conjectured in the literature to be topological 3-dimensional balls (see for instance~\cite{BM:09,Gelbrich:96}). However, until this point, no example of a nontrivial self-affine 3-manifold has been exhibited (in \cite{Malone:00} self-affine tiles that are $n$-dimensional parallelepipeds are characterized).

In the current article we give conditions that characterize self-affine 3-mani\-folds. We go on to give an effective algorithm to decide whether a 3-dimensional self-affine tile is a manifold with boundary. Finally we apply our algorithm to conjectured examples and show that they are topological 3-dimensional balls. Many of our results generalize to the $n$-dimensional case. As a self-affine manifold tiles itself by arbitrarily small copies of itself, it seems that its topology cannot be very complicated.  Indeed, we conjecture that
every self-affine manifold is homeomorphic to a handlebody (see Conjecutres~\ref{3dimconj} and~\ref{ndimconj}).

\subsection{Classical conjectures and solutions} We start with the exact definition of the fundamental objects studied in the present paper.

\begin{definition}[Self-affine tile]\label{tiledef}
Let $A$ be an expanding $n\times n$ integer matrix, that is, a matrix each of whose eigenvalues has modulus strictly greater than one. Let $\digits \subset \mathbb{Z}^n$ be a complete set of residue class representatives of $\mathbb{Z}^n/A\mathbb{Z}^n$, called the {\em digit set}.  We define the {\em self-affine tile} $T=T(A,\digits)$ as the unique nonempty compact set satisfying
\begin{equation}\label{setequation}
AT = T+ \digits.
\end{equation}
If the self-affine tile $T$ tiles $\mathbb{R}^n$ with respect to the lattice $\mathbb{Z}^n$ we say that $T$ induces a {\em self-affine tiling} and call $T$ a {\em self-affine $\mathbb{Z}^n$-tile}.
\end{definition}
The self-affine tile $T$ is well-defined because it is the unique solution of the {\it iterated function system}  
$\{\varphi_d \mid d\in \digits\}$ with $\varphi_d(x)=A\inv(x+d)$ ({\it cf.} Hutchinson~\cite{Hutchinson:81} and note that there is a norm $||\cdot||$ on $\rn$ that makes $A\inv$ a contraction; see \eqref{eq:norm} below).   In view of the results in \cite{LW:97} the tiling property in Definition~\ref{tiledef} is not a strong restriction and can easily be checked algorithmically (see for instance \cite{Vince:00}). Moreover, without loss of generality, we will assume that $0\in \digits$.

Until now it has not been known whether a nontrivial self-affine $\mathbb{Z}^n$-tile, $n>2$, could be homeomorphic to an $n$-manifold with boundary. For instance, Gelbrich~\cite{Gelbrich:96} as well as Bandt and Mesing~\cite{BM:09} give examples of self-affine $\mathbb{Z}^3$-tiles which are conjectured to be homeomorphic to 3-dimensional balls.   In this paper we give a method of checking that a self-affine tile is a manifold.  We check for the presence of an \emph{ideal tile} (Definition \ref{def:approx}) and use Theorems \ref{upperthm5} and \ref{bondingtheorem} as well as Theorem~\ref{cor:ball-algorithm} to check that classically conjectured self-affine 3-manifolds are indeed 3-manifolds.

\subsection{The characterization and recognition problems}
One of the hallmarks of a complete theory of a class of examples in mathematics is the solution of the characterization and recognition problem.  In other words, when can one formally characterize a class of examples and furthermore effectively recognize those examples given only simple data that defines them? 
The methods of the present paper go considerably beyond checking examples and allow one to characterize and recognize self-affine 3-manifolds.  In what follows we give a brief outline of our main results. Exact definitons and statements will follow from Section~\ref{sec:tilemodel} onwards.

In the case of 3-dimensional \emph{tame} tiles we give a characterization of those which are manifolds. A central object in this characterization is the notion of a {\em monotone model} $M$ of a self-affine $\mathbb{Z}^n$-tile $T$ (see Definition~\ref{def:model2}). Such a model is a compact (often polyhedral) set that tiles $\mathbb{R}^n$ by $\mathbb{Z}^n$-translates and has many topological properties in common with $T$. It admits a natural {\it quotient map} $Q:M\to T$ defined in \eqref{hdef} that is used to transfer topological properties from $M$ to $T$. 

Recall that a set $X\subset \mathbb{R}^n$ is said to be 1-locally complementary connected ({\em 1-LCC} for short) if each small loop in $\mathbb{R}^n\setminus X$ is contractible in a small subset of $\mathbb{R}^n\setminus X$ ({\it cf}.~Definition~\ref{def:1lcc} for a precise statement). Moreover, semi-contractibility is a contraction property of the point preimages of $Q$ which is defined in Definition~\ref{def:semicontr}.

\begingroup
\def\thetheorem{\ref{thm:ballchar}}
\begin{theorem}
Let $T$ be a self-affine $\mathbb{Z}^3$-tile with connected interior and $1$-LCC boundary. Then $T$ is a self-affine 3-manifold if and only if it admits a semi-contractible monotone model with a boundary that is a closed 2-manifold.
\end{theorem}
\addtocounter{theorem}{-1}
\endgroup

For detecting self-affine 3-manifolds we offer the following results. Checking that the boundary of a self-affine $\mathbb{Z}^3$-tile is a 2-sphere can be done by investigating point preimages of the quotient map $Q$. This leads to the following theorem.

\begingroup
\def\thetheorem{\ref{uppercor}}
\begin{theorem}
Let $T$ be a self-affine $\mathbb{Z}^3$-tile with connected interior which admits a monotone model $M$ whose boundary is homeomorphic to the 2-sphere $\mathbb{S}^2$. Then $\partial T$ is homeomorphic to $\mathbb{S}^2$.
\end{theorem}
\addtocounter{theorem}{-1}
\endgroup

As a second step, we algorithmically recognize which 3-dimensional tile is a 3-ball. Indeed, let $T$ be a self-affine $\mathbb{Z}^3$-tile. If $\partial T$ is a 2-sphere in $\mathbb{R}^3$ then there is an algorithmically checkable sufficient condition for $\partial T$ to be tamely embedded and, hence, for $T$ to be homeomorphic to a closed 3-ball (see Theorem~\ref{cor:ball-algorithm} for an exact statement).

Recall that a 3-dimensional handlebody is an orientable 3-manifold with boundary  containing pairwise disjoint, tamely embedded disks such that the manifold resulting from cutting along these disks is a 3-ball. Given the above, the fact that 2-dimensional self-affine manifolds are 2-disks and that Proposition~\ref{prop:g} (see also the paragraph after it) shows that every 3-dimensional handlebody is homeomorphic to a self-affine manifold, we offer the following conjecture:

\begin{conjecture}\label{I}\label{3dimconj}
Every self-affine 3-manifold is homeomorphic to a handlebody.
\end{conjecture}

For generalizations of this theory to arbitrary dimensions we refer to Section~\ref{sec:n} and Theorem~\ref{thm:ball-algorithm}.

\subsection{An illustrative example}

Although we mainly deal with self-affine $\mathbb{Z}^3$-tiles it is convenient to use a 2-dimensional example to illustrate concepts and proofs throughout the paper. We choose {\it Knuth's  twin-dragon} (see {\it e.g.}~\cite[p.~608]{Knuth:98}) which is a well-known self-affine $\mathbb{Z}^2$-tile given as the unique non-empty compact set  $K=K(A,\mathcal{D})\subset \mathbb{R}^2$ satisfying
\begin{equation}\label{eq:twin}
AK=K+\mathcal{D}
\qquad\hbox{with}\qquad 
A=\begin{pmatrix}
-1&-1\\
1&-1
\end{pmatrix} \quad\hbox{and}\quad \mathcal{D}=\left\{\begin{pmatrix}
0\\
0
\end{pmatrix},\begin{pmatrix}
1\\
0
\end{pmatrix}\right\}.
\end{equation}

Knuth's twin-dragon $K$ is depicted in Figure~\ref{fig:KT}.
\begin{figure}[ht]
\includegraphics[height=4.5cm]{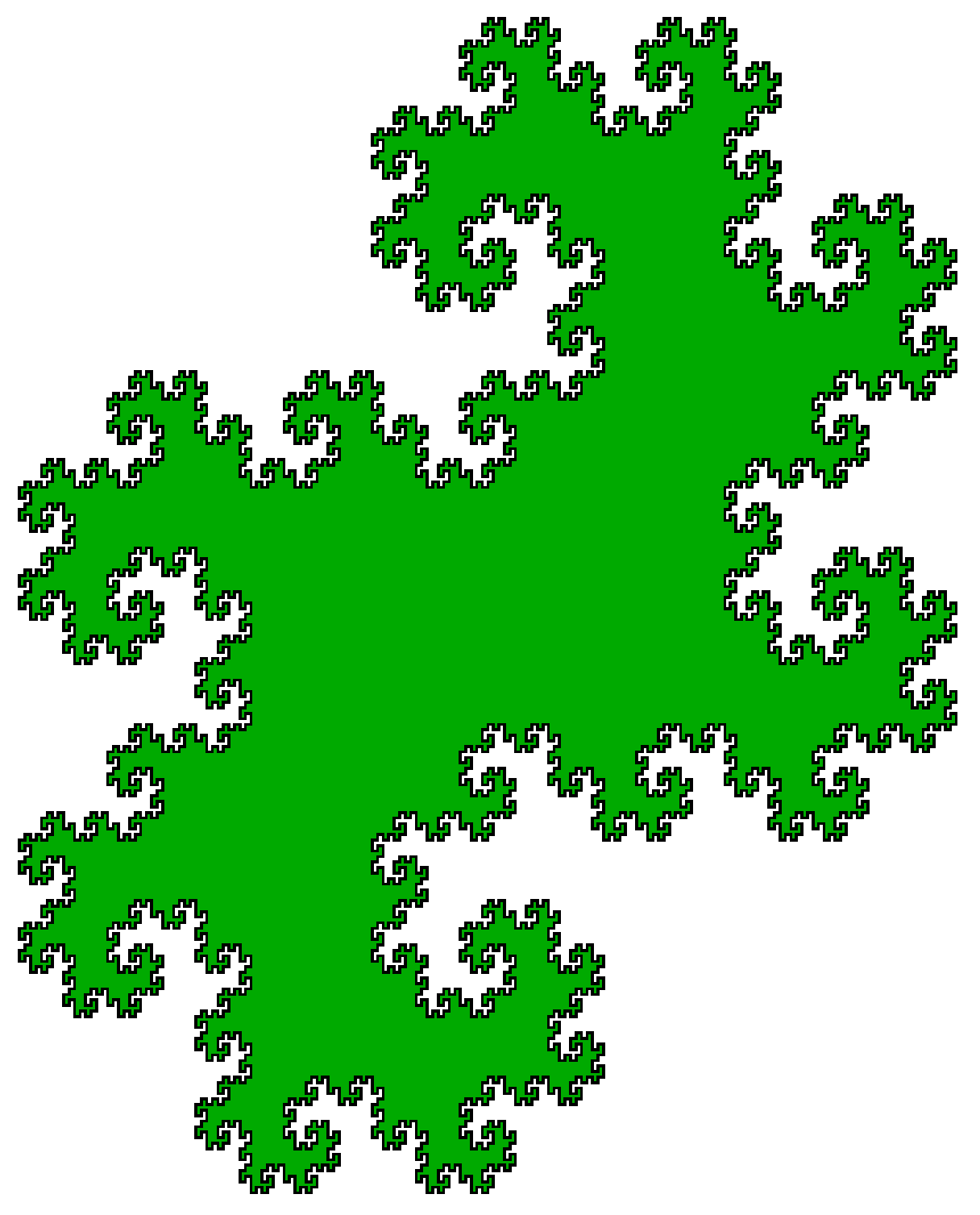}\hskip 1cm    
\caption{Knuth's twin-dragon $K$. \label{fig:KT}}
\end{figure}

\subsection{Outline of the paper}
The underlying idea behind our theory is to ``model'' a given self-affine $\mathbb{Z}^n$-tile $T$ by a set $M\subset \mathbb{R}^n$ that tiles $\mathbb{R}^n$ by $\mathbb{Z}^n$-translates and retains as many topological properties of $T$ as possible.  To understand the topology of $T$ it is then sufficient to study $M$. We give ways that allow to construct $M$ as a finite simplicial complex. This strategy will enable us to derive various new results on topological properties of self-affine $\mathbb{Z}^n$-tiles.

The definition of a \emph{model} $M$ of a self-affine $\zn$-tile $T$ is given in Section~\ref{sec:tilemodel}. In \eqref{hdef} we also define a natural \emph{canonical quotient map} $Q:M \to T$ from the model to the tile. Basic properties of $Q$ are listed in Theorem~\ref{hprop}. 

In Section~\ref{sec:nsw} we use intersections of a self-affine $\zn$-tile $T$ (or its model) with its $\zn$-translates in order to define a combinatorial complex. This \emph{tiling complex} $\complex{T}$ contains the full information on the intersection structure of the underlying tiling by $\zn$-translates of $T$ and will be of importance to derive topological properties of $T$. Moreover, we prove set equations for intersections of $T$ with its $\zn$-translates (Theorem~\ref{th:genset}) and define \emph{walks} that give coordinates to points contained in such intersections (Definition~\ref{df:walk}). Analogous definitions and results are provided for the model of $T$ (Section~\ref{sec:setm}).

In Section~\ref{sec:prop} we refine the definition of model and introduce \emph{monotone models} (Definition~\ref{def:model2}). If $M$ is a monotone model of a self-affine $\zn$-tile $T$ then the canonical quotient map $Q$ behaves nicely on intersections of tiles (Theorem~\ref{all:connectedness}). 
In this section we also introduce \emph{combinatorial} tiles: if a self-affine $\zn$-tile enjoys this general position property we can derive even more mapping properties of $Q$ (Theorem~\ref{tilingcomplex}).

Section~\ref{sec:sphere} contains our first main results. After a short section on planar tiles that are homeomorphic to a closed disk (Section~\ref{planar}), we give a criterion for a self-affine $\mathbb{Z}^3$-tile $T$ to have a boundary that is homeomorphic to a 2-manifold (Theorem~\ref{upperthm5}; see also the special case of spherical boundary in Theorem~\ref{uppercor}). If $T$ is combinatorial we can even determine the topological nature of its intersections with its $\zn$-translates (Theorem~\ref{th:complex}). In the proofs we use properties of the canonical quotient map $\canonical$ together with Moore's theorem on decompositions of 2-manifolds (Proposition~\ref{RSb}) and a criterion for a self-affine $\mathbb{Z}^3$-tile to have no separating point (Proposition~\ref{prop:nocut}).

In Section~\ref{sec:aa} we define \emph{ideal tiles} (Definition~\ref{def:approx}). As they can be constructed quite easily, together with Theorem~\ref{bondingtheorem} they provide a way to construct monotone models.

Section~\ref{sec:ball} contains the characterization theorem for self-affine 3-manifolds (Theorem~\ref{thm:ballchar}). Moreover, in this section we give an algorithm that allows to check if a given self-affine $\zn$-tile with spherical boundary is homeomorphic to an $n$-ball (Theorems~\ref{thm:ball-algorithm} and~\ref{cor:ball-algorithm}).

Section~\ref{sec:examples} contains examples for our theory. We provide two examples of self-affine $\mathbb{Z}^3$-tiles that are homeomorphic to a 3-ball (Sections~\ref{sec:tame} and~\ref{sec:Gelbrich}). Another example shows the existence of a self-affine $\mathbb{Z}^3$-tile $T$ that is a \emph{crumpled cube}, \emph{i.e.}, $\partial T$ is homeomorphic to a 2-sphere but $T$ itself is not homeomorphic to a 3-ball (Section~\ref{sec:wild}). Section~\ref{sec:torus} provides self-affine $\mathbb{Z}^3$-tiles whose boundaries are homeomorphic to a surface fo genus $g$ ($g\ge 1$).

Finally, in Section~\ref{sec:n} we show how our theory can be extended to dimension $n\ge 4$. 

\section{Models for self-affine $\mathbb{Z}^n$-tiles}\label{sec:tilemodel}

In this section we define \emph{models} for self-affine $\mathbb{Z}^n$-tiles and establish some basic results related to them. 

\subsection{Models}\label{sec:iterate} 
We will need the following notations. Let $M \subset \mathbb{R}^n$ be a compact set that is the closure of its interior and whose boundary has $\mu(\partial M)=0$, where $\mu$ denotes the $n$-dimensional Lebesgue measure. If the $\mathbb{Z}^n$-translates of $M$ cover $\mathbb{R}^n$ with disjoint interiors we say that $M$ is a \emph{$\mathbb{Z}^n$-tile}. Note that a $\mathbb{Z}^n$-tile always has $\mu(M)=1$. As a self-affine $\zn$-tile $T$ is equal to the closure of its interior and $\mu(\partial T)=0$ ({\it cf.}~\cite[Theorem~2.1]{Wang:99}), $T$ is a $\zn$-tile. 

Recall that a function $f:\mathbb{R}^n\to \mathbb{R}^n$ is called {\em $\mathbb{Z}^n$-equivariant} if $f(x+z)=f(x)+z$ holds for each $x\in\mathbb{R}^n$, $z\in\mathbb{Z}^n$.
If, in addition, we have $f(0)=0$ then $f(z)=z$ holds for each $z\in \zn$.

We can now give a fundamental definition.

\begin{definition}[Model]\label{def:model1}
A {\em model} $(M,F)$ for the self-affine $\zn$-tile $T=T(A,\digits)$ is a $\zn$-tile $\start$ equipped with a homeomorphism $\Bee:\mathbb{R}^n\to\mathbb{R}^n$ satisfying 
\begin{equation}\label{modeldef}
\Bee\start =\start + \digits
\end{equation}
such that $A\inv \Bee$ is $\mathbb{Z}^n$-equivariant and $\Bee(0)=0$.  
\end{definition}

Often the explicit knowledge of $F$ is not important. In these situations we will just write $M$ instead of $(M,F)$ for a model. Note that $\Bee$ is not assumed to be expanding in Definition~\ref{def:model1}. Thus in general, $\start$ is not uniquely defined by the set equation~\eqref{modeldef}.

\begin{example}\label{ex:model}
Let $K=K(A,\mathcal{D})$ be Knuth's twin-dragon defined in \eqref{eq:twin} and let $M$ be the parallelogram with vertices $(-\frac34,-\frac12)^t,(-\frac14,\frac12)^t,(\frac34,\frac12)^t,(\frac14,-\frac12)^t$. 
\begin{figure}[ht]
\includegraphics[height=3.5cm]{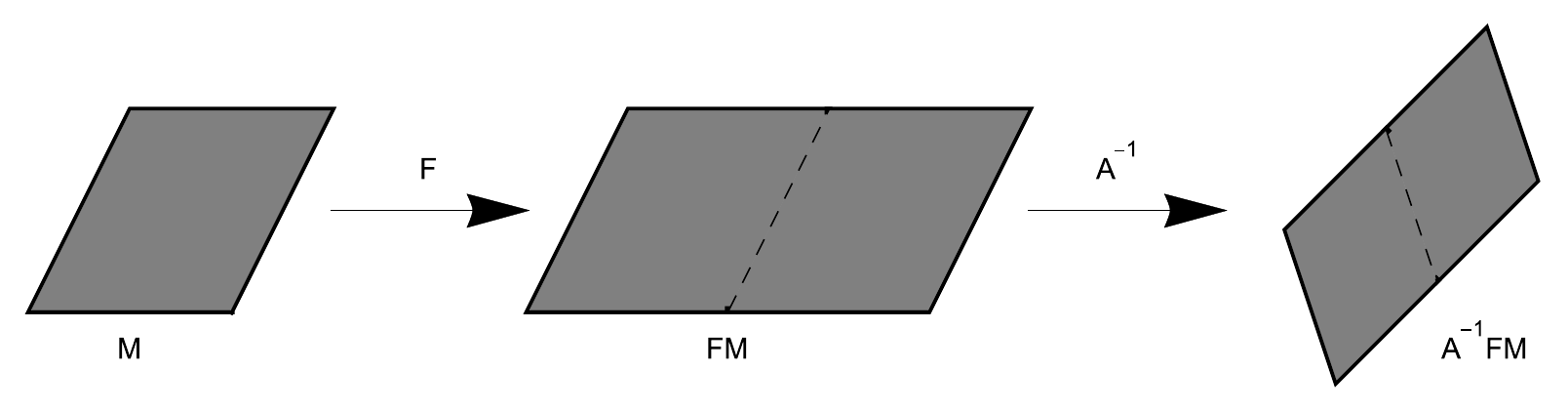}\hskip 1cm    
\caption{A model for Knuth's twin-dragon $K$. \label{fig:twinmodel}}
\end{figure}
It is possible to construct a homeomorphism $F$ with the required equivariance property explicitly in this case, so $(M,F)$ is a model for $K$ ($M$ as well as $FM$ and $A^{-1}FM$ are depicted in Figure~\ref{fig:twinmodel}). 
\end{example}

As we shall see later, the explicit construction of $F$ will not be necessary in order to find a set $M$ that makes $(M,F)$ a model for a self-affine $\mathbb{Z}^n$-tile. Indeed, it is an important feature of our theory that we can avoid explicit constructions of $F$ since such constructions are very tedious particularly for higher dimensional self-affine $\zn$-tiles. For concrete sets $M$ with certain properties (so-called \emph{ideal tiles}, see Definition~\ref{def:approx}) we are able to prove the existence of a homeomorphism $F$ that makes $(M,F)$ a model.

Let $(M,F)$ be a model for the self-affine tile $T=T(A,\digits)$. To transfer topological information from $M$ to $T$ we will make extensive use of the 
\emph{canonical quotient map} $\canonical: \mathbb{R}^n \to \mathbb{R}^n$ defined by
\begin{equation}\label{hdef}
\canonical = \lim_{k\to\infty} A^{-k}{\Bee}^{(k)}
\end{equation}
(${\Bee}^{(k)}$ denotes the $k$-th iterate of $F$).

\subsection{Basic properties of the canonical quotient map} \label{sec:basicf} 
Let $T=T(A,\digits)$ be a self-affine $\zn$-tile and $(M,F)$ a model for it.  We now give some results on models and study properties of the canonical quotient map $\canonical:\mathbb{R}^n\to\mathbb{R}^n$ defined in \eqref{hdef}. Before we start we equip $\mathbb{R}^n$ with a norm $||\cdot||$ such that the associated operator norm (also denoted by $||\cdot||$) satisfies 
\begin{equation}\label{eq:norm}
\displaystyle{||A\inv|| = \sup_{x \in \rn\setminus\{0\}} \frac{||A\inv x||}{||x||}<1.}
\end{equation}
As $A$ is expanding such a choice is possible ({\it cf.\ e.g.}~\cite[Section~3]{LW:96a}).

Our first aim is to  associate with a given model $(M,F)$ a sequence of models which converges to the tile $T$. For this reason define the set 
\begin{equation}\label{effdef}
\start_k = \eff_k \start \quad\hbox{with}\quad {{\eff}_k}= A^{-k}{\Bee}^{(k)}
\end{equation}
and let $\Bee_k=A^{-k} \Bee A^k$ for each $k\in \mathbb{N}$.

\begin{lemma}\label{efflemma}
For each $k\in\mathbb{N}$ the functions $A^{-1}\Bee_k$  and ${{\eff}_k}= A^{-k}{\Bee}^{(k)}$ are $\zn$-equivariant and fix each element of $\zn$.
\end{lemma}

\begin{proof}
Let $k\in \mathbb{N}$, $x\in \rn$, and $z\in \zn$. The $\zn$-equivariance of $A^{-1}F$ implies that 
\[
A^{-1}F_k(x+z)=A^{-k}A^{-1}F(A^kx + A^kz) = A^{-k}(A\inv F A^k x +  A^kz)=A^{-1}F_k(x) + z
\]
and $A^{-1}F_k$ is $\zn$-equivariant. Because $\Bee(0)=0$ we also have  $A^{-1}F_k(0)=0$ and, hence, by $\zn$-equivariance, $A^{-1}F_k$ fixes each element of $\zn$. Finally, as  $\eff_k=A\inv\Bee_{k-1}\circ\cdots\circ A\inv F_{0}$, also $\eff_k$ has the required properties.
\end{proof}

We now are in a position to prove the following convergence result for models. The limit of a sequence of models in its statement is taken with respect to the product of the Hausdorff metric and the metric of uniform convergence.

\begin{proposition}\label{itermodel}
Let $(M,F)$ be a model for the self-affine $\zn$-tile $T=T(A,\digits)$.
Then $(\start_k,\Bee_k)$ is a model for $T$ for each $k\in\mathbb{N}$ and $\lim_{k\to\infty}(\start_k,\Bee_k)=(T,A)$.
\end{proposition}

\begin{proof}
We first prove that $(\start_k, F_k)$ is a model for each $k\in \mathbb{N}$. We begin by showing that $M_k$ is a $\zn$-tile. The properties of $\eff_k$ asserted in Lemma~\ref{efflemma} yield that $\start_{k}= A^{-1}\eff_{k-1} F(\start) = A^{-1}(\eff_{k-1} \start + \digits)=A^{-1}(\start_{k-1}+\digits)$. Iterating this for $k$ times ($\eff_0$ is the identity) and setting 
\begin{equation}\label{digitsk}
\digits_k =\digits + A\digits + \cdots + A^{k-1}\digits 
\end{equation}
we gain 
\begin{equation}\label{eq:Mkiterate}
M_{k} = A^{-k}(M + \digits_k) \qquad (k \in \mathbb{N}).
\end{equation} 
Since $\digits$ is a complete set of residue class representatives of $\zn/ A\zn$,  $\digits_k$ is a complete set of residue class representatives of $\zn/ A^k\zn$. Thus, as $M$ is a $\zn$-tile, equation \eqref{eq:Mkiterate} implies that also $M_k$ is a $\zn$-tile. 
Next we show that $(M_k,F_k)$ satisfies the set equation \eqref{modeldef}. Indeed, as $F_k\eff _k = \eff_k F$ holds by the definition of $\eff_k$ and $F_k$, we have
\begin{equation*}
\Bee_k \start_k = \Bee_k \eff_k M =\eff_k \Bee M= \eff_k(M+ \digits)=\eff_k M + \digits=M_k+\digits.
\end{equation*}
As Lemma~\ref{efflemma} implies that $A^{-1}F_k$ is $\zn$-equivariant and $F_k(0)=0$, we have obtained that $(M_k,F_k)$ is a model for $T$ for each $k\in\mathbb{N}$.

To prove the convergence result let $\mathbf{D}(\cdot,\cdot)$ be the Hausdorff metric. For each $\varepsilon > 0$ we may choose $k \in \mathbb{N}$ in a way that $A^{-k}T$ and $A^{-k}\start$ are contained in a ball of diameter $\varepsilon$ around the origin. Using \eqref{eq:Mkiterate} and the $k$-th iterate of the set equation \eqref{setequation}, we gain
\begin{align}
\mathbf{D}(\start_{k}, T) & = \mathbf{D}(A^{-k}(\start + \digits_k), A^{-k}(T + \digits_k) ) \nonumber \\
&\le \max_{d\in \digits_k}\{\mathbf{D}(A^{-k}(\start + d), A^{-k}(T + d)) \} \label{epsdist} \\
&< \varepsilon. \nonumber
\end{align}
As $||A||>1$ and $A^{-1}F$ is continuous and $\zn$-equivariant, the maps $A^{-k} (A^{-1}F) A^{k}$ uniformly converge to the identity. Thus $\Bee_k \to A$ uniformly for $k\to\infty$ and the proof is finished.
\end{proof}

We mention that the sequence $(M_k,F_k)_{k \ge 0}$ of models yields approximations not only of the tile $T$, but also of the dynamical system defined by its set equation (see {\it e.g.} Theorem~\ref{eq:starset}). 

The following result contains basic properties of $\canonical$.

\begin{theorem}\label{upperthm}\label{hprop}
Let $(\start,F)$ be a model for the self-affine $\zn$-tile $T$. Then the {\em canonical quotient map} 
$\canonical =  \lim_{k\to\infty} A^{-k}F^{(k)}$ is continuous, $\zn$-equivariant, and satisfies the following properties.
\begin{itemize}
\item[(i)]
$ \canonical(\start) = T$, {\em i.e.}, $T$ is a quotient space of $\start$.
\item[(ii)]
$ \canonical(\partial \start) = \partial T$, {\em i.e.}, $\partial T$ is a quotient space of $\partial \start$.
\end{itemize}
\end{theorem}

\begin{proof}
We first show that $(\eff_k(x))_{k\ge 0}=(A^{-k}F^{(k)}(x))_{k\ge 0}$ is a Cauchy sequence for each $x\in\mathbb{R}^n$. By continuity and $\zn$-equivariance of $A\inv F$ there is an absolute constant $c>0$ such that
\begin{align*}
||\eff_kx-\eff_{k-1}x||&=||A^{-k}{\Bee}^{(k)}x - A^{-k+1}{\Bee}^{(k-1)}x||
\\
&=
||A^{-k+1}A\inv\Bee{\Bee}^{(k-1)}x - A^{-k+1}{\Bee}^{(k-1)}x||\\
&\leq
||A\inv||^{k-1} ||A\inv\Bee({\Bee}^{(k-1)}x) - ({\Bee}^{(k-1)}x)|| \\
&\leq c||A\inv||^{k-1}\qquad \forall x \in \rn,\; k \in \mathbb{N}.
\end{align*}
By \eqref{eq:norm} we have $||A\inv|| <1$ and, by the triangle inequality and a geometric series consideration
\begin{equation}\label{eq:Quniform}
||\eff_kx-\eff_lx||
< ||A\inv||^{l}\left(\frac{c}{1-||A\inv||}\right)\qquad \forall x \in \rn,\, k \geq l \in \mathbb{Z}.
\end{equation}
Thus  $(q_k(x))_{k\ge 0}$  is Cauchy and, hence, converges for each $x\in \mathbb{R}^n$. This defines the function $\canonical$ for each $x\in\mathbb{R}^n$. As the convergence in \eqref{eq:Quniform} is uniform in $x$ we conclude that $\canonical=\lim_{k\to \infty} q_k$ is continuous. $\zn$-equivariance of $\canonical$ follows from $\zn$-equivariance of $\eff_k$ (see Lemma~\ref{efflemma}), and (i) is an immediate consequence of the convergence statement in Proposition~\ref{itermodel}.

To prove (ii) we first note that, since $T$, $M$, and $M_k$ for $k\in\mathbb{N}$ are $\zn$-tiles, we have
\begin{equation}\label{deltapartial}
\partial T = \bigcup_{s\in\zn\setminus\{0\}} (T \cap (T+s)) \end{equation}
and (note that $q_k$ is a $\zn$-equivariant homeomorphism with $q_kM=M_k$)
\begin{equation}\label{deltapartial2}
\partial M_k = \bigcup_{s\in\zn\setminus\{0\}}(M_k \cap (M_k+s)) = q_k\bigcup_{s\in\zn\setminus\{0\}}(M \cap (M+s))=q_k\partial M.
\end{equation}
As $Q=\lim_{k\to\infty} q_k$ uniformly, (ii) follows if we show that for each $\varepsilon > 0$ there is $k_0\in \mathbb{N}$ such that 
\begin{equation}\label{boundaryhausdorff}
\mathbf{D}(\partial T, q_k \partial M) < \varepsilon
\end{equation}
holds for each $k\ge k_0$.
To prove this let $\varepsilon >0$ be arbitrary and choose $k_0$ in a way that \eqref{epsdist} holds for $k\ge k_0$. 
By \eqref{deltapartial} for each $x\in \partial T$ there is $s\in \zn\setminus\{0\}$ such that \ $x\in T$ and $x \in T + s$. From \eqref{epsdist} we gain that there exist $y_1 \in \start_k$, $y_2 \in \start_k+s$ with $||x-y_i||<\varepsilon$ ($i=1,2$). As $\start_k$ is a $\zn$-tile this implies that there is an element $y \in \partial \start_k=q_k\partial M$ with $||x-y||<\varepsilon$. By analogous reasoning, for each $x \in q_k \partial M=\partial\start_k$ there exists $y\in \partial T$ with $||x-y||<\varepsilon$. This proves \eqref{boundaryhausdorff} and, hence, also~(ii).
\end{proof}

Note that in general $\canonical$ is not injective. For any $X \subset \start$ we shall call $Q(X)$ the \emph{canonical quotient} of $X$.  In particular, Theorem~\ref{upperthm} shows that $T$ and $\partial T$ are the canonical quotients of $M$ and $\partial M$, respectively.

Although we will later often use $\canonical$ to get homeomorphisms between (subsets of) models and self-affine $\zn$-tiles we will never show that $\canonical$ itself (or certain restrictions of it) is injective. Thus $Q$ cannot serve as a homeomorphism between a model and a self-affine $\zn$-tile.

The following lemma will be needed in some computations.

\begin{lemma}\label{hproplem}
If $(\start,F)$ is a model for the self-affine $\zn$-tile $T$ then 
$A \canonical=\canonical{\Bee}$ and ${\canonical\inv}A={\Bee}{\canonical\inv}$.
\end{lemma}

\begin{proof}
The first identity follows immediately from the definition of $\canonical$. The second one follows as
\begin{align*}
{\Bee}\canonical\inv x&=\{{\Bee}y \mid \canonical y=x\}= \{{\Bee}y \mid A\canonical y=Ax\}=\{{\Bee}y\mid \canonical{{\Bee}y}=Ax\}\\
&= \{z\mid \canonical z=Ax\}={\canonical\inv}Ax. \qedhere
\end{align*}
\end{proof}

\section{Neighbor structures, set equations, and walks}\label{sec:nsw}

In this section we study neighbors of self-affine $\zn$-tiles and their models. In particular, we define \emph{neighbor structures} of these tiles and associate them with combinatorial cell complexes. We establish set equations for intersections of self-affine $\zn$-tiles (and their models) with their $\zn$-translates and define \emph{walks} which are used to address points in these tiles and their intersections.

\subsection{Cells, neighbor structures, and the tiling complex}\label{sec:tcx}
Let $M$ be a $\mathbb{Z}^n$-tile. In all what follows, the intersection of $M$  with subsets of its $\mathbb{Z}^n$-translates will play an important role.  We consider sets of $n$-tuples of integers to be coordinates for unoriented \emph{cells} in a \emph{tiling complex} defined below and accordingly introduce the notation
\begin{equation}\label{def:cell}
\cell{S}_M=\bigcap_{s \in S}(M+s) \qquad (S\subset\mathbb{Z}^n).
\end{equation}
We extend the range of $\cell{\cdot}_M$ to complexes, {\it i.e.}, to sets $\mathcal{C}$ of sets in $\mathbb{Z}^n$ by setting 
$$
\cell{\mathcal{C}}_M= \bigcup_{S\in \mathcal{C}} \cell{S}_M.
$$ 
If $M$ is the self-affine $\mathbb{Z}^n$-tile $T$,  we often omit the subscript $T$  and write $\cell{S}$ instead of $\cell{S}_T$.

\begin{definition}[Neighbor structure and tiling complex]\label{def:ns}
The \emph{neighbor structure} of a $\mathbb{Z}^n$-tile $M$ is the set 
\[
\complex{M} = \{S \subset \zn \mid S\not=\emptyset \hbox{ and } \cell{S}_M \not=\emptyset\},
\] 
which reflects the underlying intersection structure of the tiling induced by $M$ and will also be considered a formal simplicial complex. 
The \emph{faces} of $S \in \complex{M}$ are the elements of 
\[
\faces{S}{M}=\{S' \in \complex{M} \mid S \subset S' \}.
\]  
A face of $S \in \complex{M} $ is \emph{proper} if it is not equal to $S$.  We consider the set $\complex{M}$ of cells along with the notion of faces induced by $\faces{S}{M}$ to be the \emph{tiling complex} for $M$. Moreover, we set
\begin{equation*}
\complex{M}^{i} := \{ S \in \complex{M} \mid |S|=i \}
\end{equation*} 
for the set of all $i$-cells of the tiling complex $\complex{M}$.
\end{definition}

Given a simplex $S \in \complex{M}$ we define the \emph{simplicial boundary operator}
\begin{equation}\label{xstar1}
{\delta}S=\{ S\cup\{x\} \mid  x \in \mathbb{Z}^n\setminus S \} \cap \complex{M},
\end{equation}
that is, the set of maximal proper faces of $S$ in $\complex{M}$.   
We think of $\cell{{\delta}S}_M$ as a form of simplicial boundary of $\cell{S}_M$. The operator $\delta$ can be extended to a collection $\mathcal{C}$ of $i$-cells by setting
\begin{equation}\label{xstar1a}
{\delta}\mathcal{C}=\Big\{ S\cup\{x\} \mid  S \in \mathcal{C},\, x \in \mathbb{Z}^n\setminus \bigcup_{S \in \mathcal{C}}S \Big\} \cap \complex{M}.
\end{equation}

We intuit the sets $S$ as formal unoriented \emph{simplices} and $\cell{S}_M$ as their (dual) \emph{geometric realizations}. Note that it can happen that $S$ is a proper face of $S'$ but that $\cell{S}_M=\cell{S'}_M$. 

\begin{example}\label{ex:twinNeighbor}
Again we illustrate these concepts with help of Knuth's twin-dragon $K=K(A,\digits)$ defined in \eqref{eq:twin} and its model $M$ taken from Example~\ref{ex:model}. Figure~\ref{fig:twinNeighbors} shows the twin-dragon as well as its model together with its neighbors. From the left side of this figure it is easy to guess that
\begin{align*}
\complex{K}^{1} =& \; \mathbb{Z}^2, \\
\complex{K}^{2} =& \left \{  
\left \{  \tiny \begin{pmatrix} 0\\0 \end{pmatrix},\begin{pmatrix} 1\\0 \end{pmatrix}  \right\}\!,\!
\left \{  \tiny  \begin{pmatrix} 0\\0 \end{pmatrix},\begin{pmatrix} -1\\0 \end{pmatrix} \right\}\!,\!
\left \{   \tiny \begin{pmatrix} 0\\0 \end{pmatrix},\begin{pmatrix} 0\\1 \end{pmatrix}  \right\}\!,\!
\left \{   \tiny \begin{pmatrix} 0\\0 \end{pmatrix},\begin{pmatrix} 0\\ -1\end{pmatrix} \right\}\!,\!
\left \{   \tiny \begin{pmatrix} 0\\0 \end{pmatrix},\begin{pmatrix} 1\\1 \end{pmatrix}  \right\}\!,\!
\left \{   \tiny \begin{pmatrix} 0\\0 \end{pmatrix},\begin{pmatrix} -1\\-1 \end{pmatrix} \right\}\!
\right\}+\mathbb{Z}^2, \\
\complex{K}^{3} =& \Big \{   
\left \{  \tiny\begin{pmatrix} 0\\0 \end{pmatrix},\begin{pmatrix} 1\\0 \end{pmatrix},\begin{pmatrix} 1\\1 \end{pmatrix}  \right\},
\left \{  \tiny\begin{pmatrix} 0\\0 \end{pmatrix},\begin{pmatrix} 1\\1 \end{pmatrix},\begin{pmatrix} 0\\1 \end{pmatrix}  \right\},
\left \{  \tiny\begin{pmatrix} 0\\0 \end{pmatrix},\begin{pmatrix} 0\\1 \end{pmatrix},\begin{pmatrix} -1\\0 \end{pmatrix}  \right\}, \\
&\left \{ \tiny \begin{pmatrix} 0\\0 \end{pmatrix},\begin{pmatrix} -1\\0 \end{pmatrix},\begin{pmatrix} -1\\-1 \end{pmatrix}  \right\},
\left \{ \tiny \begin{pmatrix} 0\\0 \end{pmatrix},\begin{pmatrix} -1\\-1 \end{pmatrix},\begin{pmatrix} 0\\-1 \end{pmatrix}  \right\},
\left \{ \tiny \begin{pmatrix} 0\\0 \end{pmatrix},\begin{pmatrix} 0\\-1 \end{pmatrix},\begin{pmatrix} 1\\0 \end{pmatrix}  \right\}
\Big\}+\mathbb{Z}^2,
\end{align*}
and $\complex{K}^{i} = \emptyset$ for $i\ge 4$ (it is explained in Remark~\ref{rem:gifs} how to prove these well-known equalities, see also~\cite{Akiyama-Thuswaldner:05} for a proof). 
\begin{figure}[ht]
\includegraphics[height=6cm]{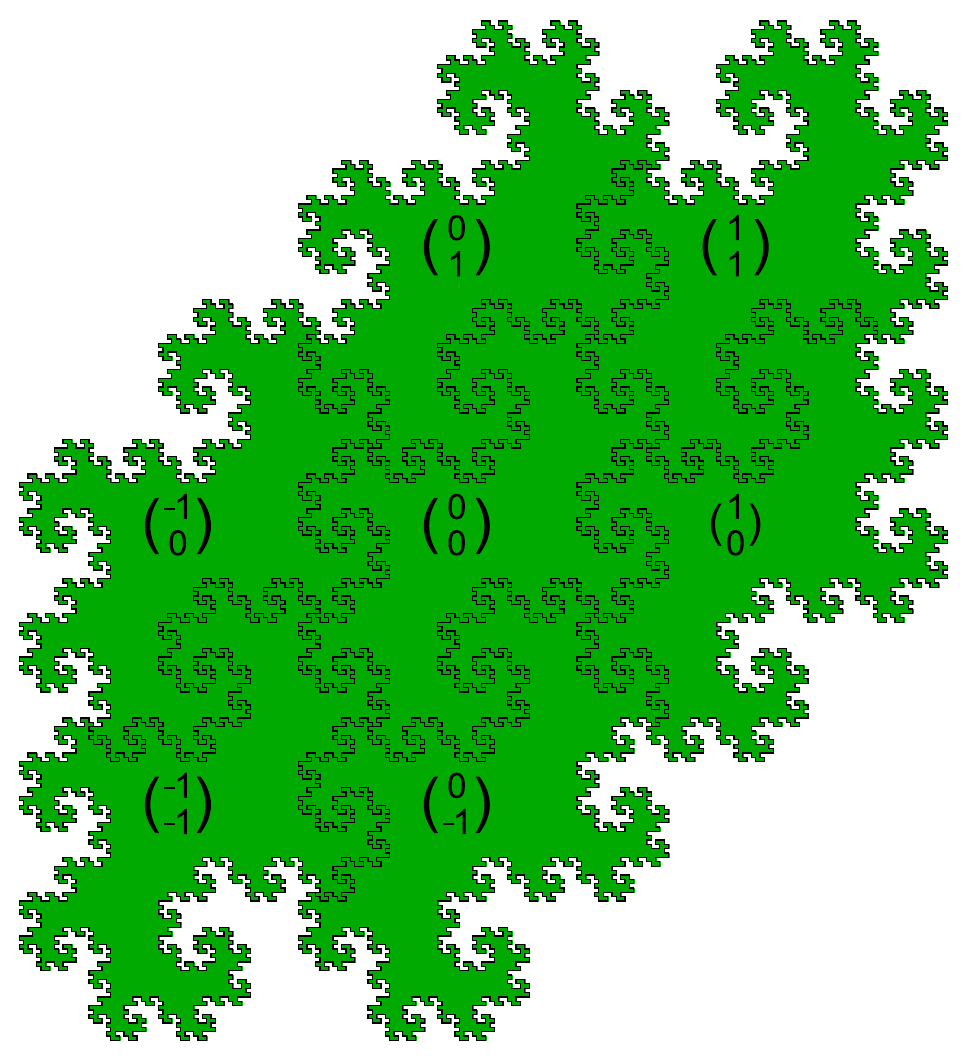}\hskip 1cm    
\includegraphics[height=6cm]{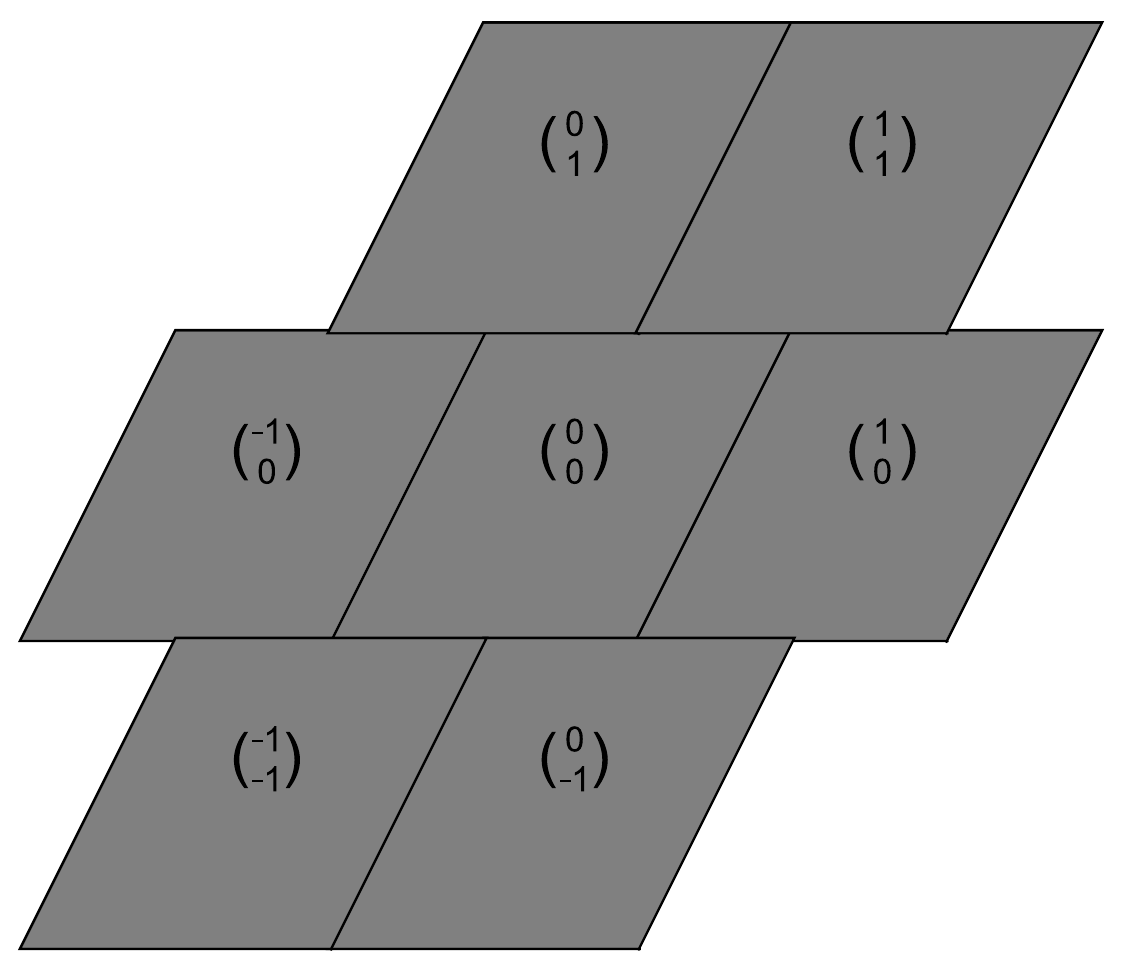}\hskip 1cm    
\caption{Illustration of Knuth's twin-dragon $K$ with its neighbors (left) and its model $M$ with its neighbors (right). The vectors inside a tile indicate its translation. \label{fig:twinNeighbors}}
\end{figure}
Since the model $M$ is just a parallelogram the sets $\complex{M}^{i}$ can be easily read off from the right hand side of Figure~\ref{fig:twinNeighbors}. In particular, it turns out that $\complex{K}^{i}=\complex{M}^{i}$ holds for each $i\in\mathbb{N}$ and, hence, $\complex{K}=\complex{M}$.

In \cite{Akiyama-Thuswaldner:05}  topological properties of $K$ are established. For instance, from \cite[Lemma~6.5]{Akiyama-Thuswaldner:05} we gain that each 2-fold intersection, {\it i.e.}, each set of the form $\cell{S}_K$ with $S\in \complex{K}^{2}$, is an arc. The fact that each 3-fold intersection, {\it i.e.}, each set of the form $\cell{S}_K$ with $S\in \complex{K}^{3}$ is a single point is contained in  \cite[Lemma~6.5]{Akiyama-Thuswaldner:05}\footnote{Observe that in \cite{Akiyama-Thuswaldner:05} Knuth's twin-dragon is defined by the matrix $\begin{pmatrix}0&-2\\1&-2\end{pmatrix}$ instead of $\begin{pmatrix}-1&-1\\1&-1\end{pmatrix}$ used in the present paper. It is easy to verify that this yields an attractor that is the same as $K$ apart from an affine transformation which also relates the corresponding tilings.}.
\end{example}

\subsection{Subdivision}\label{sec:subdivision}
Let $T=T(A,\digits)$ be a self-affine $\zn$-tile and $(M,F)$ a model for $T$ satisfying $\complex{T}=\complex{\start}$. We now describe a generalization of the set equations \eqref{setequation} and \eqref{modeldef} to the cells $\cell{S}$ and $\cell{S}_M$ for $S\in\complex{T}$, respectively. For example, if we consider $S=\{s_1, s_2\}$ then a set equation for the cell $\cell{S}$ can be derived from  \eqref{setequation} by
\begin{align*}
\cell{S} &=  (T+s_1) \cap (T+s_2) \\
&= \bigcup_{d_1\in\digits} A^{-1}(T + d_1+ As_1) \cap  \bigcup_{d_2\in\digits} A^{-1}(T + d_2 + As_2)\\
&= \bigcup_{d_1\in\digits}\bigcup_{d_2\in\digits}A^{-1}( (T + d_1 + As_1) \cap (T + d_2 + As_2) ).
\end{align*}
The unions in the last line run over $\digits$ for each element of $S=\{s_1,s_2\}$, {\it i.e.}, they run through the pairs $\digits^S=\digits^{\{s_1,s_2\}}$. The set $\digits^S$ can also be regarded as the set of functions $p: S\to \digits$. Using this interpretation we may write
\begin{align*}
\cell{S} &= \bigcup_{p\in \digits^{\{s_1,s_2\}}}A^{-1}( (T + p(s_1)+ As_1) \cap (T + p(s_2) + As_2) )\\
&=\bigcup_{p\in \digits^{\{s_1,s_2\}}}A^{-1}( (T + (p+A)(s_1)) \cap (T + (p+A)(s_2) ) \\
&= \bigcup_{p\in \digits^{\{s_1,s_2\}}}A^{-1} \cell{ \{ (p+A)(s_1),(p+A)(s_2)  \} } \\
&= \bigcup_{p\in \digits^{S}}A^{-1} \cell{  (p+A)(S)  }.
\end{align*}
As we may confine ourselves to extend the union over all intersections that are nonempty, {\it i.e.}, over all $p$ with $(p+A)(S)\in \complex{T}$, this motivates the following definition.

\begin{definition}[Subdivision operator]\label{def:subdivision} Let $T=T(A,\mathcal{D})$ be a self-affine $\zn$-tile. Then the {\em subdivision operator} $P$ is given by 
\[
P(S) =\{(p+A)(S) \in \complex{T} \mid p \in \digits^S\}\qquad (S\in\complex{T}),
\]
where $\digits^S$ denotes the set of functions from $S$ to $\digits$.
\end{definition}

It is unnecessary to define a similar operator for the model $\start$ since $\complex{T}=\complex{\start}$ and $\Bee(z)=Az$ holds for each $z \in \zn$ which implies that $P(S) =\{(p+F)(S)\in\complex{\start} \mid p \in \digits^S\}$. 

As $P(\{0\})=\{\{d\}\mid d \in \digits\}$, using the operator $P$ the set equations
in \eqref{setequation} and \eqref{modeldef}
become 
\[
A\cell{\{0\}}=\cell{P(\{0\})} \qquad\text{and}\qquad \Bee\cellstar{\{0\}}=\cellstar{P(\{0\})},
\]
respectively.
The terminology {\it subdivision operator} is further sustained by the set equations which we shall prove in Theorems~\ref{th:genset} and~\ref{th:genset2}. 

\subsection{The generalized set equation}\label{sec:set}
Let $T=T(A,\digits)$ be a self-affine $\mathbb{Z}^n$-tile. Using the subdivision operator $P$ from Definition~\ref{def:subdivision} we will now extend the standard notion of set equation \eqref{setequation} to intersections $\cell{S}$ (see also \cite{ST:03} where this is done by using so-called {\it boundary graphs}).   
To this end we need iterates of $P$ which we define inductively by $P^{(1)}(S)=P(S)$ and
\[
P^{(k)}(S) = \{ S'' \in P(S') \mid S' \in P^{(k-1)}(S) \} \qquad (k\ge 2). 
\]

\begin{theorem}[The generalized set equation]\label{th:genset}
Let $T$ be a self-affine $\mathbb{Z}^n$-tile and $S\in\complex{T}$. Then
\begin{equation}\label{sset}
\cell{S} =
A^{-k} \cell{P^{(k)}(S)} \qquad(k\in\mathbb{N}).
\end{equation}
\end{theorem}

As $(T,A)$ trivially is a model of $T$ having the same neighboring structure as $T$, Theorem~\ref{th:genset} is a special case of the first equality in Theorem~\ref{th:genset2}. Thus we refrain from proving Theorem~\ref{th:genset} here.

\begin{remark}\label{rem:gifs}
For $S\in\complex{T}$ with $0\in S$ equation \eqref{sset} yields $\cell{S} = \bigcup_{S' \in P(S)} A^{-1}\cell{S'}$. In each $S'$ occurring on the right hand side of this equation choose a fixed element $s' \in S'$. Then it becomes
\begin{equation}\label{sset:explicit}
\cell{S} = \bigcup_{S' \in P(S)} A^{-1}(\cell{S'-s'} + s') 
\end{equation}
and each shifted set $S'-s'$ is an element of $\complex{T}$ containing $0$. To each $S\in \complex{T}$ with $0\in S$ we associate an indeterminate $X_S$ whose range of values is the space of nonempty compact subsets of $\rn$. Using \eqref{sset:explicit}
we define the (finite) collection 
\begin{equation}\label{sset:explicit2}
X_S = \bigcup_{S' \in P(S)} A^{-1}(X_{S'-s'} + s')\qquad (S \in \complex{T},\, 0\in S)
\end{equation}
of set equations which defines a {\it graph directed iterated function system} whose unique solution is given by $X_S=\cell{S}$ for each $S\in \complex{T}$ with $0\in S$. Note that $|S|=|S'|$ holds for each $S'\in P(S)$. Thus, following \cite{ST:03}, for each $i\ge 1$ we define the graph $\Gamma_i$ as follows. The set of nodes of $\Gamma_i$ is given by $\complex{T}^i_0$. Moreover, there is an edge from $S$ to $S''=S'-s'$ labelled by $s'$ if and only if $A^{-1}(X_{S''}+s')$ occurs on the right hand side of \eqref{sset:explicit2}. Using these graphs, 
\eqref{sset:explicit2} becomes
\[
X_S = \bigcup_{S\xrightarrow{s'} S''\in \Gamma_i} A^{-1}(X_{S''} + s')\qquad (S \in \Gamma_i,\, i\in \mathbb{N}).
\]
The algorithmic construction of the graphs $\Gamma_i$ is detailed in \cite{ST:03} (see also \cite[Chapter~3]{Falconer:97} for basic definitions and results on graph directed iterated function systems). The following Example~\ref{ex:kngamma} as well as Figures~\ref{double-graph} and~\ref{triple-graph} in Section~\ref{sec:tame} contain examples of the graphs $\Gamma_2$ and $\Gamma_3$. The graph $\Gamma_1$ is obviously always given by the single node $\{0\}$ with $|\digits|$ self-loops each labeled by a digit $d\in\digits$.
\end{remark}

\begin{example}\label{ex:kngamma}
For Knuth's twin-dragon $K$ the graphs $\Gamma_i$ are known and easy to construct (see {\it e.g.}~\cite{Akiyama-Thuswaldner:05}). \begin{figure}[ht]
\hskip 0.5cm
\xymatrix{
*[o][F-]{\st{10} } \ar[rd]_{0}&&*[o][F-]{\st{\=1\=1} }\ar[ll]_{1}\\
&*[o][F-]{\st{01} }\ar[ur]_{0,1} \ar@/^2ex/[d]^{0}&\\
&*[o][F-]{\st{0\=1} }\ar[ld]_{0,1} \ar@/^2ex/[u]^{1}&\\
*[o][F-]{\st{11} }\ar[rr]_{0}&&*[o][F-]{\st{\=10} }\ar[lu]_{1}
}
\hskip 3cm
\xymatrix{
*[o][F-]{\state{10}{11} } \ar[d]^{0} &*[o][F-]{\state{\=10}{\=1\=1} }  \ar[d]_{1} \\
*[o][F-]{\state{\=10}{01} }  \ar[d]{1} &*[o][F-]{\state{10}{0\=1} }  \ar[d]_{0} \\
*[o][F-]{\state{0\=1}{\=1\=1} }\ar@/^4ex/[uu]^{1} &*[o][F-]{\state{01}{11} } \ar@/_4ex/[uu]_{0} 
}
\caption{The directed graphs $\Gamma_2$ (left) and $\Gamma_3$ (right) for Knuth's twin-dragon $K$. In $\Gamma_2$ the pair $ab$ stands for the node $\{(0,0)^t,(a,b)^t\}$ and $\bar a=-a$. Thus $ab$ corresponds to the nonempty 2-fold intersection $K \cap (K+ (a,b)^t)$. In the graph $\Gamma_3$ a node $a_1b_1\atop a_2b_2$ corresponds to the intersection $K\cap (K+(a_1,b_1)^t) \cap (K+(a_2,b_2)^t)$.
\label{fig:twinGraphs}}
\end{figure}
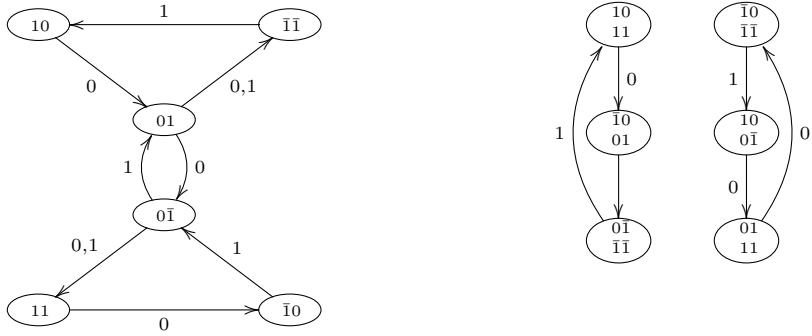
They are empty for $i\ge 4$. For $i=2,3$ they are depicted in Figure~\ref{fig:twinGraphs}. Note that these graphs (together with the trivial graph $\Gamma_1$) verify the formulas for $\complex{K}^{i}$ in Example~\ref{ex:twinNeighbor}. 
\end{example}

\subsection{Walks}
\label{walkz}
Let $\digits_k$ be as in \eqref{digitsk}. Iterating the set equation \eqref{setequation} for $T$ 
yields that
\[
T=\bigcup_{d_1\in \digits} A\inv(T+d_1) =\bigcup_{d_1\in \digits} A\inv \big(\bigcup_{d_2\in \digits} A\inv (T+d_2)+d_1\big)=\cdots.
\]
Note that each term is a subdivision of the previous one, {\it i.e.}, the tile $T$ can be subdivided finer and finer by the collections $\{A^{-k}(T+d) \mid d\in \digits_k\}$, $k\in\mathbb{N}$. As $\digits$ is a complete set of coset representatives of $\zn/A\zn$ each \emph{subtile} $A^{-k}(T+d_k)$, $d_k\in \digits_k$ lies in a unique \emph{ancestor} of the form $A^{-(k-1)}(T+d_{k-1})$, $d_{k-1}\in \digits_{k-1}$ if $k\ge 1$. Thus the set of collections $\{A^{-k}(T+d) \mid d\in \digits_k\}_{k\ge 0}$, and, equivalently, the set of collections $(\digits_k)_{k\ge 0}$ provides natural coordinates for points in $T$. We will call each element $(d_k)_{k\ge 0}\in(\digits_k)_{k\ge 0}$ a \emph{walk} in $T$. Thus each walk corresponds to a nested sequence of subtiles of $T$.

More generally, the set equation in \eqref{sset} gives coordinates for $\cell{S}$ with $S\in\complex{T}$ in a similar way. Indeed, it induces a sequence of subdivisions
\[
\cell{S} = A\inv\cell{P(S)} = A^{-2} \cell{P^{(2)}(S)}=A^{-3} \cell{P^{(3)}(S)}=\cdots.
\]
In analogy to the last paragraph we say that a \emph{walk} in $\cell{S}$ is a sequence $(S_k)$ where  $S_k \in P^{(k)}(S)$ so that $S_{k+1}$ is in $P(S_k)$ or rather $S_k \in P\inv(S_{k+1})$. According to \eqref{sset} this walk yields the nested sequence $(A^{-k}\cell{S_k})_{k\ge 0}$ whose intersection contains a single element of $\cell{S}$.  
As $\digits$ is a complete set of cosets of $\zn/A\zn$, one can check that $S \not= S'$ implies $P(S) \cap P(S') = \emptyset$.  For this reason there is again an \emph{ancestor function} $R$ from $P(S)$ to the set of finite subsets of ${\zn}$ such that $P = R\inv$, {\it i.e.}, $R(S')=S$ if $S'\in P(S)$. This motivates the following precise definitions of walks and their limit points (see Example~\ref{ex:walk} for an illustration).

\begin{definition}[Walks]\label{df:walk}
For $S\in \complex{T}$, the set $W(S)$ of \emph{walks} in $S$ is the inverse limit of the sequence ($P^{(k)}(S))_{k\in \mathbb{N}}$ with bonding map $R:P^{(k)}(S)\to P^{(k-1)}(S)$, \emph{i.e.},
\[
W(S)=\varprojlim_{k\in\mathbb{N}} P^{(k)}(S)  = \left\{ 
w=(a_k)\in\prod_{k\in\mathbb{N}} P^{(k)}(S) \; \mid \; R(a_k)=a_{k-1} \hbox{ for } k\in\mathbb{N}
\right\}.
\]
 Let $\pi_k:W(S)\to P^{(k)}(S)$ be the canonical projection and define $$\stage{w}{k}=A^{-k}\cell{\pi_k(w)}$$ for a walk $w$ and 
$$
\stage{C}{k}=A^{-k}\cell{\pi_k(C)}
$$
for a set of walks $C$.  
\end{definition}

Walks in $W(S)$ are \emph{coordinates} for a nested collection of sets. They can be regarded as \emph{codings} of points in $\cell{S}$. The set $W(S)$ can also be described as the set of infinite walks in a finite graph (see {\em e.g.} \cite{ST:03}); this motivates the terminology \emph{walk}. 

Recall that the image of $P^{(k)}$ is contained in $\complex{T}$ for each $k$ by definition. Thus for a walk $w$, the sequence $(\stage{w}{k})$ is a nested sequence of {\em nonempty} compact sets. The diameter of $\stage{w}{k}$ approaches zero (since $A \inv$ is a contraction) and consequently defines a unique limit point. 

\begin{definition}[Limit points of walks]
Let $S\in\complex{T}$ and $w\in W(S)$ be a walk. The mapping 
$$
w \mapsto \repmap{w}=\bigcap_{k\ge 1} \stage{w}{k} 
$$ 
is a continuous map of the Cantor set $W(S)$ to $\cell{S}$.
If $C$ is a set of walks, we will use the notation 
$$
\repmap{C}=\bigcap_{k\ge 1} C_k.
$$ 
\end{definition}

In particular, set $W=W(\{0\})$ and $\partial W = \bigcup_{s\in\mathbb{Z}^n\setminus\{0\}}W(\{0,s\})$. Theorem~\ref{th:genset} implies
\begin{equation}\label{decompcell}
T =\repmap{W}, \quad
\partial T = \repmap{\partial W}, \quad\hbox{and} \quad
\cell{S} = \repmap{W(S)}.
\end{equation}

We illustrate walks and their limit points by the following simple example.

\begin{example}\label{ex:walk}
Let $A=(2)$ and $\mathcal{D}=\{0,1\}$, hence, $T(A,\mathcal{D}) = [0,1]$. Consider the walk
\[
w=(a_k)=(1\cdot 2^0, 0\cdot 2^0 +1\cdot 2^1,  1\cdot 2^0 +0\cdot 2^1+1\cdot 2^2, 0\cdot 2^0 +1\cdot 2^1+0\cdot 2^2 + 1\cdot 2^3,\ldots)
\]
in $W=W(\{0\})$. This is indeed a walk in $W$ because $P^{k}(\{0\})=\left\{\sum_{i=0}^{k-1} d_i2^i \mid d_i\in \{0,1\}\right\}$ and $R(\sum_{i=0}^{k-1} d_i2^i)=\sum_{i=0}^{k-2} d_{i+1}2^i$ ($P$ adds and $R=P^{-1}$ forgets the first summand), so $R(a_k)=a_{k-1}$ holds as required by the definition of the inverse limit in Definition~\ref{df:walk}. For instance, we have $\pi_3(w)=1\cdot 2^0 +0\cdot 2^1+1\cdot 2^2 =5$ and, hence,
\[
w_3 = A^{-3}\cell{\pi_3(w)} = 2^{-3}(1\cdot 2^0 +0\cdot 2^1+1\cdot 2^2 +[0,1]) = 1\cdot 2^{-3} +0\cdot 2^{-2}+1\cdot 2^{-1} +2^{-3}[0,1].
\]
Thus, $w_3=0.101 + 2^{-3}[0,1]$ (written in binary expansion) is the subtile of the third subdivision of $T=[0,1]$ containing the element with binary expansion $0.10101010\cdots$. The walk $w$ thus can be regarded as nested sequence of $k$-th subdivisions of $T$ converging to the limit point $\repmap{w}=0.10101010\cdots$. So $w$ encodes the point $0.10101010\cdots\in T$.
\end{example}

\subsection{The generalized set equation and walks in a model}\label{sec:setm}
There are many similarities between a self-affine $\mathbb{Z}^n$-tile and its model, particularly if they have the same neighbor structure. We obtain the following analog of Theorem~\ref{th:genset} for models. Recall that $\start_{k}$ is defined in \eqref{effdef}.

\begin{theorem}\label{th:genset2}\label{eq:starset}
Let $(\start,\Bee)$ be a model for the self-affine $\mathbb{Z}^n$-tile $T$ satisfying
$\complex{T}=\complex{\start}$. 
Then for every $S\in \complex{T}$ and every $k\in\mathbb{N}$,
\begin{equation}\label{xstar3}
\Bee^{(k)}\cellstar{S} = \cellstar{P^{(k)}(S)}=A^{k}\cellstark.
\end{equation}
\end{theorem}

\begin{proof}
As $A^{-1}\Bee$ is $\zn$-equivariant and fixes $\zn$ pointwise, we have $\Bee(x+z)=\Bee(x) + Az$ for each $x\in\rn$, $z\in\zn$. 
Together with the set equation \eqref{modeldef} for $\start$ this yields 
\begin{align*}
\Bee \cell{S}_\start&=\bigcap_{s\in S}\Bee(\start+s) = \bigcap_{s\in S}(\Bee(\start)+As)  \\
&=\bigcap_{s\in S} (\start+\digits+As) =\bigcap_{s\in S}\bigcup_{d\in\digits}(\start+d+As)\\
&=\bigcup_{p\in \digits^S}\bigcap_{s\in S}(\start + p(s)+As)=\bigcup_{p\in \digits^S}\cell{(p+A)(S)}_M.
\end{align*}
Since $\complex{T}=\complex{\start}$, this subdivision  is again governed by the function $P$ and we arrive at 
\[
\Bee \cell{S}_\start = \bigcup_{S' \in P(S)} \cell{S'}_M.
\] 
Iterating this for $k$ times proves the first equality. To prove the second one recall that ${{\eff}_k}= A^{-k}{\Bee}^{(k)}$ and $\start_k=\eff_k \start$.   By Lemma~\ref{efflemma} we have
\begin{equation}\label{qkS}
\eff_k\cell{S}_\start = \bigcap_{s\in S}\eff_k(\start+s) = \bigcap_{s\in S}(\eff_k(\start)+s) = \bigcap_{s\in S}(\start_k+s) = \cell{S}_{\start_k}
\end{equation}
and, hence, $\Bee^{(k)}\cellstar{S} =A^{k}\cellstark$.
\end{proof}

Let $(\start,\Bee)$ be a model of a self-affine $\zn$-tile $T$ satisfying $\complex{\start}=\complex{T}$. For $S\in\complex{\start}$ to a walk $w\in W(S)$ we associate the nested collection
\[
\stagestar{w}{k}= (\Bee\inv)^{(k)}\cellstar{\pi_k(w)} 
\] 
and, for a set $C\subset W(S)$, we define
$$
\stagestar{C}{k}=(\Bee\inv)^{(k)}\cellstar{\pi_k(C)}.
$$
Moreover, we set
$$
\limptstar{w}=\bigcap_{k\ge 1} \stagestar{w}{k}.
$$ 
Note that $\limptstar{w}$ may contain {\em more than one point} as $\Bee^{-1}$ is not necessarily a contraction. This is an important difference between a self-affine $\mathbb{Z}^n$-tile and its model. However, as $\complex{\start}=\complex{T}$, the definition of a walk assures that $\limptstar{w}$ cannot be empty. Again, this definition extends to sets $C$ of walks by setting $\repmapf{C}=\bigcap_{k\ge 1} \stagestar{C}{k}$.
Using this notation we obtain from Theorem~\ref{th:genset2} that
\begin{equation}\label{stardecomp}
\start=\repmapf{W},\quad
\partial \start = \repmapf{\partial W},\quad\hbox{and}\quad
\cellstar{S} =\repmapf{W(S)}.
\end{equation}

\section{Monotone models for self-affine $\mathbb{Z}^n$-tiles}\label{sec:prop}

In this section we investigate mapping properties of the canonical quotient map $\canonical$ under the condition that $(\start,\Bee)$ is a so-called \emph{monotone model} for a self-affine $\zn$-tile $T$ (see Definition~\ref{def:model2}). 

\subsection{Monotone models and canonical quotients of cells}\label{sec:quotientcells} The following proposition gives results on images of certain sets under $\canonical$. Recall that the model $M_k$ is defined in \eqref{effdef}.

\begin{proposition}\label{imageLemma}
If $(\start,\Bee)$ is a model for the self-affine $\zn$-tile $T$ with $\complex{M}=\complex{T}$ then the following assertions hold.
\begin{enumerate}
\item[(i)] The sequence $\cellstark$ converges to $\cell{S}$ in the Hausdorff metric $\mathbf{D}$ for each  $S\in\complex{T}$.
\item[(ii)] $ \canonical{\cellstar{S}}=\cell{S}$ holds for each $S\in\complex{T}$.
\item[(iii)] If $S\in\complex{T}$ then for each $w\in W(S)$ we have $ \canonical(\limptstar{w}) =  \limpt{w}$.
\end{enumerate}
\end{proposition}

\begin{proof}
To prove (i) let $\varepsilon >0$ be arbitrary. Since $A$ is expanding and both $T$ and $\start$ are compact we may choose $k_0 \in \mathbb{N}$ in a way that $A^{-k}T$ and $A^{-k}\start$ are contained in a ball of diameter $\varepsilon$ around the origin for each $k\ge k_0$. Using Theorems~\ref{th:genset} and~\ref{th:genset2} we get
\begin{align*}
\mathbf{D}(\cell{S},\cell{S}_{M_k})& =
\mathbf{D}\Big(\bigcup_{S'\in P^{(k)}(S)} A^{-k}\cell{S'} ,      \bigcup_{S'\in P^{(k)}(S)} A^{-k}\cell{S'}_\start \Big)\\
&\le    \max\{ \mathbf{D}( A^{-k}\cell{S'} , A^{-k}\cell{S'}_\start)  \mid  S'\in P^{(k)}(S)\}.
\end{align*}
Note that $A^{-k}\cell{S'} \subset A^{-k}(T+s)$ and $A^{-k}\cell{S'}_\start \subset A^{-k}(\start+s)$ for each $S'\in \complex{T}$ and each $s \in S'$. Thus $A^{-k}\cell{S'}$ as well as $A^{-k}\cell{S'}_\start$ is contained in a ball of diameter $\varepsilon$ around $A^{-k}s$ implying that $\mathbf{D}( A^{-k}\cell{S'} , A^{-k}\cell{S'}_\start) < \varepsilon$ for each $S'\in \complex{T}$. Thus $\mathbf{D}(\cell{S},\cell{S}_{M_k})<\varepsilon$ and (i) is proved.

By \eqref{qkS} we have $\eff_k\cell{S}_\start = \cell{S}_{\start_k}$. Thus, as $\canonical=\lim_{k\to\infty}\eff_k$, assertion (ii) follows from (i).

To prove (iii), using Lemma~\ref{hproplem} and (ii) 
we derive
\[
 \canonical(\stagestar{w}{k})  = \canonical((\Bee\inv)^{(k)} \cellstar{\pi_k(w)}) 
= A^{-k} \canonical(\cellstar{\pi_k(w)}) 
= A^{-k}  \cell{\pi_k(w)} = \stage{w}{k}.
\]
Now (iii) follows from
\[
\canonical(\limptstar{w})
= \canonical\left(  \lim_{k\to\infty}\stagestar{w}{k} \right)
= \lim_{k\to\infty}  \canonical(\stagestar{w}{k}) 
= \lim_{k\to\infty}   w_k  
= \limpt{w}.\qedhere
\]
\end{proof}

We will now use the sets $\limptstar{w}$ to study preimages of $\canonical$.

\begin{lemma}\label{nonemptyintersection}
If $(\start,F)$ is a model for $T$ with $\complex{M}=\complex{T}$, for any 
nonempty set of walks $C$
\[
\bigcap_{w\in C} \limptstar{w} =\emptyset \quad\Longleftrightarrow\quad \bigcap_{w\in C} \limpt{w} =\emptyset \quad\Longleftrightarrow\quad |\limpt{C}|>1.
\]
\end{lemma}

\begin{proof}
As the sets $\limptstar{w}$ and $\limpt{w}$, $w\in C$, are compact it suffices to check the first equivalence for each finite subset $C'$ of $C$ by the finite intersection property for compact sets. Assume that  $\bigcap_{w\in C'}\limpt{w} \neq\emptyset$. As $\big( \bigcap_{w\in C'}\stage{w}{k}\big)_{k\in\mathbb{N}}$ is a nested sequence of compact sets this is equivalent to 
$\bigcap_{w\in C'} \stage{w}{k} \neq\emptyset$ for each $k\in\mathbb{N}$.
Because $\complex{\start}=\complex{T}$, this is in turn equivalent to
$\bigcap_{w\in C'} \stagestar{w}{k} \neq \emptyset$ 
for each $k\in\mathbb{N}$. As the sequence $\big( \bigcap_{w\in C'}\stagestar{w}{k}\big)_{k\in\mathbb{N}}$ is a nested sequence of compact sets, this is finally equivalent to $\bigcap_{w\in C'} \limptstar{w} \neq\emptyset$, proving the first equivalence. As for the second one note that the limit point $\limpt{w}$ is a singleton for each $w\in \limpt{C}$. Thus $\limpt{C}$ contains more than one element if and only if there exist $w,w'\in C$ having disjoint limit points which is equivalent to $\bigcap_{w\in C} \limpt{w} =\emptyset$. 
\end{proof}

We now refine the notion of model so that it preserves the neighbor structure of the modeled self-affine $\zn$-tile $T$ with connected intersections. Proposition~\ref{imageLemma} shows that for a model $\start$ having the same neighbor structure\footnote{Compare this with the definition of {\em respecting intersections} given in~\cite[Definition~1.2.5]{Rudnik:92}. } as $T$, the canonical quotient map $\canonical$ preserves intersections: $\canonical \cellstar{S}=\cell{S}$ for each $S\in\complex{T}$.  The following notion of {\em monotone model} enables us to say more about the point preimages of the canonical quotient map $\canonical$. 

\begin{definition}[Monotone model]\label{def:model2}
A model $(\start,F)$ is a {\em monotone model} for the self-affine $\zn$-tile $T$ if $\complex{T}=\complex{\start}$ and $\cellstar{S}$ is connected for each $S\subset \zn$. 
\end{definition}

For Knuth's twin-dragon $K$, we know from Example~\ref{ex:twinNeighbor} that the model $M$ given there satisfies $\complex{K}=\complex{\start}$. As all the sets $\cell{S}_\start$ are connected, $M$ is a monotone model of $K$.

Recall that a quotient is \emph{monotone} if point preimages of the quotient map are connected. 

\begin{theorem}\label{upperthm2}\label{all:connectedness}
Suppose that $(\start,F)$ is a monotone model for the self-affine $\zn$-tile $T$. 
\begin{itemize}
\item[(i)] $Q\cellstar{S} = \cell{S}$ is a monotone canonical quotient of ${\cellstar{S}}$ for each $S \in\complex{T}$. 
\item[(ii)] $Q\cellstar{\delta S} = \cell{\delta S}$ is a monotone canonical quotient of ${\cellstar{\delta S}}$ for each $S\in\complex{T}$. 
\item[(iii)] $Q\partial \start = \partial T$ is a monotone canonical quotient of $\partial \start$.
\end{itemize}
\end{theorem}

\begin{proof}
By Proposition~\ref{imageLemma}~(ii) we have $Q|_{\cellstar{S}}\cellstar{S}= \canonical\cellstar{S}=\cell{S}$. To show (i) we thus have to prove that $(\canonical|_{\cellstar{S}})\inv(x)$ is connected for each $x\in \cell{S}$. Let $x\in\cell{S}$ be fixed and choose $w\in W(S)$ with $x=\limpt{w}$. Proposition~\ref{imageLemma}~(iii) implies that
\begin{equation*}
(\canonical|_{\cellstar{S}})\inv(x)=\bigcup_{\substack{w' \in W(S)\\\limpt{w'} = \limpt{w} } } \limptstar{w'}.
\end{equation*}
By Lemma~\ref{nonemptyintersection} all the sets $\limptstar{w'}$ in the union on the right share a common point. It therefore remains to prove that $\limptstar{w'}$ is connected for each $w'\in W(S)$. However, since $\start$ is a monotone model for $T$, the cell $\cellstar{S'}$ is connected for each $S'\subset\mathbb{Z}^n$. Thus the set $\limptstar{w'}$ is a nested intersection of the connected sets $\stagestar{w'}{k}$ and therefore itself connected.

To prove (ii) first observe that 
\[
Q|_{\cellstar{\delta S}}\cellstar{\delta S}=
\canonical \cellstar{\delta S} = \bigcup_{S'\in\delta S }\canonical\cellstar{S'} = 
\bigcup_{s\in\mathbb{Z}^n\setminus S}\canonical\cellstar{S\cup\{s\}} =
\bigcup_{s\in\mathbb{Z}^n\setminus S}\cell{S\cup\{s\}} = \cell{\delta S}.
\]
Choose $x \in \cell{\delta S}$. Then, by Proposition~\ref{imageLemma}~(iii) we gain by choosing $w\in W(S)$ with $x=w_T$
\[
(\canonical|_{\cellstar{\delta S}})\inv(x)=
\bigcup_{s \in \zn\setminus S}(\canonical|_{\cellstar{S \cup \{s\} }})\inv(x)=
\bigcup_{s \in \zn\setminus S}\;\;\bigcup_{\substack{w' \in W(S \cup \{ s \})\\\limpt{w'} = \limpt{w} } } \limptstar{w'}
\]
and everything runs exactly as in (i). 

To show (iii) note that by \eqref{deltapartial} and \eqref{deltapartial2} we have $\partial T=\cell{\delta\{0\}}$ and $\partial M=\cellstar{\delta\{0\}}$. Thus (iii) is an immediate consequence of (ii).
\end{proof}

\subsection{Canonical quotients of boundaries} We saw in Theorem~\ref{upperthm} that $Q\partial \start = \partial T$ holds for a model $M$ of the self-affine $\zn$-tile $T$. Under certain conditions $\canonical$ behaves nicely also for boundaries of cells $\cell{S}$ with $|S|\ge 2$. To make this precise let $M$ be a $\zn$-tile and
set 
\begin{equation*}
\complexzero{M}^{i} := \{ S \in \complex{M}^i \mid 0\in S \}
\end{equation*} 
for the set of all $i$-cells of $\complex{M}$ containing $0$. Assume that the (dual) geometric realization $G_i(M):=\cell{\complexzero{M}^{i}}_M$ of $\complexzero{M}^{i}$ carries the subspace topology inherited from $\rn$ and, for $i\ge 2$, denote by $\partial_{i}=\partial_{i,M}$ (we omit the index $M$ as it is clear from the context) the boundary relative to $G_i(M)$, \emph{i.e.}, for $X\subset G_i(M)$ we have
\[
\partial_{i,M}X = \overline{X} \cap \overline{G_i(M)\setminus X}.
\]

\begin{example}\label{ex:relBd}
For Knuth's twin-dragon $K$, the set $G_2(K)$ is the union of all nonempty sets $\cell{\{(0,0)^t,s\}}_K$, where $s$ is one of the six nonzero translates depicted in  Figure~\ref{fig:twinNeighbors}. Thus $G_2(K)$ is equal to $\partial K$. We equip $\partial K$ with the subspace topology induced by the standard topology of $\mathbb{R}^2$. Then $\partial_{2}=\partial_{2,K}$ is the boundary operator w.r.t.\ this topology. For instance, the boundary of the wiggly line $\cell{\{(0,0)^t,(1,0)^t\}}_K$ which is an arc (see Example~\ref{ex:twinNeighbor}) is given by the two cells  $\cell{\{(0,0)^t,(1,0)^t,(1,1)^t\}}_K$ and $\cell{\{(0,0)^t,(1,0)^t,(0,-1)^t\}}_K$, which are single points ({\it cf.}~again Example~\ref{ex:twinNeighbor}). This also shows that in this instance we have $\partial_{2}\cell{\{(0,0)^t,(1,0)^t\}}_K=\cell{\delta \{(0,0)^t,(1,0)^t\}}_K$. 
\end{example}

The following definition is motivated by this last property, which will enable us to derive results on the behaviour of $Q$ with respect to these boundary operators.

\begin{definition}[Combinatorial tile]
Let $M$ be a $\mathbb{Z}^n$-tile. We say that $M$ is \emph{combinatorial} if $\partial_{i}\cell{S}_M=\cell{{\delta}S}_M$ for all $S \in \complexzero{M}^{i}$ and all $i \ge 2$.  
\end{definition}

The property of being combinatorial is a general position property. This is illustrated by the following easy example of a tile that is not combinatorial.

\begin{example}
Let $A=2I=\mathrm{diag}(2,2)$ be a $2\times 2$ diagonal matrix and $\mathcal{D}=\{(i,j)^t \,\mid\, 0\le i,j\le 1\}$. Then $T(A,\mathcal{D})=[0,1]^2$, the unit square, is not combinatorial. Indeed, the problem comes from the fact that 3-fold intersections are equal to 4-fold intersections. To make this precise choose $S=\{(0,0)^t,(0,1)^t,(1,0)^t\}$. Then $\cell{S}=\{(1,1)^t\}$ and thus $\partial_3 \cell{S} = \emptyset$. On the other hand $\cell{\delta S} = \cell{S \cup \{(1,1)^t\}} = \{(1,1)^t\}$. This implies that $\partial_3 \cell{S} \not=\cell{\delta S}$ for this choice of $S$. 
\end{example}
 
We can show the following result. 

\begin{theorem}\label{tilingcomplex}
Let $(\start,F)$ be a monotone model for a self-affine $\mathbb{Z}^n$-tile $T$, $i\ge 2$, and $S\in\complexzero{T}^i$.
\begin{enumerate}
\item[(i)] If $\start$ is combinatorial, then $\partial_{i} \cell{S} \subset Q \partial_{i} \cell{S}_M$.
\item[(ii)] If $T$ is combinatorial, then $Q \partial_{i} \cell{S}_M\subset \partial_{i} \cell{S} $.
\end{enumerate}
Thus, if both, $\start$ and $T$ are combinatorial we have that $Q \partial_{i} \cell{S}_M  = \partial_{i}Q \cell{S}_M$.
\end{theorem}

\begin{proof}
Let $S\in\complexzero{T}^i$, $i\ge 2$, be arbitrary but fixed.
To prove (i) let $x\in \partial_{i}\cell{S}$ be given. Then, since $\complexzero{T}^{i}$ is a finite complex and cells are closed in $\cell{\complexzero{T}^i}$, there exists $S'\in\complexzero{T}^i$, $S'\not=S$ such that $x\in \cell{S'}$. As $x\in \cell{S} \cap \cell{S'}$ there exists $s\in S'\setminus S$ such that $x\in \cell{S \cup \{s\}}$.  Proposition~\ref{imageLemma}~(ii) implies that 
$Q \cell{S \cup \{s\}}_M = \cell{S \cup \{s\}}$. Thus, because $M$ is combinatorial, we obtain that
$
x \in Q \cell{S \cup \{s\}}_M \subset Q \cell{\delta S}_M = Q \partial_{i} \cell{S }_M.
$

To prove (ii) let $x\in Q \partial_{i}\cell{S}_M$ be given.  By the same argument as in (i) there is $s\in \zn\setminus S$ with $x\in Q\cell{S\cup \{s\}}_M $. Thus, as $T$ is combinatorial,
$
x\in Q\cell{S\cup \{s\}}_M  = \cell{S\cup \{s\}}  \subset  \cell{\delta S} = \partial_{i}\cell{S}. \qedhere
$  
\end{proof}

It is easy to check that Knuth's twin-dragon $K$ as well as its model $M$ are combinatorial. Thus they fulfill the conditions of Theorem~\ref{tilingcomplex}.

\section{Self-affine $\mathbb{Z}^3$-tiles whose boundaries are manifolds}\label{sec:sphere}

This section contains our first main results. After a short treatment of the planar case (Section~\ref{sec:planar}) we continue with statements and proofs of topological preliminaries (Section~\ref{subsec:sphere}) that are used in Section~\ref{sec:23} to prove results on self-affine $\mathbb{Z}^3$-tiles whose boundaries are 2-manifolds.

\subsection{The planar case}\label{sec:planar}
We start this section with an easy criterion for a self-affine $\mathbb{Z}^2$-tile $T\subset \mathbb{R}^2$ to be homeomorphic to the closed disk $\mathbb{D}^2$.

\begin{theorem}\label{planar}
A self-affine $\mathbb{Z}^2$-tile $T$ admits a monotone model which is homeomorphic to $\mathbb{D}^2$ if and only if $T$ is homeomorphic to $\mathbb{D}^2$.
\end{theorem}

\begin{proof}
If $T$ admits a monotone model which is homeomorphic to $\mathbb{D}^2$, 
Theorem~\ref{all:connectedness}~(iii) implies that $\partial T$ is a monotone quotient of $\mathbb{S}^1$ and thus is either a singleton or homeomorphic to $\mathbb{S}^1$. As $T$ is the closure of its interior, $\partial T$ cannot be a singleton and the Jordan Curve Theorem implies that $T\cong\mathbb{D}^2$. For the converse just observe that $(T,A)$ is a monotone model for $T$ (it is not hard to check that $\cell{S}$ is always connected for a self-affine tile $T$ which is homeomorphic to $\mathbb{D}^2$, see {\it e.g.} \cite[proofs of Lemma~6.4 and~6.5]{Akiyama-Thuswaldner:05}).
\end{proof}

It is known that a $\mathbb{Z}^2$-tile $M$ which is a closed disk has either $6$ or $8$ \emph{neighbors} ({\it i.e.}, $\complex{M}^2_0$ has either $6$ or $8$ elements; see {\it e.g.} \cite[Lemma~5.1]{BG:94}). This makes it easy to find a suitable monotone model $\start$ to apply Theorem~\ref{planar} ({\it cf.}~Section~\ref{sec:aa}).  For a similar criterion see~\cite[Theorems~2.1 and~2.2]{BW:01}. 

\subsection{Topological preliminaries}\label{subsec:sphere}
We start with some notations and definitions. Let $A$ be a subset of some space. We denote by $H_i(A)=H_i(A;\mathbb{Z})$ the $i$-th homology group of $A$ with integer coefficients and by $H_i(A,B)=H_i(A,B;\mathbb{Z})$ with $B\subset A$ its relative versions.  For the reduced $i$-th homology group we write $\tilde H_i(A)$. As we will use Alexander Duality in the proofs of this section we will also exploit the $i$-th \v Cech cohomology group $\check H^i(A)$ of $A$ and its relative versions $\check H^i(A,B)$ for $B\subset A$. For more on (co)homology we refer the reader to text books like Hatcher~\cite{Hatcher:02} and Massey~\cite{Massey:78}.

We recall the following definitions.

\begin{definition}[Contractible]
A subset $Y$ of a space $X$ is called {\it contractible in $X$} if $Y$ can be deformed to a single point by a continuous homotopy in $X$. If $X$ is contractible in $X$, we just call $X$ \emph{contractible}.
\end{definition}

\begin{definition}[Separator and path separator]\label{def:ps}
A \emph{separator} of a connected space $X$ is a set $Y\subset X$ such that $X\setminus Y$ is not connected. A separator is called \emph{irreducible} if none of its proper subsets is a separator of $X$. A separator consisting of a single point is called {\it separating point}. (See \cite[\S46, VII]{Kuratowski:68}.)

A  {\it path separator} of a path connected space $X$ is a set $Y\subset X$ such that there are two points $u,v\in X\setminus Y$  with the property that there is no path\footnote{Recall that unlike an arc a path is allowed to have self-intersection.} $p\subset X\setminus Y$ leading from $u$ to $v$. In this case we say more precisely that $Y$ is a \emph{path separator between $u$ and $v$}.
\end{definition}

\begin{definition}[{Monotone upper semi-continuous decomposition, see {\em e.g.}~\cite[p.~7f]{Daverman:07}}]\label{def:usd}
A partition $G$ of a topological space $X$ is called a {\em decomposition} of $X$. $G\cong X/G$  inherits a topology from $X$ by the {\em decomposition map} $q: X\to G$ which sends each $x\in X$ to the unique element of $G$ containing $x$. Specifically, $G$ is equipped with the richest topology that makes $q$ continuous.

The decomposition $G$ is called {\it monotone}, if each of its elements is a connected subset of $X$.

The decomposition $G$ is said to be {\em upper semi-continuous} if each $g\in G$ is closed in $X$ and if, for each $g\in G$ and each open set $U\subset X$ containing $g$ there is an open set $V\subset X$ containing $g$ such that every $g'\in G$ with $g\cap V\not=\emptyset$ is contained in $U$.
\end{definition}

Recall that a map is \emph{closed} if it maps closed sets to closed sets. Continuous maps between compact Hausdorff spaces are always closed. We need the following criterion for upper semi-continuous decompositions.

\begin{lemma}[{{\em cf.\ e.g.}~\cite[I \S1, Proposition~1]{Daverman:07}}]\label{lem:upperCrit}
A decomposition $G$ of a space $X$ into closed subsets is upper semi-continuous if and only if the decomposition map $q:X\to G$ is closed.
\end{lemma}

An important tool in our proofs will be the following result on monotone upper semi-continuous decompositions of 2-manifolds that was proved by Roberts and Steenrod~\cite{Roberts-Steenrod:38} generalizing a theorem of Moore~\cite{Moore:25}. 

\begin{proposition}[{{\it cf.} \cite[Theorem 1 and its proof]{Roberts-Steenrod:38}}]\label{RSb}
Let $\mathcal{S}$ be a compact 2-manifold and let $G$ be a monotone upper semi-continuous decomposition of $\mathcal{S}$. If $G$ contains at least 2 elements and for each $g\in G$ 
\begin{itemize}
\item[(i)] $g$ is contained in an open disk and
\item[(ii)] $g$ is not a separator of $\mathcal{S}$,
\end{itemize}
then $G$ is homeomorphic to $\mathcal{S}$.
\end{proposition}

Before we relate these concepts to our setting we require the following notation.

\begin{definition}[Semi-contractible monotone model]\label{def:semicontr}
Let $\start$ be a monotone model for a self-affine $\mathbb{Z}^3$-tile with $\partial M$ being a closed surface and let $\canonical$ be the associated quotient map. We say that $M$ is {\em semi-contractible} if each point preimage of $\canonical|_{\partial M}$ is contained in a neighborhood which is contractible in $\partial M$ ({\em e.g.} a disk).
\end{definition}

The existence of such contractible neighborhoods is needed in order to exclude that the preimage of a point wraps around one of the handles of $\partial M$ in the case $\partial M$ is a surface of genus $g\ge 1$. A model whose boundary is homeomorphic to a sphere is always semi-contractible. 

\begin{corollary}\label{RS}
Let $\start$ be a semi-contractible monotone model for the self-affine $\mathbb{Z}^3$-tile $T$. Assume that $\partial\start$ is a 2-manifold $\mathcal{S}$. If for each $x\in \partial T$ the preimage $(\canonical|_{\partial \start})^{-1}(x)$ is not a separator of $\partial \start$, then $\partial T$ is homeomorphic to $\mathcal{S}$.
\end{corollary} 

\begin{proof}
Consider the quotient map $\canonical|_{\partial \start}: \partial M \to \partial T$ and let $G := \{ (\canonical|_{\partial \start})^{-1}(x) \mid x \in \partial T  \}$ be a decomposition of $\partial T$. Then $\partial T\cong \partial M/G$ and, by compactness, the decomposition map $\canonical|_{\partial \start}$ is closed. Thus Lemma~\ref{lem:upperCrit} implies that $\partial T$ is an upper semi-continuous decomposition of $\partial\start$. As Theorem~\ref{all:connectedness}~(iii) shows that $(\canonical|_{\partial \start})^{-1}(x)$ is connected for each $x\in \partial T$, this decomposition is monotone.  Moreover, since $T$ is the closure of its interior, $G$ certainly contains at least 2 elements. By the semi-contractibility of $M$, item (i) of Proposition~\ref{RSb} is satisfied,  and item (ii) is satisfied by assumption. Thus, Proposition~\ref{RSb} can be applied to $\partial T$ and the result is proved.
\end{proof}

The following easy lemma, which will be needed on several occasions, is the first in the following list of preparatory results.

\begin{lemma}\label{compcomp}
Let $M\subset\rn$ be compact and equal to the closure of its interior. Assume that $U$ is a nonempty bounded component of $\rn\setminus M$. Then $\operatorname{int}(M+a)$ is not contained in $U$ for each $a\in\rn$.
\end{lemma}

\begin{proof}
As $\operatorname{int}(M)\not=\emptyset$ the result is true for $a=0$, and we may assume that $a\not=0$. If the assertion was wrong we had
$
\partial(U + a) \subset M + a \subset \overline{U}
$
which is absurd for a bounded set $U$.
\end{proof}

\begin{lemma}\label{compconnected}
Let $M$ be a $\mathbb{Z}^n$-tile with $\operatorname{int}(M)$ connected. Then
$\mathbb{R}^n \setminus M$ is connected.
\end{lemma}

\begin{proof}
If this is wrong,  $M$ is a separator of $\mathbb{R}^n$ and we may choose a bounded complementary component $U$ of $M$. As $\operatorname{int}(M)$ is connected and $M$ tiles $\rn$ with $\zn$-translates, there exists $s\in\mathbb{Z}^n\setminus\{0\}$ such that $\operatorname{int}(M+s) \subset U$. This contradicts Lemma~\ref{compcomp}.
\end{proof}

\begin{lemma}[{see {\it e.g.}~\cite[\S 46, VII, Theorem~4]{Kuratowski:68}}]\label{lem:irredsep}
If $Y\subset \mathbb{R}^n$ is the common boundary of two components of $\rn\setminus Y$ then $Y$ is an irreducible separator of $\rn$.
\end{lemma}

The next two propositions are of interest in their own right.

\begin{proposition}\label{prop:nosep}
Let $M \subset \mathbb{R}^n$ be a $\mathbb{Z}^n$-tile. If $\operatorname{int}(M)$ is connected then $\partial M$ is an irreducible separator of $\rn$.
\end{proposition}

\begin{proof}
This is an immediate consequence of Lemmas~\ref{compconnected} and~\ref{lem:irredsep}.  
\end{proof}

\begin{remark}\label{rem:cpct}
Let $\mathbb{S}^n=\mathbb{R}^n\cup\{\infty\}$ be the one-point compactification of $\mathbb{R}^n$. As $M$ is compact, Lemma~\ref{compconnected} and Proposition~\ref{prop:nosep} remain true if $\mathbb{R}^n$ is replaced by $\mathbb{S}^n$ in their statement.
\end{remark}

Before we state the second proposition we recall some notions from algebraic topology. If $M$ is a locally compact Hausdorff space and $A,B\subset M$ are two closed sets such that $M=A\cup B$ then the \emph{Mayer-Vietoris Sequence} for \v Cech cohomology
groups
\begin{equation}\label{Mayer1} 
\cdots\rightarrow\check H^{n-2}(A\cap B)\rightarrow \check H^{n-1}(A \cup B)
\rightarrow \check H^{n-1}(A) \oplus \check H^{n-1}(B) \rightarrow\check H^{n-1}(A\cap B)\rightarrow \cdots
\end{equation}
is an exact sequence (see {\it e.g.}~\cite[Theorem~3.13]{Massey:78}). Moreover, we recall that Alexander Duality ({\it cf. e.g.} Dold~\cite[Chapter VIII, 8.15]{Dold:72}) states that if $A$ is a compact subset of $\mathbb{S}^n=\mathbb{R}^n\cup\{\infty\}$ and
$x\in A$ then
\begin{equation}\label{alexander}
\check H^{n-i}(A,\{x\}) = \tilde H_{i-1}(\mathbb{S}^n\setminus A)
\end{equation}
holds for $1\le i\le n$ (note that for $k > 0$ we have $\check H^{k}(A,\{x\})=\check H^{k}(A)$ and $\tilde H_{k}(\mathbb{S}^n\setminus A)= H_{k}(\mathbb{S}^n\setminus A)$).

\begin{proposition}\label{prop:nocut}
For $n\ge 2$ let $M \subset \mathbb{R}^n$ be a $\mathbb{Z}^n$-tile. If $\operatorname{int}(M)$ is connected then  $\partial M$ has no separating point.
\end{proposition}

\begin{proof}
For $n=2$ the result follows from~\cite[Theorem~1.1~(iii)]{LRT:02}. To prove the result for $n\ge 3$ assume on the contrary that there is a separating point $\{x\}$ of $\partial M$. Then there are nonempty compact proper subsets $A,B\subset\partial M$ satisfying $\partial M = A \cup B$ with $A\cap B = \{x\}$.
We will now use the Mayer-Vietoris-sequence for \v Cech cohomology
groups \eqref{Mayer1} in order to prove our result.
Since $A,B$ are closed subsets of $\mathbb{R}^n$ and $M=A\cup B$ is a locally compact Hausdorff space the
Mayer-Vietoris sequence \eqref{Mayer1} is exact for this choice. We now apply Alexander Duality \eqref{alexander} together with standard results from
singular homology theory to derive that (see Remark~\ref{rem:cpct})
\[
\begin{array}{rll}
\check H^{n-2}(A \cap B)=& H_1(\mathbb{S}^n\setminus(A\cap B))=0   &\hbox{(as $A\cap B = \{x\}$, a singleton)},\\
\check H^{n-1}(A \cup B)=& \tilde H_0(\mathbb{S}^n\setminus(A\cup B))=\mathbb{Z}   &\hbox{(as $\mathbb{S}^n\setminus(A\cup B)=\mathbb{S}^n\setminus\partial M$ has}\\
&&\hbox{\, 2 components by Lemma~\ref{compconnected})},\\
\check H^{n-1}(A)=& \tilde H_0(\mathbb{S}^n\setminus A)= 0  &\hbox{(as $\mathbb{S}^n\setminus A$ has one component by Proposition~\ref{prop:nosep})},\\
\check H^{n-1}(B)=& \tilde H_0(\mathbb{S}^n\setminus B)= 0  &\hbox{(as $\mathbb{S}^n\setminus B$ has one component by Proposition~\ref{prop:nosep})}.\\
\end{array}
\]
Inserting this in \eqref{Mayer1} yields that the sequence
$\cdots\rightarrow 0\rightarrow \mathbb{Z}\rightarrow0\rightarrow\cdots$ is exact, which is absurd. This yields the desired contradiction and
the result is proved.
\end{proof}

We now recall some notions and results from point set topology that will be needed in Section~\ref{sec:23}.

\begin{definition}[Cut]\label{def:Cut}
Let $X$ be a topological space. A set $Y\subset X$ is said to \emph{cut} $X$ if there are points $x_1,x_2 \in X\setminus Y$ that cannot be joined by a continuum in $X\setminus Y$, {\it i.e.}, there does not exist a continuum $C\subset X\setminus Y$ such that $x_1,x_2 \in C$ ({\it cf}.~\cite[\S47, VIII]{Kuratowski:68}). 
\end{definition}

The following result states that under particular assumptions separations and cuts are the same.

\begin{lemma}[{{\it cf.}~\cite[\S50, II, Theorem~8]{Kuratowski:68}}]\label{lem:cutcut}
Let $X$ be a connected and locally connected space. A closed set Y is a separator of $X$ if and only if it cuts $X$.
\end{lemma}

We also need the following results on joining points by arcs as well as  connectedness of preimages.

\begin{lemma}[{see \cite[\S52, II, Theorem~16]{Kuratowski:68}}]\label{lem:simpleclosed}
Let $X$ be a locally connected continuum which has no separating point. Then each pair of points $x_1,x_2\in X$ is contained in a simple closed curve $c\subset X$.
Thus $x_1$ and $x_2$ can be joined by an arc in $X\setminus \{x\}$ for each $x\in X\setminus \{x_1,x_2\}$.
\end{lemma}

\begin{lemma}[{see \cite[\S46, I, Theorem~9]{Kuratowski:68}}]\label{lem:preimageconnected}
If the point preimages of a closed mapping are connected, then the preimages of connected sets are connected as well.
\end{lemma}

\subsection{Self-affine tiles and manifolds}\label{sec:23}
We shall now prove results on boundaries of self-affine $\mathbb{Z}^3$-tiles that are homeomorphic to a closed surface. 
In the statement of the next result recall the definition of semi-contractibility given in Definition~\ref{def:semicontr}.

\begin{theorem}\label{upperthm5}\label{sphere2thm}\label{spherecor}
Let $T$ be a self-affine $\mathbb{Z}^3$-tile with connected interior which admits a semi-con\-tract\-ible monotone model $M$ whose boundary is the closed surface $\mathcal{S}$. Then the following assertions hold.
\begin{itemize}
\item[(i)] $\partial T$ is homeomorphic to $\mathcal{S}$.
\item[(ii)] Under the restriction $\canonical|_{\partial \start}$ 
a preimage of a point cannot be a path separator of $\partial \start$.
\end{itemize}
\end{theorem}

\begin{proof}
We start with the proof of assertion (ii). 
Assume on the contrary that there is $x\in \partial T$ such that $(\canonical|_{\partial \start})\inv(x)$ is a path separator of $\partial \start$ between two elements $u,v\in \partial \start$ in the sense of Definition~\ref{def:ps}. We first observe that being the continuous image of the locally connected continuum $\partial \start$ (by Theorem~\ref{hprop}~(ii)), the set $\partial T$ is a locally connected continuum. Moreover, as $\operatorname{int}(T)$ is connected, Proposition~\ref{prop:nocut} implies that $\partial T$ contains no separating point. Thus by Lemma~\ref{lem:simpleclosed} we may join $ \canonical(u)$ and $ \canonical(v)$ by an arc $\ell$ in $\partial T\setminus\{x\}$. Since $(\canonical|_{\partial \start})\inv(y)$ is connected for each $y\in \partial T$ by Theorem~\ref{all:connectedness}~(iii) and $\canonical|_{\partial \start}$ is a closed mapping, Lemma~\ref{lem:preimageconnected} implies that $(\canonical|_{\partial \start})\inv(\ell)$ is a continuum. By construction, $(\canonical|_{\partial \start})\inv(\ell)$ contains $u$ and $v$, and is disjoint from $(\canonical|_{\partial \start})\inv(x)$.  
As $(\canonical|_{\partial \start})\inv(\ell)$ and $(\canonical|_{\partial \start})\inv(x)$ are disjoint compact sets there is $\varepsilon > 0$ such that 
\begin{equation}\label{eq:proofstr}
\min\{||y-y'||\mid y\in(\canonical|_{\partial \start})\inv(\ell),\, y'\in(\canonical|_{\partial \start})\inv(x)\} > \varepsilon.
\end{equation}
Moreover, since $(\canonical|_{\partial \start})\inv(\ell)$ is a continuum, $u$ and $v$ are chain connected in $(\canonical|_{\partial \start})\inv(\ell)$, {\it i.e.}, for each $\eta > 0$ there exist $y_0, y_1,\ldots, y_m\in (\canonical|_{\partial \start})\inv(\ell)$ with $||y_i-y_{i-1}|| < \eta$ for each $i\in\{1,\ldots, m\}$ such that $y_0=u$ and $y_m=v$. As $\partial M$ is a locally connected continuum we may choose $\eta$ in a way that $y_{i-1}$ and $y_{i}$ can be connected by an arc $a_i$ that is entirely contained in an $\varepsilon$-ball around $y_i$ for each $i\in\{1,\ldots, m\}$ (see {\it e.g.}~\cite[\S50,~I,~Theorem~4]{Kuratowski:68}). Note that by \eqref{eq:proofstr}, each arc $a_i$ is disjoint from $(\canonical|_{\partial \start})\inv(x)$. Thus we can concatenate the arcs $a_1,\ldots,a_{m}$ to obtain a path $p$ connecting $u$ and $v$ in $\partial M$ which avoids $(\canonical|_{\partial \start})\inv(x)$. This contradicts our assumption and assertion~(ii) is proved.

As $\start$ is a semi-contractible monotone model for $T$ whose boundary $\partial T$ is a 2-manifold $\mathcal{S}$, we want to use Corollary~\ref{RS} to prove assertion~(i) of the theorem. To be able to do so, we have to show that for each $x\in \partial T$ the preimage $(\canonical|_{\partial \start})\inv(x)$ is not a separator of $\partial \start$. 
Suppose on the contrary that $(\canonical|_{\partial \start})\inv(x)$ is a separator of $\partial \start$. Then Lemma~\ref{lem:cutcut} implies that $(\canonical|_{\partial \start})\inv(x)$ cuts $\partial \start$, {\it i.e.}, 
there are $u,v\in\partial \start\setminus (\canonical|_{\partial \start})\inv(x)$ that cannot be joined by a continuum in $\partial \start\setminus (\canonical|_{\partial \start})\inv(x)$ and, {\it a fortiori}, cannot be joined by a path in $\partial \start\setminus (\canonical|_{\partial \start})\inv(x)$. Thus $(\canonical|_{\partial \start})\inv(x)$ is a path separator of $\partial M$ and we get a contradiction to assertion (ii) that was proved before. Thus assertion~(i) follows from Corollary~\ref{RS}. 
\end{proof}

We immediately obtain the following consequence (note that we do not have to assume semi-contractibility here in view of the paragraph after Definition~\ref{def:semicontr}). 

\begin{theorem} \label{uppercor}
Let $T$ be a self-affine $\mathbb{Z}^3$-tile with connected interior which admits a monotone model $M$  whose boundary is homeomorphic to the 2-sphere $\mathbb{S}^2$. Then $\partial T$ is homeomorphic to $\mathbb{S}^2$ and a point preimage of the quotient map $\canonical|_{\partial \start}$ cannot be a path separator of $\partial \start$.
\end{theorem}

There is a version of Theorem~\ref{upperthm5} for finite unions of $\mathbb{Z}^3$-translates of $T$. Before we make this precise, for a $\zn$-tile $M$ we set 
\begin{equation}\label{eq:unions}
[S]_M= \bigcup_{s\in S} (M + s) \qquad (S\subset M).
\end{equation}
Again, we write $[S]$ instead of $[S]_T$ if $T$ is a self-affine $\zn$-tile.

\begin{proposition}\label{prop:finiteunion}
Let $T$ be a self-affine $\mathbb{Z}^3$-tile which admits a semi-contractible monotone model $M$ and let $S \subset \mathbb{Z}^3$ be nonempty and finite. Assume that $\operatorname{int}([S])$ and $\mathbb{R}^3\setminus [S]$ are connected. If $\partial[S]_M\cong\mathcal{S}$ for a closed surface $\mathcal{S}$ then $\partial[S]_M\cong\mathcal{S}\cong\partial[S]$.
\end{proposition}

\begin{proof}
As $\partial[S]$ is the boundary of the components $\operatorname{int}([S])$ and $\mathbb{R}^3\setminus [S]$, it is an irreducible separator of $\mathbb{R}^3$ by Lemma~\ref{lem:irredsep}. Thus $[S]$ satisfies analogs of Lemma~\ref{compconnected} and Proposition~\ref{prop:nosep} and we may literally imitate the proof of Proposition~\ref{prop:nocut} to see that $\partial[S]$ has no separating point. The result now follows by the same proof as the one of Theorem~\ref{upperthm5} (just note that Theorem~\ref{all:connectedness} ~(iii) can be extended immediately to show that $\partial[S]$ is a monotone quotient of $\partial[S]_M$).
\end{proof}

The next theorem shows that under certain conditions $\partial T=\bigcup_{s\in\mathbb{Z}^3\setminus\{0\}}\langle\{0,s\}\rangle$ admits a natural CW-structure defined by the intersections $\cell{S}$. Recall that a set is \emph{degenerate} if it contains fewer than 2 points.

\begin{theorem}\label{th:complex}
Let $T$ be a self-affine $\mathbb{Z}^3$-tile with connected interior which admits a semi-contractible combinatorial monotone model $M$ whose boundary is the closed surface $\mathcal{S}$.

Let $S\in\complex{T}$ be nondegenerate. If $\cell{S}_M$ is a closed topological manifold or a ball then its canonical quotient $\cell{S}$ is either homeomorphic to $\cell{S}_\start$ or degenerate.  
\end{theorem}

\begin{proof}
Let $S\in\complex{T}$ be nondegenerate and assume, without loss of generality, that $0\in S$. This implies that $\cellstar{S}\subset\partial M$. Assume that $\cellstar{S}$ is a closed manifold or a ball inside of the surface $\partial\start \cong \mathcal{S}$. Thus $\cellstar{S}$ is either a closed surface, a closed disk, a circle, an arc, or a point. We have to show that $\cell{S}$ has the required properties. To this end we will use the fact that $\cell{S}=Q\cellstar{S}$ is a monotone quotient of $\cellstar{S}$ by Theorem~\ref{all:connectedness}~(i).

We first dispose of the easy cases. If $\cellstar{S}$ is a point, observing that the monotone quotient of a point is a point, we gain that $\cell{S}=Q\cellstar{S}$ is a point. Similarly, if $\cellstar{S}$ is an arc, the fact that the monotone quotient of an arc is either an arc or a point yields that $\cell{S}$ is an arc or a point. The case where $\cellstar{S}$ is a circle is also settled because it is well-known that the monotone quotient of a circle is a circle or a point. Finally, if $\cellstar{S}$ is a closed surface,  we have $\cellstar{S}=\mathcal{S}$ as a closed surface has no proper subsurface. Thus Theorem~\ref{upperthm5}~(i) yields that $\cell{S}\cong \mathcal{S}$.

It remains to consider the case where $\cellstar{S}$ is a closed disk. Since $M$ is combinatorial, this implies that $|S|=2$ and  $\partial_2 \cellstar{S} = \cellstar{\delta S}$. Thus, as $G_2(M)=\cell{\complexzero{M}^2}_\start=\partial M \cong \mathcal{S}$, the boundary $\partial_2\cellstar{S}$ is a circle and by Theorem~\ref{all:connectedness}~(ii) the canonical quotient $Q \partial_2 \cellstar{S}$ of $\partial_2 \cellstar{S}$ is monotone.  Therefore, $Q \partial_2 \cellstar{S}$ is either a point or a circle. We treat these alternatives separately.

Assume first that $Q \partial_2 \cellstar{S}$ is a point $x$. By Theorem~\ref{upperthm5}~(ii), the preimage $(\canonical|_{\partial\start})\inv(x)$ cannot be a path separator of $\start$. As this preimage contains the  circle $\partial_2 \cellstar{S}$ (which is a path separator by our semi-contractibility assumption), it therefore has to contain one of the two complementary components of $\partial_2 \cellstar{S}$.  Consequently, either $\cellstar{S}\subset(\canonical|_{\partial\start})\inv(x)$ or $\partial\start\setminus \cellstar{S}\subset(\canonical|_{\partial\start})\inv(x)$. If $\cellstar{S}\subset(\canonical|_{\partial\start})\inv(x)$ we have $Q\cellstar{S}=\cell{S}=\{x\}$ and we are done. If, on the other hand, $\partial\start\setminus \cellstar{S}\subset(\canonical|_{\partial\start})\inv(x)$ we gain $\cell{S}=\canonical(\cellstar{S})=\canonical(\partial\start)=\partial T$. 
As $S=\{0,s\}$ for some $s\in\mathbb{Z}^3\setminus\{0\}$ this implies that $\partial T = \cell{S} = \cell{\{0,s\}} \subset \partial (T +s)$. Thus $\operatorname{int}(T)$ is a bounded component of $\rn\setminus(T+s)$. As this contradicts Lemma~\ref{compcomp}, this situation does not occur. 

Now assume that $Q \partial_2 \cellstar{S}$ is a circle. 
As the monotone image of a disk cannot be a circle (see {\it e.g.}~\cite[I, \S4, Exercise~9]{Daverman:07}; observe that a disk is unicoherent but a circle is not)  $\canonical \cellstar{S}\setminus \canonical \partial_2 \cellstar{S}=\cell{S}\setminus \canonical \partial_2 \cellstar{S}\not=\emptyset$.  Since $M$ is combinatorial, Theorem~\ref{tilingcomplex}~(i) implies that $\partial_2 \cell{S}\subset\canonical \partial_2 \cellstar{S}$, hence, $\cell{S}\setminus \partial_2 \cell{S}\not=\emptyset$. As $\cell{S}$ is a proper subset of $\partial T$ (otherwise we get a contradiction to Lemma~\ref{compcomp} as above) we conclude that $\partial_2  \cell{S}\subset \canonical \partial_2 \cellstar{S}$ is a separator of $\partial T \cong \mathcal{S}$. This implies that $\partial_2  \cell{S}=\canonical \partial_2 \cellstar{S}$ since a proper subset of a circle cannot be a separator of $\mathcal{S}$.
Recall that a circle $c$ in a closed surface is homotopic to a single point if and only if it has two complementary components at least one of which is a disk. In particular, if $H:\mathbb{S}^1\times[0,1] \to D$ is a homotopy that deforms $c$ to a point then $D$ has to contain the closure of at least one of these complementary disks.
Since the circle $\partial_2 \cellstar{S}$ bounds the disk $\cellstar{S}$ there is a homotopy $H:\mathbb{S}^1\times[0,1] \to \cellstar{S}$ that deforms $\partial_2 \cellstar{S}$ to a single point in the closed surface $\partial M \cong \mathcal{S}$. Applying $Q$ to $H$ yields a homotopy $Q\circ H:\mathbb{S}^1\times[0,1] \to \cell{S}$ that deforms the circle $\partial_2 \cell{S}=\canonical \partial_2 \cellstar{S}$ to a single point in the closed surface $\partial T \cong \mathcal{S}$. As  $\cell{S}$ is the closure of one of the two complementary components of $\partial_2 \cell{S}$ it must be a closed disk.
%
%
\end{proof}

To check that $\operatorname{int}(T)$ is connected, the following easy criterion is often applicable ({\it cf.}~\cite[Proposition~13.1]{Bandt:12}).

\begin{lemma}\label{BandtIntConnected}
Let $T=T(A,\digits)$ be a self-affine $\mathbb{Z}^n$-tile. If there is a connected set $E\subset\operatorname{int}(T)$ such that $E \cap A\inv(E+d)\neq\emptyset$ holds for each $d\in \digits$ then $\operatorname{ int}(T)$ is connected.
\end{lemma}

\begin{proof}
We note that  $\operatorname{int}(T)=\bigcup_{k \in \mathbb{N}}\left ( \bigcup_{i=1}^k \bigcup_{d \in \digits_i}A^{-i}(E + d)\right )$ with $\digits_i$ as in \eqref{digitsk} is a nested union of open sets each of which is, by induction, connected.
\end{proof}

It is important to note that the homeomorphisms asserted in Theorems~\ref{upperthm5} and~\ref{th:complex}  are usually \emph{not} $\canonical$ since $\canonical$ is not necessarily injective. However, we used $\canonical$ to construct these homeomorphisms with the help of \emph{Moore's decomposition theorem} (Proposition~\ref{RSb}). 

Theorem~\ref{th:complex} requires $S$ to contain at least 2 elements. Under the given conditions we cannot expect $T=\cell{\{0\}}$ to be homeomorphic to a 3-ball even if the same is true for $\start$. 
Indeed, in Section~\ref{sec:wild} we give an example of a monotone model which is homeomorphic to a 3-ball but whose underlying self-affine $\mathbb{Z}^3$-tile is \emph{wild} and indeed not even simply connected, even though its boundary can be shown to be a 2-sphere by applying Theorem~\ref{uppercor}. In studying this example we shall prove the following result (see Proposition~\ref{prop:84}). 

\begin{theorem}\label{crumpled}
There exists a self-affine $\mathbb{Z}^3$-tile whose boundary is a 2-sphere, but which is \emph{not} homeomorphic to a 3-ball (a self-affine wild crumpled cube).
\end{theorem}

\section{Ideal tiles}\label{sec:aa}

In this section we define {\em ideal tiles}. These tiles can be constructed quite easily and they give rise to monotone models for self-affine $\zn$-tiles.

\subsection{Ideal tiles}\label{sec:bond}
Up to this point it is unclear how one could construct nontrivial models for a given self-affine $\mathbb{Z}^n$-tile. We now give necessary conditions that define an \emph{ideal tile} for a self-affine $\mathbb{Z}^n$-tile which allows to construct a monotone model (see Definition~\ref{def:model2}) for that tile.

\begin{definition}[Ideal tile]\label{def:approx}
 A $\mathbb{Z}^n$-tile $Z$  is an \emph{ideal tile} for a self-affine $\zn$-tile $T$, if $Z$ has connected interior and the following conditions hold.
 \begin{enumerate}
\item[(i)] $\complex{Z}=\complex{T}$.
\item[(ii)]  For each $S\in\complex{Z}$ the set $\cell{S}_Z$ is connected and homeomorphic to $\cell{P(S)}_Z$.
\item[(iii)] For each $S\in\complex{Z}$ each homeomorphism from $\cell{\delta S}_Z$ to $\cell{\delta P(S)}_Z$ extends to a homeomorphism from $\cell{S}_Z$ to $\cell{P(S)}_Z$.
\end{enumerate}
\end{definition}

\begin{remark}
Note that (ii) could be replaced by the following weaker statement.
\begin{enumerate}
\item[(ii)'] 
For each $S\in\complex{Z}$ the set $\cell{S}_Z$ is connected and $\cell{S}_Z=\emptyset$ if and only if $\cell{P(S)}_Z=\emptyset$. 
 \end{enumerate}
Indeed, in the proof of Theorem~\ref{bondingtheorem} (see Section~\ref{sec:polyeasy}) we see that (ii)' and (iii) imply (ii). However, to verify (iii) one needs homeomorphisms from $\cell{\delta S}_Z$ to $\cell{\delta P(S)}_Z$ and these are constructed by patching together homeomorphisms from $\cell{S'}_Z$ to $\cell{P(S')}_Z$. Thus one needs
(ii) in order to verify (iii). For this reason we decided to state Definition~\ref{def:approx} the way we did. 
\end{remark}

Note that since $Z$ is a $\zn$-tile one only needs to check the above hypotheses for all sets $S\subset\zn$ containing $0$. Condition (i) can be checked algorithmically by using so-called {\em boundary} and {\em vertex graphs} ({\em cf.\ e.g.}~\cite{ST:03} and see Remark~\ref{rem:gifs}). In most cases ideal tiles are chosen to be polyhedra. Then conditions (ii) and (iii) can be checked by direct inspection for $n=3$. For higher dimensions, in Section~\ref{sec:ballapprox} we show how to check (ii) by techniques from algebraic topology in certain cases. Lemma~\ref{pinched-ball} shows that (iii) is true for large classes of topological spaces.


\begin{example}\label{ex:idealKnuth}
Let $K=K(A,\digits)$ be Knuth's twin-dragon as defined in \eqref{eq:twin} and let $Z=M$, where $M$ is the parallelogram with vertices $(-\frac34,-\frac12)^t,(-\frac14,\frac12)^t,(\frac34,\frac12)^t,(\frac14,-\frac12)^t$ defined in Example~\ref{ex:model}. We now show that $Z$ satisfies the conditions of Definition~\ref{def:approx}. Firstly, it is clear that $Z$ has connected interior. From Example~\ref{ex:twinNeighbor} we get that $\complex{Z}=\complex{T}$ and item (i) follows. To prove (ii) we
start with the choice $S=\{0\}$. Obviously, $\cell{\{0\}}_Z$ is homeomorphic to $\mathbb{D}^2$. As we saw in Section~\ref{sec:subdivision}, $P(\{0\})=\{\{d\} \mid d \in \digits\}=\{\{(0,0)^t\}, \{(1,0)^t\}\}$. Thus
\[
\cell{P(\{0\})}_Z = \cell{ \{(0,0)^t\}, \{(1,0)^t\}}_Z = Z \cup (Z + (1,0)^t),
\]
which is a union of two parallelograms intersecting in an edge, hence, $\cell{P(\{0\})}_Z$ is homeomorphic to $\mathbb{D}^2$ and thus also to $\cell{\{0\}}_Z$. By translation invariance this implies that $\cell{P(\{s\})}_Z$ is homeomorphic to $\cell{\{s\}}_Z$ for each $s\in \mathbb{Z}^2$. 

To check (ii) for all $S$ with two elements we can confine ourselves again to sets $S$ containing $0$. In view of 
Figure~\ref{fig:twinNeighbors}, there are only six choices of such sets $S$ with nonempty $\cell{S}_Z$. We explain the verification of (ii) for the choice $S=\{(0,0)^t,(0,1)^t\}$ (the other choices can be treated in the same way). First, it is clear (see Figure~\ref{fig:twinNeighbors} again) that $\cell{\{(0,0)^t,(0,1)^t\}}_Z$ is the line connecting $(-\frac14,\frac12)^t$ and $(\frac14,\frac12)^t$. By direct calculation the complex $P(S)$ is easily seen to be equal to
\[
P(S) = \{\{(0,0)^t,(-1,-1)^t\},  \{(0,0)^t,(0,-1)^t\},   \{(1,0)^t,(0,-1)^t\} \},
\]
thus $\cell{P(S)}_Z$ is a union of three consecutive lines and, hence, $\cell{S}_Z$ and $\cell{P(S)}_Z$ are both homeomorphic to an arc. All nonempty cells corresponding to sets with three elements are single points and (ii) is checked easily also for this case (again there are six instances to check).

It remains to check item (iii). However, in our setting the cells are either disks, arcs, or points. Moreover, we know that $Z$ is combinatorial, so $\delta$ is nothing but the boundary operator. Since it is well-known that homeomorphisms from circles to circles can be extended to disks, and homeomorphisms from pairs of points to paris of points can be extended to arcs, (iii) follows immediately. Thus $Z$ is an ideal tile of $K$.
\end{example}

Further examples of ideal tiles are given in Figures~\ref{fig:T0T1} and~\ref{fig:T0G}.

\subsection{Ideal tiles and monotone models}\label{sec:polyeasy}

We now show that each ideal tile is a monotone model up to translation. To this end we need to prove the following lemma which shows that  $P$ behaves nicely with respect to the simplicial boundary operator $\delta$ defined in \eqref{xstar1} and \eqref{xstar1a}.

\begin{lemma}\label{Pdelta}
Let $T=T(A,\digits)$ be a self-affine $\zn$-tile. For each $S\in\complex{T}$ we have 
$
\delta P(S) = P(\delta S).
$
\end{lemma}

\begin{proof}
Since $\digits$ is a complete set of residue classes of $\zn/A\zn$ we have (recall the definition of the simplicial boundary operator $\delta$ for complexes in \eqref{xstar1a})
\begin{align*}
P(\delta S) &= \bigcup_{s\in \zn \setminus S} P(S \cup \{ s \})\\
&= \bigcup_{s\in \zn \setminus S} \{(p+A)(S \cup \{ s \}) \in \complex{T} \mid p \in \digits^{S \cup \{s\} }  \} \\
&= \bigcup_{s\in \zn \setminus S}\;\bigcup_{d\in \digits} \{(p+A)(S) \cup \{A s + d\}) \in \complex{T} \mid p \in \digits^{S}  \} \\
&= \bigcup_{s'\in \zn \setminus (AS+\digits)} \{(p+A)(S) \cup \{s'\}) \in \complex{T} \mid p \in \digits^{S}  \} \\
&= \delta(P(S)). \qedhere
\end{align*}
\end{proof}

We can now state and prove the main theorem of this section.

\begin{theorem}\label{bondingtheorem}
Let $Z$ be an ideal tile for a self-affine $\mathbb{Z}^n$-tile $T$ and $u\in\operatorname{int}(Z)$ then there is a homeomorphism $F$ such that $(Z-u,F)$ is a monotone model for $T$.
\end{theorem}

\begin{proof}
By Definition~\ref{def:approx} each translate of $Z$ is again an ideal tile of $T$, hence, $M=Z-u$ is an ideal tile of $T$. To prove the theorem we first construct a $\zn$-equivariant homeomorphism $f$ fixing $0$ so that $\Bee=Af$ satisfies $\Bee M = M +\digits$. Let  
\[
C_{-1}=\{S\subset\mathbb{Z}^n\mid \cell{S}_M=\emptyset \},
\]
and inductively set
\[
C_{\ell} = \{ S\in\complex{M} \mid {\delta}S \subset C_{\ell-1} \}  \qquad (\ell \ge 0).
\]
Note that by compactness of $M$, $S\in C_\ell$ for $\ell \ge 0$ implies that $S$ is finite. By condition~(ii) of Definition~\ref{def:approx}, $\cell{P(S)}_M=\emptyset$ for each $S\in C_{-1}$. Thus it is trivial to define $f$ on $\cell{C_{-1}}_M$ in a way that it is $\zn$-equivariant and satisfies $f(\cell{S}_M)=A^{-1}\cell{P(S)}_M$ for each $S\in C_{-1}$.

\begin{figure}[ht]
\includegraphics[height=2.5cm]{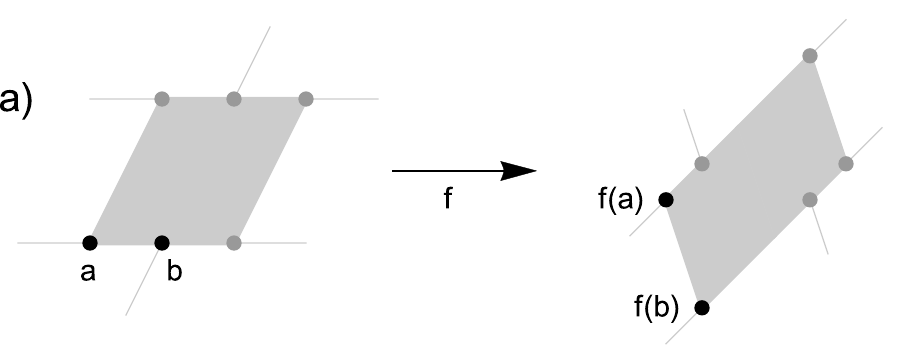}\hskip 1cm    
\includegraphics[height=2.5cm]{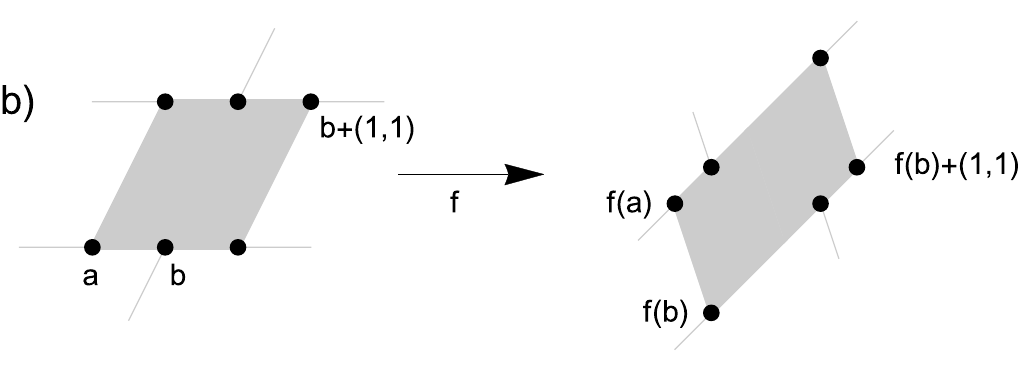}    
\includegraphics[height=2.5cm]{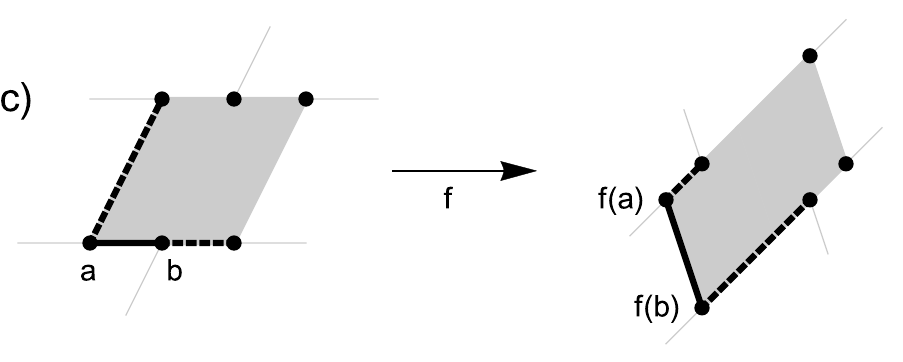}\hskip 1cm    
\includegraphics[height=2.5cm]{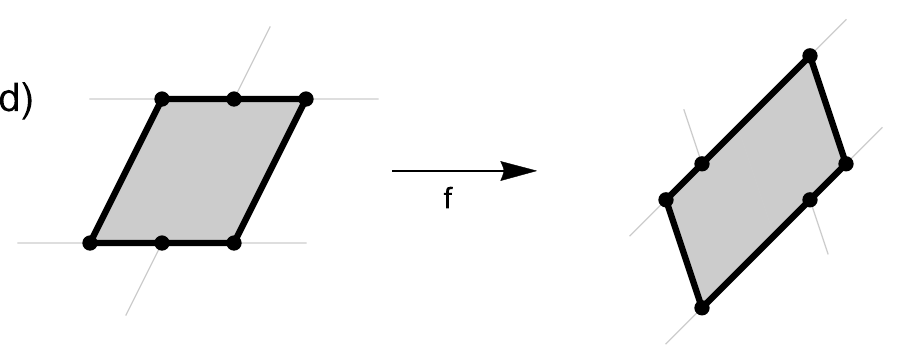}  
\caption{Illustration of the proof of Theorem~\ref{bondingtheorem} for the ideal tile $Z$ of Knuth's twin-dragon (see Example~\ref{ex:idealKnuth}). In this case $E_0$ contains two inequivalent elements corresponding to the points $a$ and $b$ in a). On these points $f$ can be defined as indicated. $f$ can then be extended equivariantly to the remaining cells $\cell{S}$ with $S\in C_0$ as in b). In c) we see that $E_1$ consists of three elements corresponding to three cells which are lines. By property (iii) of Definition~\ref{def:approx}, $f$ can now be extended to these three lines. In d) we extend $f$ equivariantly to the remaining cells $\cell{S}$ with $S\in C_1$.
\label{fig:proof}}
\end{figure}

Now assume inductively that the homeomorphism $f$ has been defined $\mathbb{Z}^n$-equivariantly on $\cell{C_{\ell-1}}_M$ such that $f(\cell{S}_M)=A^{-1}\cell{P(S)}_M$ holds for each $S\in C_{\ell-1}$. Choose a maximal set $E_{\ell}$ of pairwise \emph{inequivalent} elements of $C_{\ell}\setminus C_{\ell-1}$, that is, $E_{\ell}$ contains precisely one element of $S + \zn$ for each $S \in C_{\ell}\setminus C_{\ell-1}$.
Let $S\in E_{\ell}$. By definition, $\delta S$ is a union of cells of $C_{\ell-1}$. 
The induction hypothesis yields that each cell in ${\cell{{\delta}S}}_M$ is mapped homeomorphically to the according cell in $A^{-1}\cell{P(\delta S)}_M$. Thus $f|_{{\cell{{\delta}S}}_M}$ is a $\zn$-equivariant homeomorphism onto $A\inv\cell{P(\delta S)}_M=A\inv\cell{\delta(P(S))}_M$ (the equality follows from Lemma~\ref{Pdelta}). 
Now $f$ is defined from $\cell{\delta S}_M$ to $A\inv\cell{\delta(P(S))}_M$ in an appropriate way. By condition (iii) of the definition of the ideal tile $M$ we may extend the domain of $f$ to $\cell{S}_M$ in a way that $f$ maps $\cell{S}_M$ homeomorphically to $A\inv\cell{P(S)}_M$. 
Doing this for each of the finitely many elements of $E_{\ell}$ and extending $f$ equivariantly to $\cell{C_{\ell}}_M$ concludes the induction step. This process is illustrated in Figure~\ref{fig:proof}.

At the end of this construction we arrive at $\cell{S}_M=\cell{\{0\}}_M=M$ and $fM=A\inv(M+\digits)$ which implies that $FM=AfM=M+\digits$. In this case we need to make sure that $0$ is fixed by $f$. However, as $0\in \operatorname{int}(M)$ and  $0\in \digits$ we gain $0\in \operatorname{int}(A\inv(M+\digits))=\operatorname{int}(\cell{P(\{0\})}_M)$. Thus by the connectedness of $\operatorname{int}(M)$ we can define $f$ on $\operatorname{int}(M)$ in a way that $f(0)=0$. Thus $(M,F)$ is a model for $T$. 

Finally, Definition~\ref{def:approx}~(i) and~(ii) imply that this model is monotone.
\end{proof}

As Example~\ref{ex:idealKnuth} shows that the twin-dragon $K$ admits an ideal tile and, by Theorem~\ref{bondingtheorem}, a monotone model which is homeomorphic to $\mathbb{D}^2$, Theorem~\ref{planar} yields the following well-known result.

\begin{proposition}
Knuth's twin-dragon $K$ is homeomorphic to $\mathbb{D}^2$.
\end{proposition}

The following lemma is useful because it can be used to check property~(iii) of the definition of \emph{ideal tile} (which is Definition~\ref{def:approx}) in nonpathological examples.

\begin{lemma}\label{pinched-ball}
Let $X$ be a contractible union of cones over $n$-spheres (for varying $n\ge 1$) which pairwise intersect only at single points on their boundaries. The \emph{boundary} of $X$ is the union of the bases of the cones which comprise it. Then $X$ has the property that any self-homeomorphism of its boundary can be extended to a self-homeomorphism of $X$.
\end{lemma}
\begin{proof}
Let $C_i$ be the cones which comprise $X$.  Suppose $\psi:\partial X \to \partial X$ is a homeomorphism.  Then, by considering cut points in the boundary, we see that $\psi\mid_{\partial C_i}$ is a homeomorphism onto $\partial C_{i'}$ for some unique choice of $i'$.  Coning the map  $\psi\mid_{\partial C_i}$, we obtain an extension homeomorphism $\bar\psi\mid_{C_i}: C_i \to {C_{i'}}$ satisfying $\bar\psi\mid_{\partial C_i}= \psi\mid_{\partial C_i}.$ By again considering cut points on the boundary, we see that the map $i \to i'$ is a permutation. Since the $C_i$ only meet at single boundary points, the maps $\bar\psi\mid_{\operatorname{int}(C_i)}$ have disjoint images.  Thus the homeomorphisms $\bar\psi\mid_{C_i}$ patch together to form the desired homeomorphism $\bar\psi.$  
\end{proof}

\begin{figure}[ht]
\includegraphics[height=2.3cm]{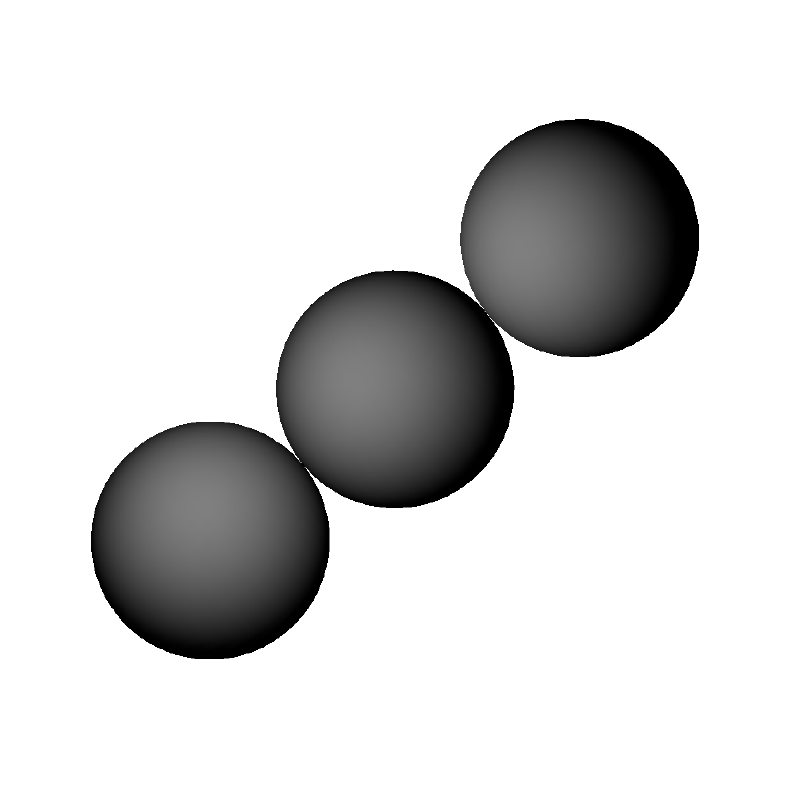}\hskip 1cm    
\includegraphics[height=2.3cm]{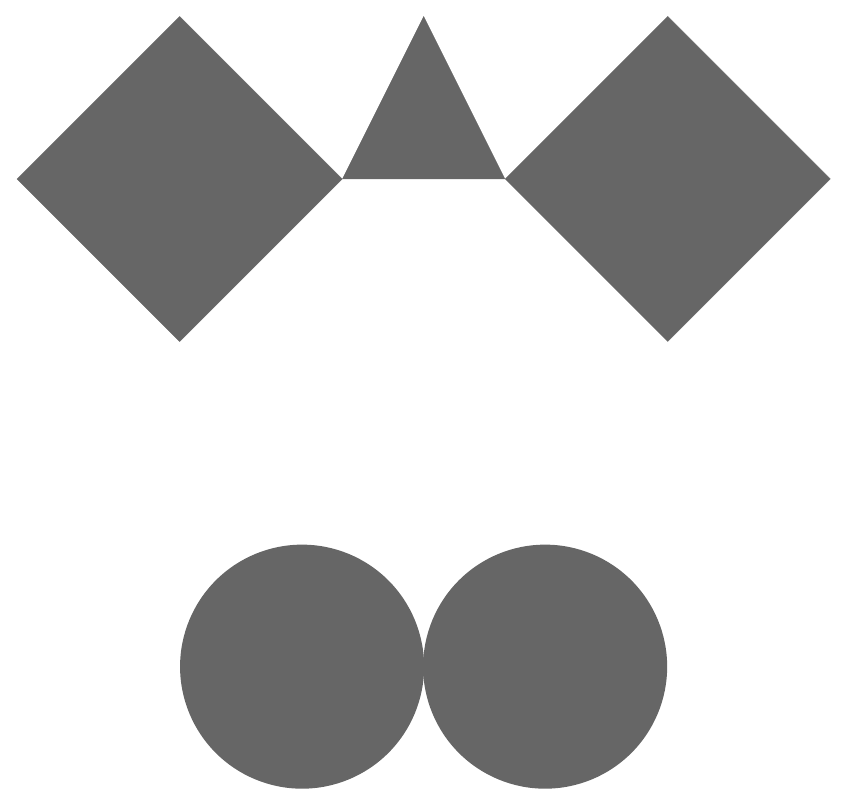} \hskip 1cm    
\includegraphics[height=2.3cm]{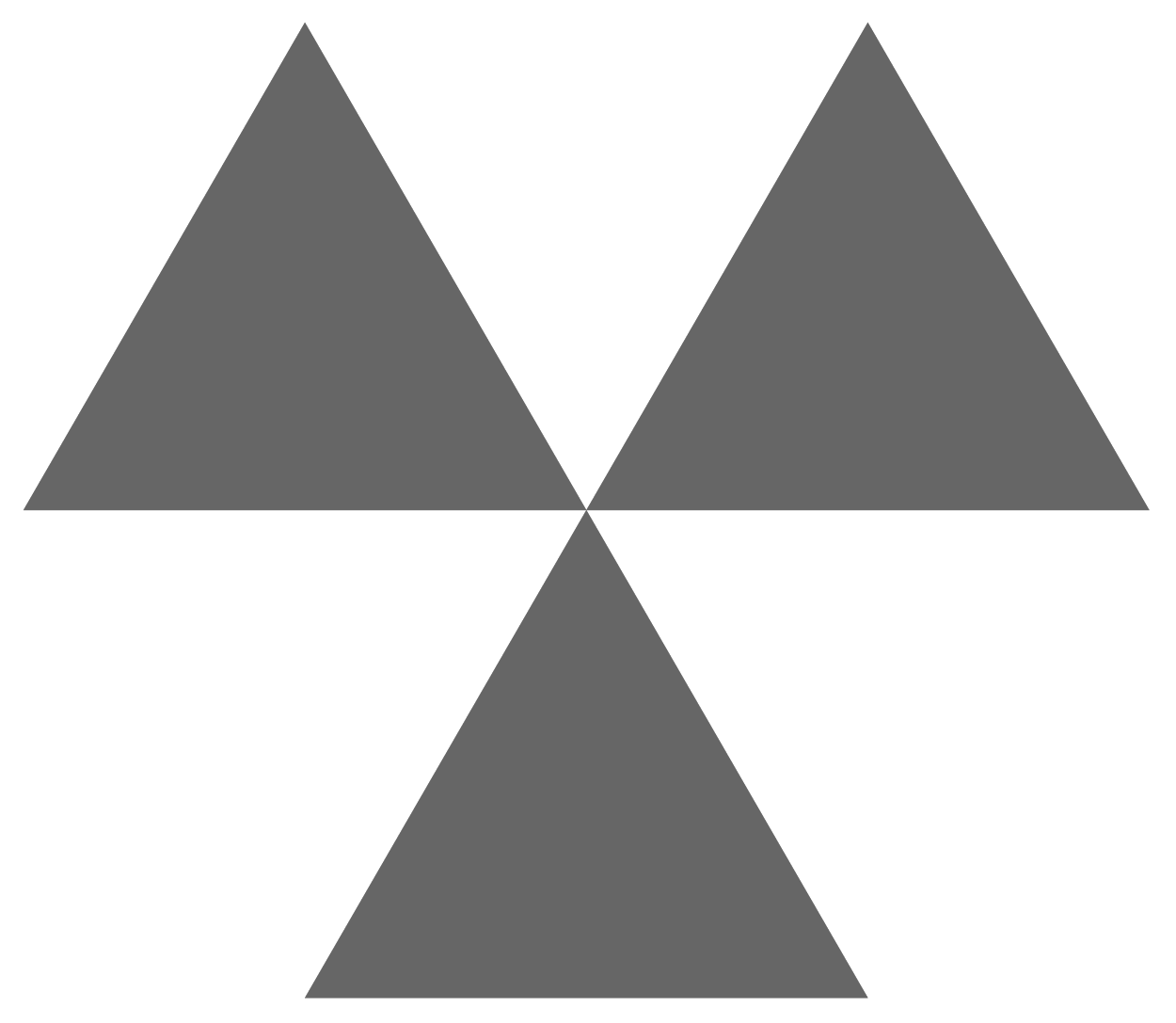}    
\caption{Four examples satisfying the conditions of Lemma~\ref{pinched-ball}.
\label{fig:6.2}}
\end{figure}

Figure~\ref{fig:6.2} contains examples for  contractible unions of cones over spheres which pairwise intersect only at single points on their boundaries. These are balls of varying dimension arranged either in a row or as a bouquet.

\subsection{Checking whether a simplicial complex is a sphere or a ball}\label{sec:ballapprox} 

Let $T$ be a self-affine $\zn$-tile. In Theorem~\ref{bondingtheorem} we require that  $Z$ is an ideal tile for $T$. For low dimensions $n$, condition~(ii) of Definition~\ref{def:approx} can be checked easily by direct inspection. However, in higher dimensions this may be tricky and we are forced to use a systematic approach based on methods from classical algebraic topology.  Such an approach is discussed in the present section. Indeed, for many instances, the ideal tile $Z$ can be chosen to be a triangulable complex such that each nonempty set ${\cell{S}}_Z$ is a closed ball that is the closure of a single cell of this complex (see {\it e.g.} the examples discussed in Sections~\ref{sec:tame} and~\ref{sec:wild}). We shall discuss how one can check Property~(ii) of Definition~\ref{def:approx} in this case. 

In particular, we have to check that ${\cell{P(S)}}_Z$ is a ball of the same dimension as ${\cell{S}}_Z$. By translation invariance it suffices to check this for all $S\in\complex{Z}$ containing $0$. The intersections ${\cell{S}}_Z$ as well as ${\cell{P(S)}}_Z$ can be nonempty only for finitely many of these sets since $Z$ as well as $\cell{P(\{0\})}_Z$ is compact. Thus there are only finitely many instances to check. By definition, ${\cell{P(S)}}_Z$ is a triangulable complex made up of finitely many triangulable complexes of the form $\cell{S'}_Z$. We have to check that the realization of this simplicial complex is a ball. To this end we first check that the realization of the boundary of this complex is a sphere. If the complex is embedded in a suitable space one can apply a result of Cannon~\cite{Cannon:73} (see Section~\ref{sec:cannon}) to prove that this sphere bounds a ball. We note that for finite simplicial complexes this criterion can be turned into a finite algorithm.

Thus we are left with checking whether the realization of a given simplicial complex is a sphere. Since this is important for several different complexes occurring throughout this paper, for example for complexes of the form $\partial[S]_M$ (see {\it e.g.} Proposition~\ref{prop:finiteunion}), we now switch to a general setting. 

In particular, let $K$ be a simplicial complex and $\sigma$ be one of its cells. Recall that the \emph{star} of $\sigma$ in $K$, denoted by $\operatorname{ st}(\sigma)$, is defined to be the set of all simplices in $K$ that have $\sigma$ as a face. Moreover, the \emph{link} of $\sigma$ in $K$, denoted by $\operatorname{ lk}(\sigma)$, is given by
\[
\operatorname{ lk}(\sigma) = \operatorname{ cl}(\operatorname{ st}(\sigma)) \setminus \operatorname{ st}(\operatorname{ cl}(\sigma)),
\]
where $\operatorname{ cl}(X)$ is the smallest subcomplex of $K$ that contains each simplex in $X$.

As we first have to check that the realization of $K$ is a manifold, we need the following result ({\it cf.} \cite[Theorem~8.10.2]{Daverman-Venema:09}; note that this is a consequence of Cannon's Double Suspension Theorem, see~\cite{Cannon:78}).

\begin{lemma}\label{linklemma}
The realization of a simplicial complex $K$ is a topological $n$-manifold if and only if, for each $k$-simplex $\sigma\in K$, $\operatorname{ lk}(\sigma)$ has the homology of $\mathbb{S}^{n-k-1}$ ({\it i.e.}, $H_i(\operatorname{ lk}(\sigma);\mathbb{Z})=0$ for $i\neq n-k-1$ and $H_{n-k-1}(\operatorname{ lk}(\sigma);\mathbb{Z})=\mathbb{Z}$) and, for each vertex $v \in K$,  $\operatorname{ lk}(v)$ is simply connected.
\end{lemma}

To check that $K$ is a ball we have to proceed as follows (for all classical theorems from algebraic topology we are using here, we refer the reader {\it e.g.} to Hatcher~\cite{Hatcher:02}):
\begin{itemize} 
\item[(i)] Check that the simplicial complex $K$ is a manifold. By Lemma~\ref{linklemma} it suffices to calculate homology groups and the fundamental group of certain links.  This can be done by using the Mayer-Vietoris-Sequence and the Seifert-van Kampen Theorem, respectively. 

\item[(ii)] Check that $K$ has trivial homology groups and trivial fundamental group. Again we use  the Mayer-Vietoris-Sequence and the Seifert-van Kampen Theorem. 

\item[(iii)] The Theorems of Hurewicz and Whitehead now imply that $K$ is homotopy equivalent to a ball. 

\item[(iv)] The generalized Poincar\'e Theorem proved by Smale~\cite{Smale:61}, Freedman~\cite{Freedman:82}, and Perelman (see \cite{Cao-Zhu:06}) yields that $K$ is a sphere. 
\end{itemize}

Summing up, checking that $K$ is a sphere of appropriate dimension is achieved by calculating homology groups and fundamental groups of simplicial complexes.

\section{Self-affine $\mathbb{Z}^n$-tiles that are homeomorphic to a ball}\label{sec:ball}

Let $T$ be a self-affine $\zn$-tile. In Theorem~\ref{upperthm5} we gave a criterion for $\partial T$ to be homeomorphic to a manifold (see also Theorem~\ref{upperthm4}).  In the present section we shall assume that $\partial T$ is homeomorphic to $\mathbb{S}^{n-1}$ and give criteria under which this implies that the tile $T$ itself is homeomorphic to the closed $n$-dimensional disk $\mathbb{D}^n$.

\subsection{Cannon's criterion and fundamental neighborhoods}\label{sec:cannon}

We start with some terminology.

\begin{definition}[1-LCC]\label{def:1lcc}
A set $X\subset \mathbb{R}^n$ is said to be {\em 1-LCC} if for each $x\in\rn\setminus X$ and each neighborhood $U$ of $x$ there is a neighborhood $V\subset U$ of $x$ such that every loop in $V\setminus X$ is contractible in $U\setminus X$.
\end{definition}

\begin{definition}[Locally spherical]\label{def:locspher}
An $(n-1)$-sphere $\mathcal{S} \subset \mathbb{R}^n$ is said to be \emph{locally spherical} if each $p\in \mathcal{S}$ has a neighborhood basis $\{U_m \mid m\in  \mathbb{N}\}$ such that $\partial U_m \cong \mathbb{S}^{n-1}$ and  $\partial U_m \setminus \mathcal{S}$ is simply connected.
\end{definition}

If $n=3$ the simple connectedness of $\partial U_m \setminus\mathcal{S}$ is equivalent to the fact that $\partial U_m \cap \mathcal{S}$ is connected. Just to give the reader a feeling for this definition we consider the standard unit ball $\mathcal{S}$ in $\mathbb{R}^3$. For each $p\in\mathcal{S}$ we can select an arbitrarily small sphere centered at $p$ that intersects $\mathcal{S}$ in a circle. Since a circle is connected, $\mathcal{S}$ is locally spherical. To see what can go wrong consider Alexander's Horned Sphere $\mathcal{A}$ ({\it cf. e.g.}~\cite[Example~2B.2, page~170ff]{Hatcher:02}). If we choose $p\in \mathcal{A}$ in a way that $p$ is located on an accumulation point of the ``horns'' of $\mathcal{A}$, each small sphere around $p$ intersects $\mathcal{A}$ in a disconnected set, the components coming from each side on which the horns approach $p$. These are hints pointing at the following criterion of Cannon~\cite{Cannon:73}.

\begin{proposition}[{\emph{cf.}~\cite[5.1 Theorem]{Cannon:73}}]\label{prop:cc}
If $\mathcal{S}$ is an $(n-1)$-sphere in $\mathbb{R}^n$ that is locally spherical then $\mathcal{S}$ is 1-LCC.
\end{proposition}

We combine this criterion with the following result.

\begin{proposition}\label{prop:bing}
If an $(n-1)$-sphere $\mathcal{S}$  in $\mathbb{R}^n$ is 1-LCC then $\mathcal{S}$ can be mapped to the standard $n$-sphere by a homeomorphism from $\rn$ to itself. In particular, $\mathcal{S}$ bounds a closed $n$-ball.
\end{proposition}

Bing~\cite{Bing:61} proved this result for $n=3$, for $n=4$ it is proved by Freedman and Quinn~\cite{FQ:90}, and for $n\ge 5$ it is due to Daverman (see \cite[Theorem~7.6.5]{Daverman-Venema:09}). We use this to prove the following theorem.

\begin{theorem}\label{thm:ballchar}
Let $T$ be a self-affine $\mathbb{Z}^3$-tile with connected interior and $1$-LCC boundary. Then $T$ is a self-affine 3-manifold if and only if it admits a semi-contractible monotone model with a boundary that is a closed 2-manifold. 
\end{theorem}

\begin{proof} By Theorem~\ref{upperthm5} $\partial T$ is a 1-LCC surface. As 1-LCC is a local property Proposition~\ref{prop:bing} implies that $\partial T\subset \mathbb{R}^3$ looks locally like $\mathbb{R}^2$ embedded in $\mathbb{R}^3$, {\it i.e.}, for each $x\in \partial T$ there is a neighborhood $U$ such that $(U,\partial T\cap U) \cong (\mathbb{R}^3,\mathbb{R}^2)$. Thus $\partial T$ bounds the closed 3-manifold $T$. 
\end{proof}

We will need the next corollary, which is just a combination of Propositions~\ref{prop:cc} and~\ref{prop:bing}.

\begin{corollary}\label{cor:cc}
If $\mathcal{S}$ is an $(n-1)$-sphere in $\mathbb{R}^n$ that is locally spherical then $\mathcal{S}$ can be mapped to the standard $n$-sphere by a homeomorphism from $\rn$ to itself. Hence, $\mathcal{S}$ bounds a closed $n$-ball.\end{corollary}

The self-affine structure of $T$ produces natural candidates for neighborhood bases as follows. 

\begin{definition}[Fundamental neighborhood]
Let $T$ be a self-affine $\zn$-tile and let $S \subset \mathbb{Z}^n$ be given in a way that
\begin{equation}\label{maximal}
\cell{S} \neq\emptyset \quad\hbox{and}\quad \langle S \cup \{s\}\rangle = \emptyset
\end{equation}
holds for each $s \in \mathbb{Z}^n \setminus S$. In this case we call the union
$[S]= \bigcup_{s\in S}(T + s) $ for $S \subset \zn$ the \emph{fundamental neighborhood} of $\cell{S}$.
\end{definition}

\begin{lemma}\label{Alemma}
Let $T$ be a self-affine $\zn$-tile. The set 
$$
\mathcal{A}=\{ A^{-k} L\mid L \hbox{ a fundamental neighborhood of some } \cell{S} \hbox{ with }S\subset\zn,\,  \cell{S} \neq\emptyset \hbox{ and } \langle S \cup \{s\}\rangle = \emptyset \}
$$ 
forms a basis for the topology of $\mathbb{R}^n$. In particular, each $x\in \partial T$ admits a neighborhood basis made up of elements of $\mathcal{A}$.
\end{lemma}

\begin{definition}[Level of a neighborhood]\label{def:level}
If $N\in \mathcal{A}$ is given in a way that $A^k N$ is a fundamental neighborhood of some $\cell{S}$ with $S\subset\zn$,  $\cell{S} \neq\emptyset$ and $\langle S \cup \{s\}\rangle = \emptyset$ for each $s\in \zn\setminus S$, we say that $N$ is {\em a neighborhood of level $k$} and write $\operatorname{ level}(N)=k$. 
\end{definition}

Lemma~\ref{Alemma} implies that the set 
$
\mathcal{B}=\{N \in \mathcal{A}\mid \operatorname{ int}(N) \cap \partial T\neq\emptyset \}
$
contains a neighborhood basis for each $x\in \partial T$. Our aim is to provide an algorithm that allows to check whether this neighborhood basis can always be chosen in a way that it meets the conditions of Corollary~\ref{cor:cc}. To this end we define the following equivalence relation on $\mathcal{B}$.

\begin{definition}[Equivalent neighborhoods]\label{def:neighbor}
Let $T$ be a self-affine $\zn$-tile, let $N_1,N_2\in\mathcal{B}$ be given and set $k_i=\operatorname{ level}(N_i)$ for $i\in\{1,2\}$. If there exists $u\in \mathbb{Z}^n$ such that
\[
A^{k_1}N_1=A^{k_2}N_2+u \quad\hbox{and}\quad A^{k_1}(N_1\cap \partial T)=A^{k_2}(N_2\cap \partial T)+u
\]
we say that \emph{$N_1$ is equivalent to $N_2$}. In this case we write $N_1\sim N_2$ to indicate that the equivalence of neighborhoods is an equivalence relation.
\end{definition}

If $N\in \mathcal{B}$ with $\operatorname{ level}(N)\ge 1$ is given, we often need a larger neighborhood in $\mathcal{B}$ that contains $N$. To this end we define
\[
\operatorname{ Parents}(N) = \{N'\in\mathcal{B}\mid N\subset N' \hbox{ and }
\operatorname{ level}(N')=\operatorname{ level}(N)-1  \}.
\]
We need the following result.

\begin{lemma}\label{lem:pred}
If $N\in\mathcal{B}$ with $\operatorname{ level}(N)\ge 1$ then $\operatorname{ Parents}(N)\neq\emptyset$.
\end{lemma}

\begin{proof}
Let $k=\operatorname{ level}(N)$ and choose $S$ in a way that $N=A^{-k}[S]$. For each $s\in \mathbb{Z}^d$ there is a unique $s'(s) \in \mathbb{Z}^d$ such that  $A^{-k}(T+s)$ is contained in $A^{-(k-1)}(T+s')$ (see Section~\ref{walkz}). Let $S'=\{s'(s)\mid s\in S)\}$. If $S'$ satisfies the maximality condition in \eqref{maximal} we are done, if not, successively add elements of $\mathbb{Z}^d$ to $S'$ until it satisfies this condition. Since $\mathbb{Z}^n$ is discrete and $T$ is compact at most finitely many elements have to be added.
\end{proof}

\subsection{The in-out graph} \label{sec:in}

Using the operator $\emph{Parents}$ we can define an infinite directed graph $\mathcal{I}$ whose nodes are the elements of $\mathcal{B}$ and whose edges are defined by
\[
N\to N'\qquad \Longleftrightarrow \qquad N' \in \operatorname{ Parents}(N).
\]
The following lemma shows that equivalent nodes in $\mathcal{I}$ have equivalent sets of predecessors with respect to the equivalence relation defined in Definition~\ref{def:neighbor}.

\begin{lemma}\label{lem:equivalence}
Let $N_1',N_2'\in \mathcal{B}$ with $N_1'\sim N_2'$. If $N_1\to N_1'$ is an edge in $\mathcal{I}$ then there is $N_2\in \mathcal{B}$ such that $N_1\sim N_2$ and $N_2 \to N_2'$ is an edge in $\mathcal{I}$.
\end{lemma}

\begin{proof}
Let $k_i:=\operatorname{ level}(N_i')$. As $N_1'\sim N_2'$ there is $u\in\mathbb{Z}^d$ such that
\begin{align}
A^{k_1}N_1'&=A^{k_2}N_2'+u, \label{eq:equiv1}\\
A^{k_1}(N_1'\cap \partial T)&=A^{k_2}(N_2'\cap \partial T)+u. \label{eq:equiv2}
\end{align}
We will show that
$
N_2 := A^{-k_2}(A^{k_1}N_1-u)
$
satisfies the requirements of our lemma. It is clear that $N_2\in\mathcal{A}$ with $\operatorname{ level}(N_2)=k_2+1$. Moreover, as $N_1 \subset N_1'$ equation \eqref{eq:equiv1} implies that $N_2\subset N_2'$. It remains to show that $N_2 \in \mathcal{B}$ and $N_1\sim N_2$. To this end observe that \eqref{eq:equiv2} yields
\begin{align*}
A^{k_1}(N_1\cap \partial T) &= A^{k_1}(N_1\cap N_1'\cap \partial T) 
= A^{k_1}N_1 \cap A^{k_1}(N_1'\cap \partial T) \\
&= A^{k_1}N_1 \cap (A^{k_2}(N_2'\cap \partial T) + u) 
= (A^{k_2}N_2+u) \cap (A^{k_2}(N_2'\cap \partial T) + u) \\
&= A^{k_2}(N_2 \cap N_2' \cap \partial T)+ u
=A^{k_2}(N_2 \cap \partial T) + u.
\end{align*}
From this we get the desired properties.
\end{proof}

\begin{definition}[In-out graph]\label{def:inout}
For $N\in \mathcal{B}$ denote by $\geometric{N}$ the equivalence class of $N$ with respect to the equivalence relation $\sim$. The \emph{in-out graph} is a directed graph $\geometric{\mathcal{I}}$ which is defined as follows.
\begin{itemize}
\item The nodes of $\geometric{\mathcal{I}}$ are the equivalence classes $\{\geometric{N}\mid N\in \mathcal{B}\}$.
\item There is a directed edge $\geometric{N}\to\geometric{N'}$ in $\geometric{\mathcal{I}}$ if there is an edge $N\to N'$ in $\mathcal{I}$.
\end{itemize}
\end{definition}

\begin{lemma}
The in-out graph $\geometric{\mathcal{I}}$ is finite.
\end{lemma}

\begin{proof}
Choose some order on $\mathbb{Z}^n$, set 
\begin{equation}\label{eq:neighbors}
\mathcal{N}=\{s\in\zn\mid\cell{\{0,s\}}\not=\emptyset\}, 
\end{equation}
and let $\digits_k$ be defined as in \eqref{digitsk}. For each finite set $Y\subset \mathbb{Z}^n$ define the
functions
\begin{equation*}
\alpha(Y)=Y-u,\quad
\beta(Y)=((Y+\mathcal{N})\cap\mathcal{D}_k)-u,\quad
\gamma(Y)=((Y+\mathcal{N})\cap(\mathbb{Z}^n\setminus \mathcal{D}_k))-u,
\end{equation*}
where $u\in Y$ is chosen to be minimal with respect to this order. Let $N=A^{-k}[S]$ be an element of $\mathcal{B}$. As $\cell{S} \neq\emptyset$ we have $\operatorname{ diam}(S) \le \operatorname{ diam}([S]) \le 2\cdot\operatorname{ diam}(T)$. Moreover, as $\langle \{s, s+ v\} \rangle \neq\emptyset$ holds for each \emph{neighbor} $v\in \mathcal{N}$ we get
\begin{equation}\label{eq:bound2}
\operatorname{ diam}(S+\mathcal{N}) \le \operatorname{ diam}([S+\mathcal{N}]) \le 4\cdot\operatorname{ diam}(T).
\end{equation}
Now pick $u\in S$ minimal with respect to the above order on $\mathbb{Z}^d$. Then
\begin{equation}\label{Ak1}
A^{k}N - u = [S] -u = [\alpha(S)].
\end{equation}
Moreover, as $A^k \partial T = [\mathcal{D}_k] \cap [\mathbb{Z}^n\setminus \mathcal{D}_k]$ we get $A^k(N\cap \partial T) = [S]\cap[\mathcal{D}_k] \cap [\mathbb{Z}^n\setminus \mathcal{D}_k]$. As $[S]\cap (T+x)=\emptyset$ whenever $x\not\in S+\mathcal{N}$, this implies that $A^k(N\cap \partial T)= [S]\cap[(S+\mathcal{N})\cap\mathcal{D}_k] \cap (S+\mathcal{N})\cap(\mathbb{Z}^n\setminus \mathcal{D}_k)]$. Subtracting $u$ this yields
\begin{equation}
A^k(N\cap \partial T) - u = [\alpha(S)] \cap [\beta(S)] \cap [\gamma(S)]. \label{Ak2}
\end{equation}
From \eqref{Ak1} and \eqref{Ak2} we see that each equivalence class $\geometric{N}$ is completely characterized by the sets $\alpha(S)$, $\beta(S)$, $\gamma(S)$. As these sets are contained in the finite set $S+\mathcal{N}-u$, the estimate in \eqref{eq:bound2} implies that they are contained in the ball of radius $4\cdot\operatorname{ diam}(T)$ around the origin. Thus there are only finitely many choices for these sets.
\end{proof}

Using Lemma~\ref{lem:equivalence} we can provide the following algorithm to calculate $\geometric{\mathcal{I}}$.

\begin{proposition}\label{inoutalgorithm}
The in-out graph $\geometric{\mathcal{I}}$ can be constructed by the following finite recurrence process.
\begin{description}
\item[\rm Recurrence start] The equivalence class $\geometric{N}$ of each fundamental neighborhood $N$ contained in $\mathcal{B}$ is a node of $\geometric{\mathcal{I}}$.

\item[\rm Recurrence step] Suppose that $\geometric{N'}$ is a node of $\geometric{\mathcal{I}}$. For all $N$ satisfying $N' \in \operatorname{ Parents}(N)$ the  node $\geometric{N}$ together with the edge $\geometric{N} \to \geometric{N'}$ belong to $\geometric{\mathcal{I}}$.

\item[\rm End of recurrence] Iterate until no new nodes occur in a recurrence step.
\end{description}
\end{proposition}

\begin{proof}
Let $R$ be the graph constructed by this recurrence process. Obviously, each node of $R$ is also a node of $\geometric{\mathcal{I}}$. Suppose that there is a node $\geometric{N}$ of $\geometric{\mathcal{I}}$ that is not a node of $R$. Choose $N\in \mathcal{B}$ in a way that $\operatorname{ level}(N)$ is minimal with this property. As $R$ contains all equivalence classes of fundamental neighborhoods we have $\operatorname{ level}(N)\ge 1$. Let $N' \in \operatorname{ Parents}(N)$. Then $\geometric{N'}\in R$ by the choice of $N$. The recurrence step above now implies together with Lemma~\ref{lem:equivalence} that $\geometric{N}$ is a node of $\geometric{\mathcal{I}}$, a contradiction. Thus $R$ and $\geometric{\mathcal{I}}$ have the same set of nodes. Since the edges are defined in the same way, the result follows.
\end{proof}

\subsection{Results on self-affine balls}\label{sec:br} We can now prove the following theorem.

\begin{theorem}\label{thm:ball-algorithm}
Let $T$ be a self-affine $\zn$-tile. If $\partial T$ is an $(n-1)$-sphere in $\mathbb{R}^n$ and each loop in the in-out graph $\geometric{\mathcal{I}}$ contains a node $\geometric{N}$ such that
\begin{itemize}
\item[(i)] $\partial N \cong \mathbb{S}^{n-1}$,
\item[(ii)] $\partial N \setminus \partial T$ is simply connected,
\end{itemize}
then $\partial T$ is locally spherical and thus $T$ is homeomorphic to $\mathbb{D}^n$.
\end{theorem}

\begin{proof}
Let $|\geometric{\mathcal{I}}|$ be the number of nodes in $\geometric{\mathcal{I}}$ and assume that $k>|\geometric{\mathcal{I}}|$. Let $N\in\mathcal{B}$ be a neighborhood of an element $x\in \partial T$ with $\operatorname{ level}(N)=k$ and let
$
N\to N_{k-1} \to\cdots\to N_1\to N_{0} 
$
be a walk in the graph $\mathcal{I}$. The associated walk in $\geometric{\mathcal{I}}$ is
$
\geometric{N}\to \geometric{N_{k-1}} \to\cdots\to \geometric{N_1}\to \geometric{N_{0}}.
$
As $k>|\geometric{\mathcal{I}}|$ the first $|\geometric{\mathcal{I}}|$ edges of this walk contain a loop. Thus, by assumption, there is $\ell \in \{k-|\geometric{\mathcal{I}}|,\ldots, k\}$ such that $\partial N_\ell \setminus \partial T$  is simply connected and $\partial N_\ell \cong \mathbb{S}^{n-1}$.
As $k$ was arbitrary and $\ell\in\{k-|\geometric{\mathcal{I}}|,\ldots, k\}$, we constructed an arbitrarily small neighborhood $N_\ell$ of $x$ that satisfies the properties of Corollary~\ref{cor:cc}. Since $x\in \partial T$ was arbitrary, this proves the result.
\end{proof}

For $n=3$ we can simplify this by using the remark after Definition~\ref{def:locspher}. 

\begin{theorem}\label{cor:ball-algorithm}
Let $T$ be a self-affine $\mathbb{Z}^3$-tile. If $\partial T$ is a $2$-sphere in $\mathbb{R}^3$ and each loop in the in-out graph $\geometric{\mathcal{I}}$ contains a node $N$ such that
\begin{itemize}
\item[(i)] $\partial N \cong \mathbb{S}^{2}$,
\item[(ii)] $\partial N\cap  \partial T$ is connected,
\end{itemize}
then $\partial T$ is locally spherical and thus tame. Consequently, $T$ is homeomorphic to $\mathbb{D}^3$.
\end{theorem}

\begin{remark}\label{rem:basis}
It seems that the neighborhood basis $\mathcal{B}$ leads to satisfactory results if the self-affine $\zn$-tile $T$ has only face-neighbors ({\it i.e.}, neighbors that intersect $T$ is an $(n-1)$-dimensional set; see Section~\ref{sec:tame}). However, as Cannon's criterion is necessary and sufficient, for tiles that are homeomorphic to balls neighborhood bases with the locally spherical property always exist.
In Section~\ref{sec:Gelbrich} we consider a tile with ``degenerate'' neighbors. To show that this tile is homeomorphic to a ball, we will have to change the fundamental neighborhoods. 

As being locally spherical is a local property, the results of the present section can be adapted to check whether $T$ is homeomorphic to another manifold with boundary (see Section~\ref{sec:torus}).  
\end{remark}

\section{Examples}\label{sec:examples}

We now illustrate our theory by examples of self-affine $\mathbb{Z}^3$-tiles. In Section~\ref{sec:tame} we give a first example of a tile that is homeomorphic to a 3-ball. Section~\ref{sec:wild} gives an example of a self-affine $\mathbb{Z}^3$-tile whose boundary is a wild sphere, \emph{i.e.}, this self-affine $\mathbb{Z}^3$-tile is an example of a \emph{crumpled cube}. Section~\ref{sec:Gelbrich} is devoted to a tile that was already studied in 1996 by Gelbrich~\cite{Gelbrich:96}. We can now show that it is homeomorphic to a 3-ball. 
Finally, in Section~\ref{sec:torus} we state an existence result for a self-affine $\mathbb{Z}^3$-tile whose boundary is a surface of genus $g$ for each $g\in \mathbb{N}$. All the computer aided proofs in this section have been checked independently by {\tt sage} and {\tt Mathematica}. In particular, for the example in Section~\ref{sec:tame} we work out detailed proofs. 

\subsection{A self-affine $\mathbb{Z}^3$-tile that is homeomorphic to a 3-ball}\label{sec:tame}
Let
\begin{equation}\label{ex:data}
A=\begin{pmatrix}
0&0&-4\\
1&0&-2\\
0&1&-1
\end{pmatrix} \quad\hbox{and}\quad
\mathcal{D}=\left\{
\begin{pmatrix}
0\\
0\\
0
\end{pmatrix},
\begin{pmatrix}
1\\
0\\
0
\end{pmatrix},
\begin{pmatrix}
2\\
0\\
0
\end{pmatrix},
\begin{pmatrix}
3\\
0\\
0
\end{pmatrix}
\right\},
\end{equation}
and let $T\subset \mathbb{R}^3$  be the unique nonempty compact set satisfying $AT=T+\mathcal{D}$. As
$A$ is expanding and $\digits$ is a complete set of residue class representatives of $\mathbb{Z}^3/A\mathbb{Z}^3$ the set $T$ is a self-affine tile. Moreover, \cite[Corollary~6.2]{LW:97} yields that $T$ tiles $\mathbb{R}^3$ by $\mathbb{Z}^3$-translates making $T$ a self-affine $\mathbb{Z}^3$-tile. An image
of $T$ is depicted in Figure~\ref{fig:tile}. In this section we shall prove the following result.
\begin{figure}[h]
\includegraphics[height=4.5cm]{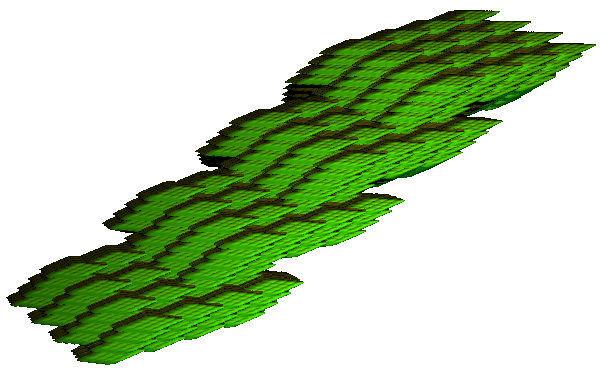}\hskip 1cm
\caption{\label{fig:tile} A self-affine $\mathbb{Z}^3$-tile that is homeomorphic to a closed ball.}
\end{figure}

\begin{theorem}\label{th:ex}
Let $T$ be the self-affine $\mathbb{Z}^3$-tile defined by the set equation $AT=T+\mathcal{D}$ with $A$ and $\mathcal{D}$ as in \eqref{ex:data}. $T$ is homeomorphic to the closed 3-ball $\mathbb{D}^3$.
\end{theorem}

Moreover, we are able to establish the following topological characterization result for the cells of $T$ (the finite graphs $\Gamma_2$, $\Gamma_3$, and $\Gamma_4$ are defined in Remark~\ref{rem:gifs} and will be constructed explicitly).

\begin{proposition}\label{intersectioncharacterization}
Let $T$ be the self-affine $\mathbb{Z}^3$-tile defined by $AT=T+\mathcal{D}$ with $A$ and $\mathcal{D}$ as in \eqref{ex:data}, and let $S\subset\mathbb{Z}^3$ with $0\in S$ be given.
\begin{itemize}
\item If $|S|=2$ then $\langle S\rangle\cong \mathbb{D}^2$ if $S$ is a node of the graph $\Gamma_2$ and $\langle S\rangle=\emptyset$ otherwise.
\item If $|S|=3$ then $\langle S\rangle\cong[0,1]$ if $S$ is a node of the graph $\Gamma_3$ and $\langle S\rangle=\emptyset$ otherwise.
\item If $|S|=4$ then $\cell{S}\cong\{0\}$ if $S$ is a node of the graph $\Gamma_4$ and $\langle S\rangle=\emptyset$ otherwise.
\item If $|S|\ge 5$ then $\langle S\rangle=\emptyset$.
\end{itemize}
\end{proposition}

Our tools in the proofs of these results will be Theorem~\ref{uppercor}, Theorem~\ref{th:complex}, and Theorem~\ref{cor:ball-algorithm}.

\medskip

\noindent
{\bf The neighbor structure of $T$.}  We will construct an ideal tile that can be used as a monotone model for $T$ by Theorem~\ref{bondingtheorem}. According to Definition~\ref{def:approx}, such an ideal tile has to have the same neighbor structure as $T$. Thus, we first determine the neighbor structure $\complex{T}$ of $T$.

\begin{figure}[ht]
\hskip 0.5cm
\xymatrix{
*[o][F-]{\st{\=100} } \ar[rr]^{1,2,3} \ar[d]^{0,1,2,3}
&
&*[o][F-]{\st{\=1\=10} } \ar[rr]^{2,3} \ar[rdd]^(0.3){1,2,3}
&
&*[o][F-]{\st{\=2\=1\=1} } \ar@/^6ex/[dddd]^{3}
\\
*[o][F-]{\st{0\=10} } \ar[rr]^(0.4){2,3} \ar[dd]^{1,2,3}
&
&*[o][F-]{\st{\=20\=1} } \ar@/_2ex/[dd]_{2,3} \ar[rr]^{3}
&
&*[o][F-]{\st{101} } \ar[u]^{0,1} \ar[ld]^{0}
\\
&*[o][F-]{\st{111} } \ar[uur]^(0.2){0}
&
&*[o][F-]{\st{\=1\=1\=1} } \ar[ddl]_(0.2){3}
&
\\
*[o][F-]{\st{\=10\=1} } \ar[d]^{2,3} \ar[ru]^{3}
&
&*[o][F-]{\st{201} } \ar@/_2ex/[uu]_{0,1} \ar[ll]_{0}
&
&*[o][F-]{\st{010} } \ar[ll]_(0.4){0,1} \ar[uu]^{0,1,2}
\\
*[o][F-]{\st{211} } \ar@/^6ex/[uuuu]^{0}
&
&*[o][F-]{\st{110} } \ar[ll]_{0,1} \ar[luu]^(0.3){0,1,2}
&
&*[o][F-]{\st{100} } \ar[ll]_{0,1,2} \ar[u]^{0,1,2,3}
}
\caption{The directed graph $\Gamma_2$ of 2-fold intersections of $T$. The triple $abc$ stands for the node $\{(0,0,0)^t,(a,b,c)^t\}$ and $\bar a=-a$. Thus $abc$ corresponds to the nonempty 2-fold intersection $T \cap (T + (a,b,c)^t)$.\label{double-graph}}
\end{figure}
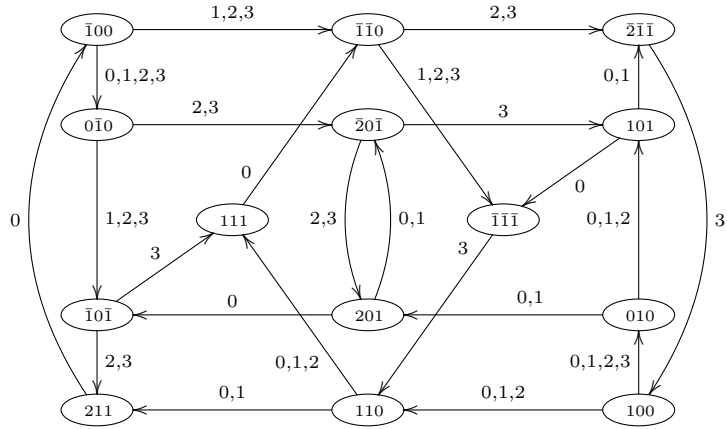

To this end we have to check which of the sets $\cell{S}$ with $S\subset\mathbb{Z}^3$ are nonempty. By translation invariance, we can confine ourselves to sets $\cell{S}$ with $0\in S$ and to characterize all nonempty cells $\cell{S}$ with $0\in S$, we use the algorithm mentioned in Remark~\ref{rem:gifs}. In particular, it suffices to construct the graphs $\Gamma_i$ ($i\ge 1$). From these graphs we can read off the nonempty cells $\cell{S}$ directly: they just correspond to their nodes. The graphs $\Gamma_i$ may be constructed by standard algorithms (see for instance~\cite{ST:03}). Indeed, $\Gamma_1$ being the graph with the single node $\{0\}$ with $4$ self-loops, we obtain the Graphs $\Gamma_2$ and $\Gamma_3$ depicted in Figures~\ref{double-graph} and~\ref{triple-graph}, respectively. Moreover, the nodes of $\Gamma_4$ are contained in Table~\ref{t4table} and $\Gamma_i$ is empty for $i\ge5$.
By inspecting the nodes of these graphs we can find all sets $S\subset\mathbb{Z}^3$ containing $0$ that correspond to a nonempty intersection $\langle S\rangle$ and, hence, by translation invariance, the neighbor structure $\complex{T}$ of $T$.

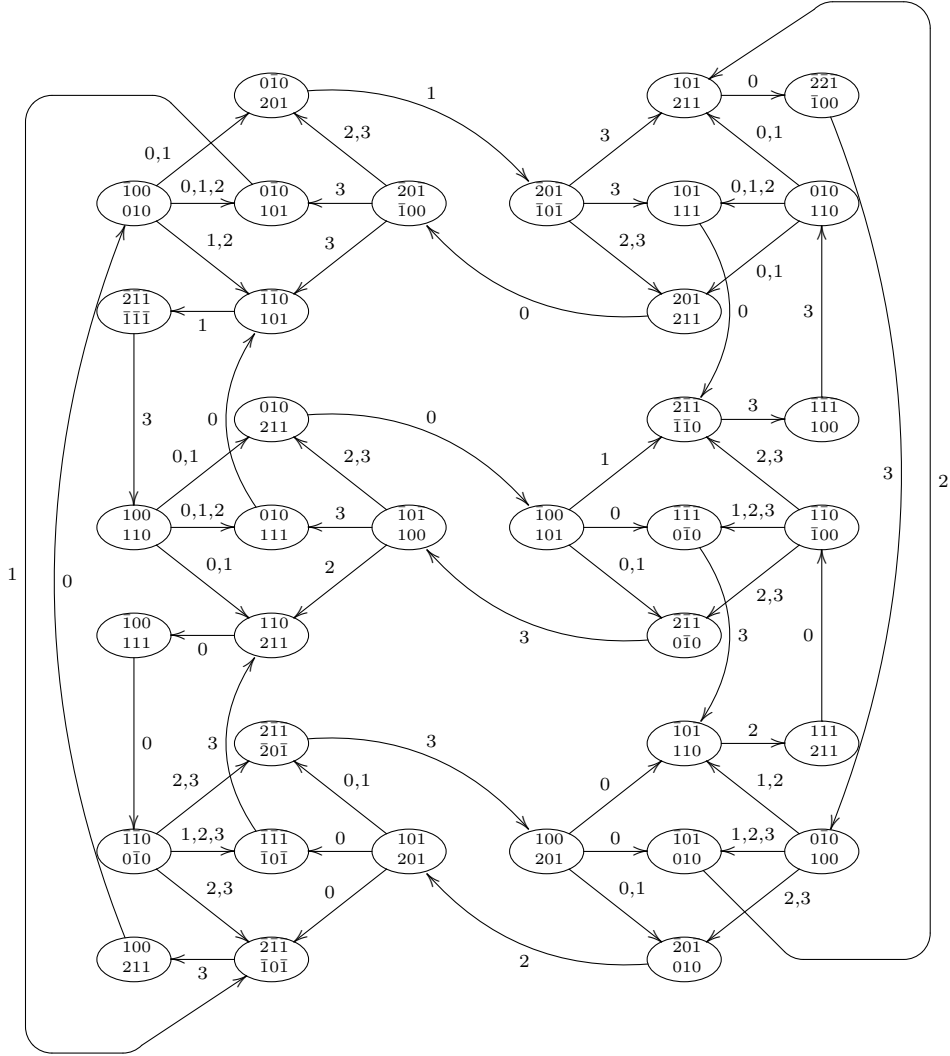
\begin{figure}[ht]
\hskip 0.5cm 
\xymatrix{
&&&&&&&\\
&
&
*[o][F-]{\state{0\=10}{201} } \ar@/^4ex/[rrd]^{1}
&
&
&
*[o][F-]{\state{101}{211} } \ar[r]^{0}
& *[o][F-]{\state{\=2\=2\=1}{\=100} }  \ar@/^7ex/[ddddddd]_{3}
&\\
&*[o][F-]{\state{\=100}{010} } \ar[rd]^{1,2} \ar[r]^{0,1,2} \ar[ru]^(0.3){0,1}
&*[o][F-]{\state{0\=10}{101} }  \ar `lu^l[lu]`^d[llu]`^r[lldddddddd]_{1}`^ru[ldddddddd][ddddddd]
&*[o][F-]{\state{\=20\=1}{\=100} } \ar[ld]_{3} \ar[l]_{3} \ar[lu]_{2,3}
&
*[o][F-]{\state{\=20\=1}{\=10\=1} } \ar[rd]^{2,3} \ar[r]^{3} \ar[ru]^{3}
&*[o][F-]{\state{101}{111} } \ar@/^4ex/[dd]^{0}
&*[o][F-]{\state{010}{110} } \ar[ld]^{0,1} \ar[l]_{0,1,2} \ar[lu]_{0,1}
&\\
&*[o][F-]{\state{\=2\=1\=1}{\=1\=1\=1} }\ar[dd]^{3}
&*[o][F-]{\state{\=1\=10}{101} } \ar[l]^{1}
&
&
&
*[o][F-]{\state{201}{211} } \ar@/^4ex/[llu]^{0}
&
&\\
&&
*[o][F-]{\state{010}{211} } \ar@/^4ex/[rrd]^{0}
&
&
&
*[o][F-]{\state{\=2\=1\=1}{\=1\=10} } \ar[r]^{3}
& *[o][F-]{\state{\=1\=1\=1}{100} }\ar[uu]^{3}
&\\
&*[o][F-]{\state{100}{110} }\ar[rd]^{0,1} \ar[r]^{0,1,2} \ar[ru]^{0,1}
&*[o][F-]{\state{010}{111} } \ar@/^4ex/[uu]^{0}
&*[o][F-]{\state{\=10\=1}{100} } \ar[ld]_{2} \ar[l]_{3} \ar[lu]_{2,3}
&
*[o][F-]{\state{\=100}{101} } \ar[rd]^{0,1} \ar[r]^{0} \ar[ru]^{1}
&*[o][F-]{\state{\=1\=1\=1}{0\=10} }\ar@/^4ex/[dd]^{3}
&*[o][F-]{\state{\=1\=10}{\=100} } \ar[ld]^{2,3} \ar[l]_{1,2,3} \ar[lu]_{2,3}
&\\
&*[o][F-]{\state{\=100}{111} }  \ar[dd]^{0}
&*[o][F-]{\state{110}{211} } \ar[l]^{0}
&
&
&
*[o][F-]{\state{\=2\=1\=1}{0\=1}0 } \ar@/^4ex/[llu]^{3}
&
&\\
&&
*[o][F-]{\state{\=2\=1\=1}{\=20\=1} }\ar@/^4ex/[rrd]^{3}
&
&
&
*[o][F-]{\state{\=10\=1}{110} } \ar[r]^{2}
& *[o][F-]{\state{111}{211} } \ar[uu]^{0}
&\\
&*[o][F-]{\state{\=1\=10}{0\=10} } \ar[rd]^{2,3} \ar[r]^{1,2,3} \ar[ru]^{2,3}
&*[o][F-]{\state{\=1\=1\=1}{\=10\=1} } \ar@/^4ex/[uu]^{3}
&*[o][F-]{\state{101}{201} } \ar[ld]_{0} \ar[l]_{0} \ar[lu]_{0,1}
&
*[o][F-]{\state{100}{201} } \ar[rd]^{0,1} \ar[r]^{0} \ar[ru]^{0}
&*[o][F-]{\state{\=10\=1}{010} } \ar `rd^r[rd]`^u[rrd]`^l[rruuuuuuuu]_{2}`^ld[ruuuuuuuu][uuuuuuu]
&*[o][F-]{\state{0\=10}{100} } \ar[ld]^(0.3){2,3} \ar[l]_{1,2,3} \ar[lu]_{1,2}
&\\
&*[o][F-]{\state{100}{211} } \ar@/^7ex/[uuuuuuu]_{0}
&*[o][F-]{\state{\=2\=1\=1}{\=10\=1}  }\ar[l]^{3}
&
&
&
*[o][F-]{\state{\=20\=1}{010} }\ar@/^4ex/[llu]^{2}
&
&\\
&&&&&&&
}
\caption{The directed graph $\Gamma_3$ of 3-fold intersections of $T$. Here a node $a_1b_1c_1\atop a_2b_2c_2$ corresponds to the intersection $T \cap (T+(a_1,b_1,c_1)^t) \cap (T+(a_2,b_2,c_2)^t) $.\label{triple-graph}}
\end{figure}

\begin{table}
\centering
\begin{tabular}{|c|c|c|c|c|c|c|c|}
\hline
\multicolumn{8}{|c|}{The nodes of $\Gamma_4$}\\
\hline
\vectv{\=2\=1\=1}{\=20\=1}{\=10\=1} &
\vectv{\=2\=1\=1}{\=20\=1}{\=100} &
\vectv{\=2\=1\=1}{\=1\=1\=1}{\=10\=1} &
\vectv{\=2\=1\=1}{\=1\=1\=1}{0\=10} &
\vectv{\=2\=1\=1}{\=1\=10}{\=100} &
\vectv{\=2\=1\=1}{\=1\=10}{0\=10} &
\vectv{\=20\=1}{\=10\=1}{010} &
\vectv{\=20\=1}{\=100}{010} \\
\hline
\vectv{\=1\=1\=1}{\=10\=1}{100} &
\vectv{\=1\=1\=1}{0\=10}{100} &
\vectv{\=1\=10}{\=100}{101} &
\vectv{\=1\=10}{0\=10}{101} &
\vectv{\=10\=1}{010}{110} &
\vectv{\=10\=1}{100}{110} &
\vectv{\=100}{010}{111} &
\vectv{\=100}{101}{111} \\
\hline
\vectv{0\=10}{100}{201} &
\vectv{0\=10}{101}{201} &
\vectv{010}{110}{211} &
\vectv{010}{111}{211} &
\vectv{100}{110}{211} &
\vectv{100}{201}{211} &
\vectv{101}{111}{211} &
\vectv{101}{201}{211}
\\
\hline
\end{tabular}

\bigskip

\caption{This table contains the nodes of $\Gamma_4$. Each node corresponds to a nonempty 4-fold intersection. There is one infinite walk starting from each node. This implies that each nonempty 4-fold intersection is a single point.}
\label{t4table}
\end{table}

\medskip

\noindent
{\bf Constructing an ideal tile for $T$.} We now build an ideal tile $Z$ of $T$. By inspecting the neighbor structure $\complex{T}$ of  $T$, we discover that choosing $Z$ to be equal to the prism spanned by the vectors $(0,1,0)^t, (1,\frac54,0)^t,(\frac32,\frac12,1)^t$ is a good candidate. The prism $Z$ and $Z_1=\cell{P(\{0\})}_Z=A\inv(Z+\mathcal{D})$ are shown in Figure~\ref{fig:T0T1}.
\begin{figure}[ht]
\includegraphics[height=5cm]{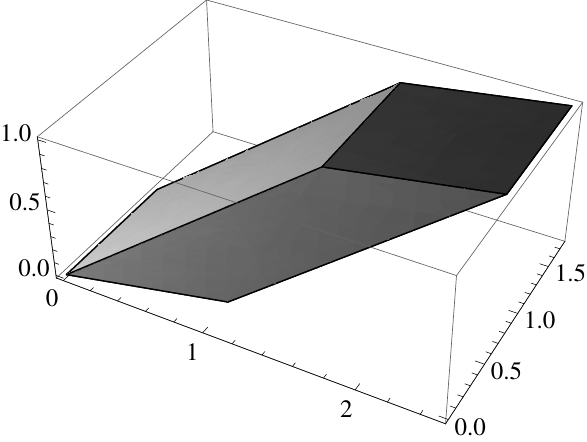}\hskip 1cm     
\includegraphics[height=5cm]{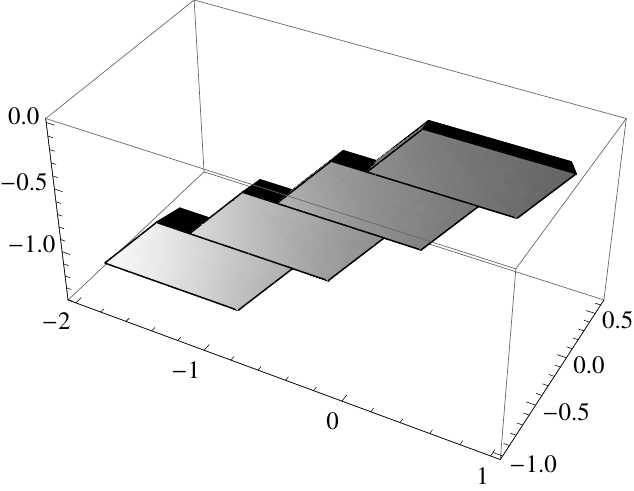}
\caption{The ideal tile $Z$ of $T$ and $\cell{P(\{0\})}_Z$.\label{fig:T0T1}}
\end{figure}
It is easy to see that $Z$ is a $\mathbb{Z}^3$-tile of $\mathbb{R}^3$ with connected interior. Thus it suffices to check items (i), (ii), and (iii) of Definition~\ref{def:approx} to make sure that $Z$ is an ideal tile for $T$.

In order to verify item~(i) we observe that the nonempty intersections $\cell{S}_Z$ with $0\in S$ can easily be determined as $Z$ is an explicitly given prism in $\mathbb{R}^3$. Comparing the collections of nonempty sets $\cell{S}$ and $\cell{S}_Z$ we obtain that $\complex{T}=\complex{Z}$ and item~(i) is verified.

To check item (ii) we need to make sure that $\cell{S}_Z$ and $\cell{P(S)}_Z$ are connected and that $\cell{S}_Z\cong\cell{P(S)}_Z$ holds for each $S\subset \mathbb{Z}^3$. Since $Z$ and $\cell{P(\{0\})}_Z$ are explicitly given polyhedra (see Figure~\ref{fig:T0T1}) it is a routine calculation to check (ii). Indeed it is easy to see that all the nonempty sets $\cell{S}_Z$ and $\cell{P(S)}_Z$ are balls of dimension $4-|S|$.

To check (iii) we observe that $Z$ is combinatorial, {\it i.e.}, for each $S\subset\mathbb{Z}^3$ with $|S|=i$, $i\ge 1$ we have $\cell{\delta S}_Z=\partial_i\cell{S}_Z$ and $\cell{\delta P(S)}_Z=\partial_i\cell{P(S)}_Z$. As each homeomorphism between the spheres $\cell{\delta S}_Z$ and $\cell{\delta P(S)}_Z$ extends to a homeomorphism between the balls $\cell{S}_Z$ and $\cell{P(S)}_Z$ item~(iii) is shown.

Therefore $Z$ is an ideal tile for $T$.

\medskip

\noindent
{\bf Sphere checking.} We are now in a position to prove that $\partial T$ is homeomorphic to the 2-sphere. 
In the preceding paragraph we proved that the prism $Z$ is an ideal tile for $T$. Thus, by Theorem~\ref{bondingtheorem} there is $u\in \mathbb{Z}^3$ such that $M=Z-u$ is a monotone model for $T$. As $\partial M$ is a 2-sphere, this monotone model is semi-contractible in the sense of Definition~\ref{def:semicontr}. To apply Theorem~\ref{uppercor} it therefore remains to check that $\operatorname{int}(T)$ is connected. To this end we use Lemma~\ref{BandtIntConnected}. To construct the set $E$ with the properties required in this lemma, note that a subtile $A^{-k}(T+s)$ of $T$ ($s\in \mathbb{Z}^3$) is contained in the interior of $T$ if and only if each of its \emph{neighbors} $A^{-k}(T+s+v)$ ($v\in \mathcal{N}$ with $\mathcal{N}$ as in \eqref{eq:neighbors}) is a subtile of $T$. The tile $T$ has $4^4$ subtiles of the shape $A^{-4}(T+s)$ ($s\in \mathbb{Z}^3$). Examining their neighbors, from this collection we select those which are contained in $\operatorname{ int}(T)$. It turns out that these are the 30 subtiles $A^{-4}(T+s)$ with $s\in I_4$ given by Table~\ref{int4Table}. 
\begin{table}
\centering
\begin{tabular}{|c|c|c|c|c|c|c|c|c|c|}
\hline
\multicolumn{5}{|c|}{The elements of $I_4:=\{s\in \mathbb{Z}^3\mid A^{-4}(T+s)\subset \operatorname{ int(T)}\}$ }\\
\hline
\vect{(-10,-4,-2)} &
\vect{(-10,-4,-1)} &
\vect{(-9,-4,-2)} &
\vect{(-9,-4,-1)} &
\vect{(-9,-4,0)} \\
\hline
\vect{(-8,-3,-2)} &
\vect{(-8,-3,-1)} &
\vect{(-8,-3,0)} &
\vect{(-7,-3,-1)} &
\vect{(-7,-3,0)} \\
\hline
\vect{(-6,-2,-1)} &
\vect{(-6,-2,0)} &
\vect{(-5,-2,-1)} &
\vect{(-5,-2,0)} &
\vect{(-5,-2,1)} \\
\hline
\vect{(-4,-1,-1)} &
\vect{(-4,-1,0)} &
\vect{(-4,-1,1)} &
\vect{(-3,-1,0)} &
\vect{(-3,-1,1)} \\
\hline
\vect{(-2,0,0)} &
\vect{(-2,0,1)} &
\vect{(-1,0,0)} &
\vect{(-1,0,1)} &
\vect{(-1,0,2)} \\
\hline
\vect{(0,1,0)} &
\vect{(0,1,1)} &
\vect{(0,1,2)} &
\vect{(1,1,1)} &
\vect{(1,1,2)} \\
\hline
\end{tabular}

\bigskip

\caption{The translates $s\in\mathbb{Z}^3$ corresponding to the subtiles $A^{-4}(T+s)$ lying in the interior of $T$.}
\label{int4Table}
\end{table}
Now set $E:= \bigcup_{s\in I_4} A^{-4}(T+s)$. As $A^{-4}(T+s)$ is connected for each $s\in I_4$, we check that $E$ is connected by showing that $\{A^{-4}(T+s)\mid s\in I_4\}$ forms a \emph{chain}. In other words, define a graph $C$ whose nodes are the elements of $I_4$. There is an edge between $s_1$ and $s_2$ if and only if
\begin{equation}\label{nono}
(T+s_1) \cap (T+s_2)\neq\emptyset.
\end{equation}
We have to show that $C$ is a connected graph. As we know that \eqref{nono} holds if and only if $s_1-s_2\in \mathcal{N}$ it is easy to set up this graph and to verify it is connected. It is now straightforward to show that $E \cap \varphi_d(E)\neq\emptyset$ for each $d\in\digits$. Applying Lemma~\ref{BandtIntConnected} we conclude that $\operatorname{ int}(T)$ is connected. 

Summing up, we may invoke Theorem~\ref{uppercor} and have thus proved the following result.

\begin{proposition}
Let $T$ be the self-affine $\mathbb{Z}^3$-tile defined by $AT=T+\mathcal{D}$ with $A$ and $\mathcal{D}$ as in \eqref{ex:data}. Then $\partial T$ is homeomorphic to the sphere $\mathbb{S}^2$.
\end{proposition}

Recall that, being a translate of $\cell{S}_Z$, each nonempty $\cell{S}_M$ is a ball of dimension $4-|S|$. Moreover, from each node in the graphs $\Gamma_2$ and $\Gamma_3$ there lead away infinitely many different infinite walks. Thus the sets $\cell{S}$ with $S$ being a node of these graphs, contain infinitely many points (and, hence, are nondegenerate). The sets $\cell{S}$, with $S$ being a node of $\Gamma_4$ are single points. Thus, since we already saw that $Z$ and, hence, $M$, is combinatorial, Theorem~\ref{th:complex} implies Proposition~\ref{intersectioncharacterization}.

\medskip

\noindent
{\bf Ball checking.}
In order prove that $T$ is homeomorphic to a ball we want to apply Theorem~\ref{cor:ball-algorithm}. To this end we have to construct the in-out graph $\geometric{\mathcal{I}}$ which can be computed by the algorithm described in Proposition~\ref{inoutalgorithm}. In the present example there exist $24$ fundamental neighborhoods in $\mathcal{B}$, one for each node of $\Gamma_4$ (see Table~\ref{t4table}). As these lie in pairwise different equivalence classes (in the sense of Definition~\ref{def:neighbor}), the recurrence starts with 24 nodes. After eight recurrence steps we arrive at the in-out graph $\geometric{\mathcal{I}}$ which has $2438$ nodes. We now have to verify conditions (i) and (ii) of Theorem~\ref{cor:ball-algorithm} to prove that $T$ is homeomorphic to a ball.

As $\Gamma_5$ is empty and each node of $\Gamma_3$ is a subset of a node of $\Gamma_4$, each set $S\subset \mathbb{Z}^3$ with the property $\cell{S}\not=\emptyset$ and $\cell{S\cup\{s\}}=\emptyset$ for all $s\in \mathbb{Z}^3\setminus S$ has exactly $4$ elements. Therefore, each fundamental neighborhood can be written as $[S] + u$, with $S\in \Gamma_4$
and $u\in \mathbb{Z}^3$ and, hence, each node $\geometric{N}$ of $\geometric{\mathcal{I}}$ is of the form $N = A^{-k}([S]+u)$, with $k\in \mathbb{N}$, $u\in\mathbb{Z}^n$, and $S\in \Gamma_4$. This implies that $N$ is homeomorphic to $[S]$ for some $S\in \Gamma_4$ and checking condition~(i) of  Theorem~\ref{cor:ball-algorithm} amounts to checking whether $\partial[S]\cong \mathbb{S}^2$ holds for each of the $24$ nodes of $\Gamma_4$.
To prove that $\partial[S]\cong \mathbb{S}^2$ we may use Proposition~\ref{prop:finiteunion}. To check the conditions of this proposition, it remains to verify the following items for each $S\in \Gamma_4$: 
\begin{itemize}
\item[(a)] $\partial[S]_M \cong \mathbb{S}^2$.
\item[(b)] $\operatorname{int}([S])$ is connected.
\item[(c)] $\mathbb{R}^3\setminus[S]$ is connected.
\end{itemize}
As $[S]_\start$ is a union of four prisms one can check (a) by direct inspection or standard methods (see Section~\ref{sec:ballapprox}). To check (b) observe that $T$ tiles $\mathbb{R}^3$ by $\mathbb{Z}^3$-translates. Thus the  definition of the fundamental neighborhood implies that the singleton $\cell{S}$ is contained in the interior of $[S]$ and, hence, there is a small open $B$ ball centered in $\cell{S}$ that is contained in $\operatorname{ int}([S])$. As $\operatorname{int}(T)$ is connected and $B$ contains inner points of $T+s$ for each $s\in S$, the interior of $[S]$ is connected. As the tiling $T+\mathbb{Z}^3$ is locally finite, also (c) can be checked combinatorially by using the connectedness of $\operatorname{int}(T)$.

Condition~(ii) of  Theorem~\ref{cor:ball-algorithm} has to be checked for each of the $2438$ nodes of $\geometric{\mathcal{I}}$. We explain how this is done for a given node of $\geometric{\mathcal{I}}$. Let $\geometric{N}$ be a node of $\geometric{\mathcal{I}}$ and consider $\partial N \cap \partial T$. Suppose $N= A^{-k}([S]+u)$, then there exist $S_1,\ldots, S_m\subset \mathbb{Z}^n$ such that
\[
\partial N \cap \partial T = A^{-k}\bigcup_{i=1}^m \langle S_i \rangle.
\]
As $\langle S_i \rangle$ is connected for each $i\in\{1,\ldots,m\}$, this set is connected if $\{S_1,\ldots,S_m\}$ forms a chain. In other words, define a graph $C(N)$ whose nodes are the sets $S_1,\ldots, S_m$ and there is an undirected edge between $S_i$ and $S_j$ if and only if $\langle S_i \rangle \cap \langle S_j\rangle  =\langle S_i \cup S_j\rangle \neq\emptyset$. All the information required to construct this graph is contained in Proposition~\ref{intersectioncharacterization}. The set $\partial N \cap \partial T$ is connected if and only if $C(N)$ is a connected finite graph. We checked connectedness for each node of $\geometric{\mathcal{I}}$ with the aid of {\tt sage} and {\tt Mathematica}. It turns out that in each walk of length 2 of $\mathcal{I}$ there is at least one node satisfying condition~(ii) of  Theorem~\ref{cor:ball-algorithm}.

Summing up, in each loop of $\geometric{\mathcal{I}}$ there is at least one node satisfying the conditions of Theorem~\ref{cor:ball-algorithm}. This proves that $T$ is homeomorphic to a closed ball and Theorem~\ref{th:ex} is established.

\subsection{A self-affine $\mathbb{Z}^3$-tile whose boundary is a wild horned sphere}\label{sec:wild}

Let $A=9I=\operatorname{ diag}(9,9,9)$ be the $3\times 3$ diagonal matrix with the number $9$ in the main diagonal. We define the set of digits $\mathcal{D}$ as follows. Let
$
C := \{(x_1,x_2,x_3)^t\mid 0\le x_1,x_2,x_3 \le 8\}
$
be the \emph{basic cube}. We construct the digit set by attaching and cutting out \emph{horns} from $C$. For the \emph{upper horns} set
\[
\begin{array}{rcl}
H_1^{u} &:=& \{(1,4,x_3)^t\mid 0\le x_3\le 6\} \cup \{(x_1,4,6)^t\mid 2\le x_1\le 7\},\\
H_2^{u} &:=& \{(7,4,x_3)^t\mid 0\le x_3\le 4\}.
\end{array}
\]
The \emph{lower horns} we define by  $H_i^{l} := \{(x_2,x_1,8-x_3)^t\mid (x_1,x_2,x_3)^t\in H_i^{u} \}$ for $i\in \{1,2\}$. Then
\begin{equation}\label{horneddigits}
\begin{array}{rcl}
\mathcal{D} &=& (C \cup (H_1^{u} + (9,0,0)^t)\cup (H_2^{u} + (9,0,0)^t)\cup (H_1^{l} - (9,0,0)^t)\cup (H_2^{l} - (9,0,0)^t))\\ &&
\hskip 0.2cm\setminus (H_1^{u}\cup H_2^{u}\cup H_1^{l}\cup H_2^{l}).
\end{array}
\end{equation}
It is easy to see that $\mathcal{D}$ has $9^3$ elements and is a complete set of coset representatives of $\mathbb{Z}^3/A\mathbb{Z}^3$. Moreover, using well-known algorithms ({\it cf.\ e.g.}\ Vince~\cite{Vince:00}), one checks that $T=T(A,\mathcal{D})$ is a self-affine $\mathbb{Z}^3$-tile. The image of $\cell{P(\{0\})}_\start$ on the left side of Figure~\ref{fighorned} gives a suggestive geometric visualization of the digit set $\mathcal{D}$. 

\begin{figure}
\includegraphics[height=8cm]{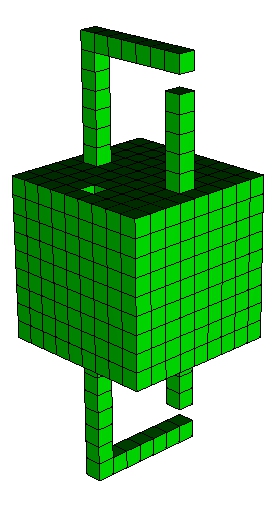}\hskip 1cm     
\includegraphics[height=8.3cm]{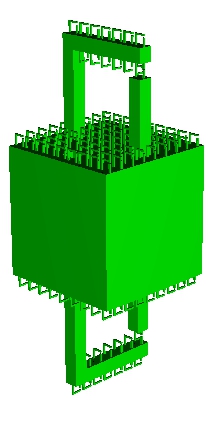}
\caption{A geometric visualization of the digits $\digits$ and an image of the self-affine {\em crumpled cube} $T=T(A,\digits)$ defined in Section~\ref{sec:wild}.\label{fighorned}}
\end{figure}

\begin{proposition}\label{prop:84}
Let the self-affine $\mathbb{Z}^3$-tile $AT=T+\mathcal{D}$  with $A=9I=\operatorname{diag}(9,9,9)$ and $\mathcal{D}$ defined as in \eqref{horneddigits} be given.
\begin{itemize}
\item[(i)] $\partial T$ is homeomorphic to $\mathbb{S}^2$.
\item[(ii)] $T$ is \emph{not} homeomorphic to a ball.
\end{itemize}
\end{proposition}

We omit routine calculations in the proof that follows. 

\begin{proof}
To prove (i) we will apply Theorem~\ref{uppercor}. To this end we first construct a monotone model.  By Theorem~\ref{bondingtheorem} it suffices to come up with an ideal tile for $T$. Set
$Z=[0,1]^3$. To show that $Z$ is an ideal tile we need to verify conditions (i), (ii) and (iii) of Definition~\ref{def:approx} ($\operatorname{int}(Z)$ is obviously connected). To check that $\complex{Z}=\complex{T}$ we have to describe which set $\cell{S}_Z$ with $0\in S$ is nonempty. As $Z$ is the unit cube, this is an easy task. To characterize the nonempty sets $\cell{S}$ we can construct the graphs $\Gamma_i$ as we did in Section~\ref{sec:tame} (indeed, in the present situation it would even be possible to check this directly). Comparing the two characterizations one sees that $\complex{Z}=\complex{T}$.  As $Z$ is a cube and $\cell{P(\{0\})}_Z$ is the complex depicted on the left hand side of Figure~\ref{fighorned}, it is easy to verify conditions (ii) and (iii) of Definition~\ref{def:approx} and it follows that $Z$ is an ideal tile. Thus, Theorem~\ref{bondingtheorem} shows that $T$ admits a monotone model $(M,F)$ with $\partial M \cong \mathbb{S}^2$. It remains to show that $\operatorname{ int}(T)$ is connected. In view of Lemma~\ref{BandtIntConnected} we construct a connected set $E\subset \operatorname{ int}(T)$ with the property that $E \cap \varphi_d(E) \neq\emptyset$ for each $d\in \mathcal{D}$. It is easy to see that the midpoint of the cube $\varphi_d(Z)$ is an element of $\operatorname{int}(T)$. For each face of $\varphi_d(Z)$ that is also contained in another
cube $\varphi_{d'}(Z)$ connect the midpoint of $\varphi_d(Z)$ to the midpoint of this face by an arc that is contained in $\operatorname{ int}(T)$. Call the union of all these arcs $Y_d$ and set $E=\bigcup_{d\in \mathcal{D}} Y_d$. One easily checks that $E$ has the required properties. 

To prove (ii) we proceed as in the classical proof for Alexander's Horned Sphere and show that the complement $\mathbb{R}^3\setminus T$ is not simply connected since we cannot homotope out a loop that surrounds one of the horns of $T$ (see {\em e.g.}~\cite[Example 2B.2, page~170ff]{Hatcher:02}).
\end{proof}

This apparently is the first example of a self-affine wild \emph{crumpled cube} that tiles $\mathbb{R}^3$ and therefore proves Theorem~\ref{crumpled}. It is of interest in the study of possible embedding types of spheres that admit a tiling of $\mathbb{R}^3$. We refer to Tang~\cite{Tang:04} (particularly to Question 2 on p.~422 of this paper) and the references given there.

\subsection{Gelbrich's twin-dragon}\label{sec:Gelbrich} This set is defined as
$T=T(A,\mathcal{D})$ with
\begin{equation}\label{ex2:data}
A=\begin{pmatrix}
0&0&2\\
1&0&1\\
0&1&-1
\end{pmatrix} \quad\hbox{and}\quad
\mathcal{D}=\left\{
\begin{pmatrix}
0\\
0\\
0
\end{pmatrix},
\begin{pmatrix}
1\\
0\\
0
\end{pmatrix}
\right\}.
\end{equation}
Again it is easy to check that $T$, which is depicted in Figure~\ref{fig:T0G}, is a self-affine $\mathbb{Z}^3$-tile. In his paper, Gelbrich~\cite{Gelbrich:96} asked whether $T$ is homeomorphic to a closed 3-dimensional ball. We are now able to answer his question in the affirmative.
\begin{figure}[ht]
\includegraphics[height=4.5cm]{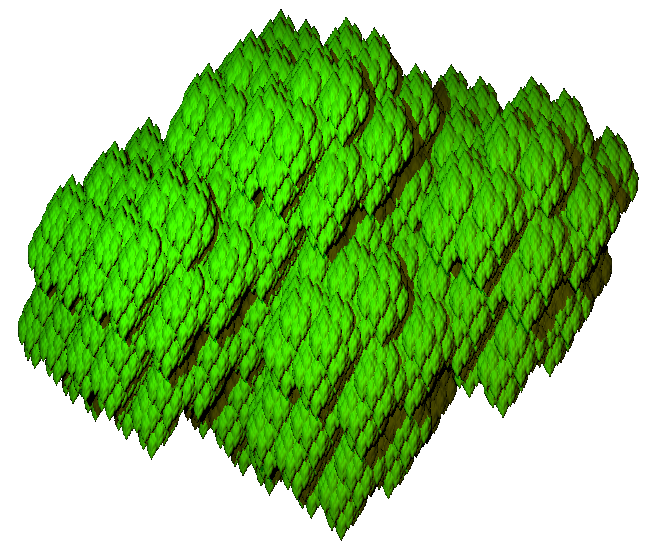}\hskip 1cm     
\includegraphics[height=4.5cm]{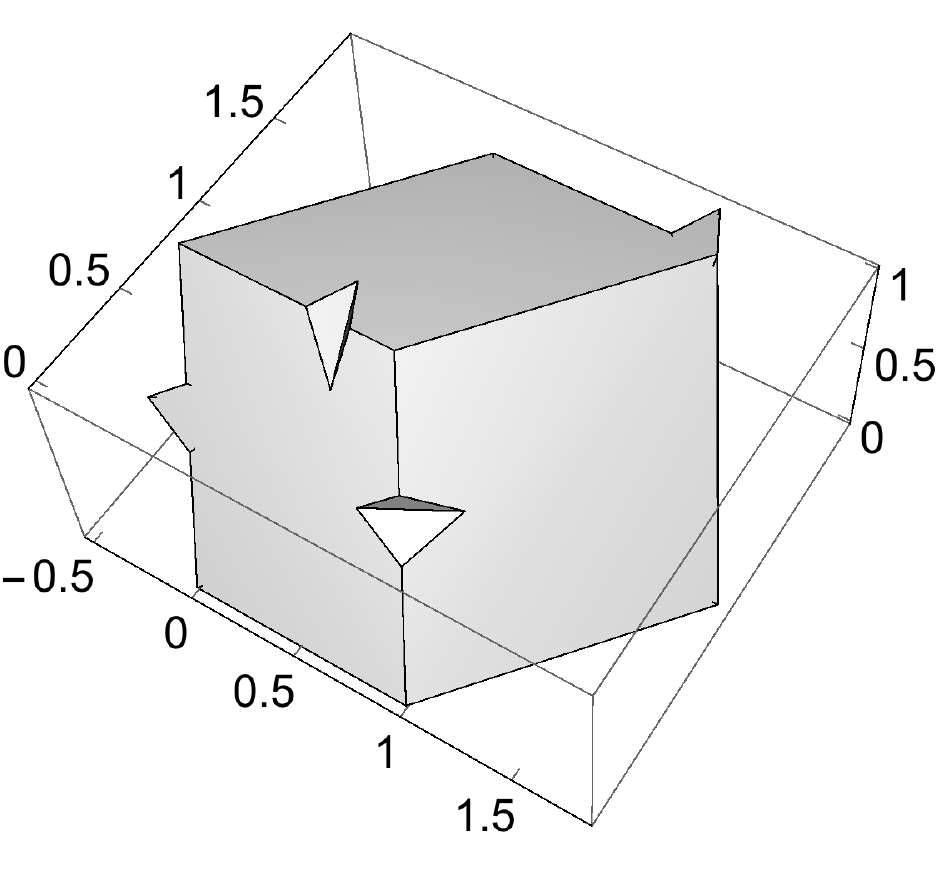}
\caption{Gelbrich's twin-dragon $T$ (left) and its ideal tile $Z$ (right).\label{fig:T0G}}
\end{figure}

\begin{theorem}\label{th:gelbrich}
Let $T$ be Gelbrich's twin-dragon which is defined by the set equation $AT=T+\mathcal{D}$ with $A$ and $\mathcal{D}$ given as in \eqref{ex2:data}. $T$ is homeomorphic to the closed 3-ball $\mathbb{D}^3$.
\end{theorem}

\begin{proof}[Sketch of the proof] This example is more complicated than the one studied in Section~\ref{sec:tame}. It has 18 neighbors, four of which correspond to single point intersections. For this reason it is harder to construct an ideal tile $Z$ for $T$ that makes Theorem~\ref{uppercor} applicable. We choose $Z$ as follows. Let $P$ be the prism spanned by the vectors $(1,0,0)^t, (\frac45,1,0)^t,(-\frac12,\frac45,1)^t$ and let
\begin{align*}
\Sigma_1 &:= \operatorname{ convex hull}\Big\{\Big(-\frac1{5}, \frac8{25}, \frac2{5}\Big)^t,\Big(-\frac3{10}, \frac{12}{25}, \frac35\Big)^t,\Big(-\frac9{100}, \frac35, \frac12\Big)^t,\Big(-\frac9{20}, \frac25, \frac12\Big)^t\Big\}, \\
 \Sigma_2 &:= \operatorname{ convex hull}\Big\{\Big(\frac{11}{10}, \frac95, 1\Big)^t,\Big(\frac{13}{10}, \frac95, 1\Big)^t,\Big(\frac{13}{10}, \frac{41}{25}, \frac45\Big)^t,\Big({\frac65, 2, 1}\Big)^t  \Big  \}.
\end{align*}
Then we set
\[
Z := P \cup \Sigma_1 \cup \Sigma_2 \setminus ((\Sigma_1 + (1,0,0)^t) \cup (\Sigma_2 - (1,1,0)^t)).
\]
A picture of $Z$ is provided in Figure~\ref{fig:T0G}. It can be checked by a lengthy but simple direct (computer aided) calculation that $Z$ satisfies the conditions Definition~\ref{def:approx} and is therefore an ideal tile for $T$: As for condition (i) it is again a matter of calculating the graphs $\Gamma_i$ $(i\ge 1)$ by known algorithms and comparing the results with the intersection structure of the polyhedron $Z$. As for condition (ii) it turns out that four of the intersections $\langle \{0,s\} \rangle_Z$ are the union of two disks intersecting in a single point so that we are not in the situation covered by Section~\ref{sec:ballapprox}. Nevertheless, condition (ii) can be checked easily by direct inspection and condition (iii) follows from Lemma~\ref{pinched-ball}.  Thus $Z$ is an ideal tile for $T$ and, hence, Theorem~\ref{bondingtheorem} implies that there is a monotone model $(M,F)$ for $T$ whose boundary is homeomorphic to $\mathbb{S}^2$.

To apply Theorem~\ref{upperthm5} it therefore remains to check that $\operatorname{int}(T)$ is connected. This is again done with the help of Lemma~\ref{BandtIntConnected}. Summing up we obtain that Gelbrich's tile $T$ satisfies $\partial T \cong \mathbb{S}^2$.

Also, running the ball-checking algorithm of Section~\ref{sec:ball} is more tricky in this case (see Remark~\ref{rem:basis}). Indeed, we have to define the fundamental neighborhoods in the following way. Let $[S]$ with $0\in S$, $|S|=4$, and $\cell{S}\not=\emptyset$ and let $\digits_k$ be defined as in \eqref{digitsk}. Then, to $S$ we associate $S_7=\{ z\in A^7s+d \mid s\in S,\, d\in \digits_7\}$, hence, the union $A^{-7}[S_7]$ corresponds to the $7^{\rm th}$ subdivision of $[S]$. Clearly $[S]=A^{-7}[S_7]$. 
We now have to avoid all single point intersections of $[S]$ with tiles $T+r$, $r\not\in S$ because otherwise too many neighborhoods have disconnected intersection with $\partial T$. To this end let 
$$
B^i_7:=\{z \in S_7 \mid \exists r\in \mathbb{Z}^3\setminus S, \, 
|A^{-7}[S_7] \cap (T+r)|=1 \hbox{ and } A^{-7}[S_7\setminus\{z\}] \cap (T+r)=\emptyset\}
$$
be those elements of $S_7$ whose corresponding subpieces have single point intersections with tiles outside $[S]$. Moreover, let
$$
B^o_7:=\{z \in \mathbb{Z}^3\setminus S_7 \mid \exists s\in S, \, 
A^{-7}[\mathbb{Z}^3\setminus S_7] \cap (T+s) \hbox{ contains an isolated point }x \hbox{ with } x\in A^{-7}[\{z\}] \}.
$$
As $[(S_7\setminus B^i_7)\cup B^o_7]$ doesn't always have spherical boundary, we have to add a subpiece for some choices of $S$. In particular, set
\[
C_7 := \begin{cases}
\{(-31,13, 31)^t\},& \hbox{for } S =\{(1, -1, -1)^t, (0, 0, 0)^t, (1, 0, -1)^t, (2, 0, -1)^t\}, \\ 
\{(-59,3,30)^t\},& \hbox{for } S= \{(1, 0, 0)^t, (1, 1, 0)^t, (0, 0, 0)^t, (2, 0, -1)^t \}, \\
\emptyset,& \hbox{otherwise},\\
\end{cases}
\]
and choose the set $
\{ A^{-7}[(S_7\setminus B^i_7)\cup B^o_7 \cup C_7] \mid 0\in S,\, |S|=4, \hbox{ and } \cell{S}\not=\emptyset \}
$
as the set of fundamental neighborhoods.
%
%
Running the ball checking algorithm with these fundamental neighborhoods yields an in-out graph with 9414 nodes. The fact that the fundamental neighborhoods have spherical boundary can be checked by Proposition~\ref{prop:finiteunion} (note that $N$ is of the form $A^{-7}[S']$ for some $S'\subset \mathbb{Z}^3$; the fact that $\partial [S']_Z \cong \mathbb{S}^2$ can be checked by using the methods discussed in Section~\ref{sec:ballapprox}), and the connectedness of their intersections with $\partial T$ are treated in the same way as in Section~\ref{sec:tame}.
\end{proof}

\subsection{Self-affine $\mathbb{Z}^3$-tiles whose boundary is a surface of positive genus}\label{sec:torus} We first construct a self-affine $\mathbb{Z}^3$-tile whose boundary is a torus. Let $A=6I=\operatorname{diag}(6,6,6)$ and define the digit set as follows. First set 
$
C=\{(x_1,x_2,x_3)^t \mid 0 \le x_1,x_2,x_3\le 5\},
$
\[
C_1 = \{(3,x_2,x_3)^t,(3,4,5)^t,(3,5,5)^t 
\mid 
2\le x_2\le 3,\, 3\le x_3\le 5\},
\]
and $C_2 = \{(x_1,5-x_2,5-x_3)^t \mid (x_1,x_2,x_3)^t\in C_1\}$. The sets $C_1$ and $C_2$ cut out a ``hole'' from the
``cube'' $C$. To make this a digit set that forms a complete set of residue classes of $\mathbb{Z}^3/A\mathbb{Z}^3$ 
we need to insert $C_1$ and $C_2$ at another place. Indeed, we define the digit set by
\begin{equation}\label{torusdigits}
\digits=(C \cup (C_1 + (0,6,0)^t) \cup (C_2-(0,6,0)^t)) \setminus (C_1\cup C_2).
\end{equation}
The self-affine $\mathbb{Z}^3$-tile $T=T(A,\digits)$ is depicted in Figure~\ref{fig:torus}. From these pictures it is plausible to assume that $T$ is a solid torus. Using Theorem~\ref{spherecor} we can prove the following result.

\begin{figure}
\includegraphics[height=8cm]{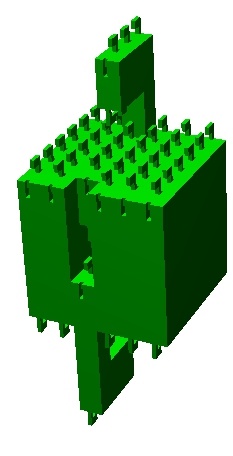} \hskip 1cm    
\includegraphics[height=8cm]{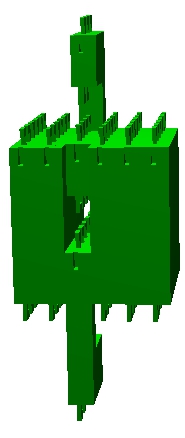}
\caption{Two images of the self-affine torus $T$ defined in Section~\ref{sec:torus}. In the right one, the hole is visible.\label{fig:torus}}
\end{figure}

\begin{proposition}\label{prop:torus}
Let the self-affine $\mathbb{Z}^3$-tile $AT=T+\mathcal{D}$  with $A=6I=\operatorname{ diag}(6,6,6)$ and $\mathcal{D}$ defined as in \eqref{torusdigits} be given. Then $\partial T$ is a 2-torus.
\end{proposition}

To construct this self-affine torus $T$, starting with a $6\times 6 \times 6$ cube, we cut out a hole and -- to compensate for the digits killed by digging this hole -- we added a ``half handle'' on the top and on the bottom of the cube. To construct boundary surfaces of genus $g$, we have to dig $g$ holes and to add $2g$ such ``half handles''. Since, after seeing the case $g=1$ above, this construction is straightforward we omit the details for the proof of the following result.

\begin{proposition}\label{prop:g}
For each genus $g\in \mathbb{N}$ there is a self-affine $\mathbb{Z}^3$-tile $T$ whose boundary is a surface of genus $g$.
\end{proposition}

Although we do not want to go into details, we mention that it is possible to show that $T$ is a self-affine 3-manifold by adapting the ball-checking algorithm provided in Section~\ref{sec:ball} (see in particular Remark~\ref{rem:basis}). The fundamental neighborhoods can be chosen to be cubes here. Thus each 3-dimensional handlebody is homeomorphic to a self-affine $\mathbb{Z}^3$-tile.

\section{Self-affine manifolds in dimension 4 and higher}\label{sec:n}

For dimension $n\ge 4$ we can still recognize self-affine $\zn$-tiles whose boundaries are {\it generalized manifolds} (see Definition~\ref{def:gm}). In the corresponding recognition theorem we need again that each point preimage of the quotient map $Q$ behaves nicely (see Theorem~\ref{upperthm3}). However, in dimension $n\ge4$, one needs more effort to guarantee that the boundary of a self-affine $\zn$-tile is a manifold.  If $n\ge 6$, the appropriate condition is the {\it disjoint disks property} (DDP) made famous by Cannon and Edwards \cite{Cannon:78,Edwards:80} (see Definition~\ref{d:dpp} and the recognition theorem in Theorem~\ref{upperthm4}).

\subsection{Cell-like maps}\label{sec:cellike}

To prove our higher dimensional results we need some preliminaries. In particular, we need to make sure that $\canonical$ is a cellular mapping in the following sense. Recall that a closed $n$-cell is the image of an $n$-dimensional closed ball under an attaching map.

\begin{definition}[Cellular and cell-like]\label{cellcell}
A compact subset $K$  of an $n$-manifold $\mathcal{M}$ is {\em cellular} in $\mathcal{M}$ if $K$ is the intersection of a properly nested decreasing sequence of closed $n$-cells in $\mathcal{M}$, {\em i.e.}, if there is a sequence $(C_i)_{i\ge 1}$ of $n$-cells such that $C_{i+1}\subset \operatorname{ int}(C_i)$ and $K=\bigcap_{i\ge 1} C_i$.  A space $X$ is {\em cell-like} if there is an embedding $\iota$ of $X$ in a manifold $\mathcal{M}$ such that $\iota(X)$ is cellular in $\mathcal{M}$.  A mapping is {\em cellular} or {\em cell-like} if its point preimages are cellular or cell-like, respectively.
\end{definition}

A simple diagonalization argument gives the following lemma.

\begin{lemma}\label{pbintcellular}
A set is cellular if it is the intersection of a properly nested decreasing sequence of cellular sets.  
\end{lemma}

To formulate the result on the cellularity of $\canonical$ we need one more definition.

\begin{definition}[Boundary star]\label{def:boundaryball}

Let $\start$ be a monotone model for a self-affine $\mathbb{Z}^n$-tile $T$. For $y\in \partial \start$ and $k\in \mathbb{N}$ define the {\em boundary star} for $y$ of level $k$ by 
\begin{equation*}
\displaystyle{\boundarystarstar{y}{k}=\bigcup_{\substack{w \in \partial W \\ y \in \stagestar{w}{k}}} \stagestar{w}{k}.}
\end{equation*}
\end{definition}

One checks that $\boundarystarstar{y}{k}$ is the closed star of $y$ in the complex $\partial \start$ induced by the sets $(F\inv)^{k}\cellstar{S}$. The following lemma contains a more convenient representation for boundary stars.

\begin{lemma}\label{walkextension}
If $(\start,\Bee)$ is a monotone model for the self-affine $\zn$-tile $T$ satisfying $\complex{M}=\complex{T}$ and $y\in \partial\start$ then
\begin{equation}\label{boundarystareq}
\boundarystarstar{y}{k}=\bigcup_{\substack{w \in \partial W \\ \canonical y = \limpt{w}}} \stagestar{w}{k}.
\end{equation}
\end{lemma}

\begin{proof}
The lemma follows if we prove that for each $w \in \partial W$ with $y \in \stagestar{w}{k}$  there exists $w' \in \partial W$ with  $\stagestar{w'}{k}=\stagestar{w}{k}$ and $\canonical y = \limpt{w'}$. To construct $w'$, let $S_i=\pi_i(w)$ for $1 \leq i \leq k$.  
Inductively, assume that for $1\leq i < j$ a choice of $S_{i} \in P(S_{i-1})$ has been made such that $y \in (F\inv)^{(i)} \cellstar{S_i}.$  By the generalized set equation for models in \eqref{xstar3}, we may choose $S_j \in P(S_{j-1})$ such that $y \in (F\inv)^{(j)}\cellstar{S_j}.$  Then $w'=(S_i)$ satisfies $\pi_k(w')=\pi_k(w)$ and, hence, $\stagestar{w'}{k}=\stagestar{w}{k}$. Moreover, we have $y\in \limptstar{w'}$ by the definition of $\limptstar{w'}$ and $\canonical y \in \limpt{w'}$ by Proposition~\ref{imageLemma}~(iii).
\end{proof}

\begin{proposition}\label{cellprop}
Let $(\start,F)$ be a monotone model for the self-affine $\zn$-tile $T$.  Assume that all but finitely many boundary stars of $M$ are cellular and that $\partial M$ is an $(n-1)$-manifold. Then $\canonical|_{\partial \start}$ is a cellular map.
\end{proposition}

\begin{proof}
We have to show that $(\canonical\mid_{\partial \start})\inv(x)$ is cellular for $x\in \partial T$. Setting $P=\{w \in \partial W \mid \limpt{w}=x\}$ Proposition~\ref{imageLemma}~(iii) implies that $\limptstar{P}=(\canonical\mid_{\partial \start})\inv(x)$. Now,  $\limptstar{P}=\bigcap_{k\ge 1} \stagestar{P}{k}$, with $\stagestar{P}{k}=\bigcup_{w \in P}(\Bee\inv)^{(k)}\cellstar{\pi_k(w)}$, is the intersection of a nested sequence.

Suppose that $\limptstar{P}\cap \partial_2 \stagestar{P}{j} \neq\emptyset$ for some $j\ge 0$. Since $(\stagestar{P}{k})$ is a nested sequence containing $\limptstar{P}$, we have that $\limptstar{P} \cap \partial_2 \stagestar{P}{k}\neq\emptyset$ for each $k\ge j$. As $\stagestar{\partial W}{k}$
covers $\partial \start$ for each fixed $k$ by the set equation for $\cellstar{S}$ in \eqref{xstar3}, there is a walk $w' \in \partial W$ such that $\canonical\limptstar{w'}\neq x $ but  $\limptstar{P} \cap \stagestar{w'}{k} \neq\emptyset$ holds for each $k\ge j$. Thus $\limptstar{P}\cap \limptstar{w'} \neq\emptyset$, which implies that there is some $w\in P$ satisfying $\limptstar{w'}\cap \limptstar{w} \neq\emptyset$. However, by Lemma~\ref{nonemptyintersection} this yields that $\canonical \limptstar{w'}=\limpt{w'}=\limpt{w}=x$, a contradiction. Thus for each $k\ge j$ we have $\limptstar{P}\subset \operatorname{ int}(\stagestar{P}{k})$ (where the interior is taken relative to $\partial M$) and we may choose a properly nested subsequence of $(\stagestar{P}{k})$.

In view of Lemma~\ref{pbintcellular} it remains to prove that $\stagestar{P}{k}$ is cellular for large $k$. To this end let $y \in (\canonical|_{\partial\start})\inv(x)$. Then Lemma~\ref{walkextension} implies that
\[
\stagestar{P}{k}= \bigcup_{w\in P} \stagestar{w}{k} =  \bigcup_{\substack{w \in \partial W \\ x = \limpt{w}}} \stagestar{w}{k}=\bigcup_{\substack{w \in \partial W \\ \canonical y = \limpt{w}}}\stagestar{w}{k} = \boundarystarstar{y}{k}
\]
and $\stagestar{P}{k}$ is cellular for large $k$ by the assumption that all but finitely many boundary stars are cellular.
\end{proof}

\subsection{Manifolds and the disjoint disks property}\label{sec:higherdims}

In the present section we deal with generalizations of Theorem~\ref{upperthm5} to higher dimensions. We are able to give a checkable criterion for the boundary of a self-affine $\mathbb{Z}^n$-tile $T$ to be a generalized manifold. 

\begin{definition}[Generalized $n$-manifold; see {\it e.g.}~\cite{Cannon:80}]\label{def:gm}
A space $X$ is a {\em generalized $n$-manifold} if it has the following properties.
\begin{itemize}
\item $X$ is a {\em Euclidean neighborhood retract} (ENR), {\em i.e.}, for some integer $n$ it embeds in $\mathbb{R}^n$ as a retract of an open subset of $\mathbb{R}^n$.
\item $X$ is a {\em homology $n$-manifold}, {\em i.e.}, $H_\ast(X,X\setminus\{x\};\mathbb{Z}) \cong H_\ast(\mathbb{R}^n,\mathbb{R}^n\setminus\{0\};\mathbb{Z})$ for each $x\in X$.
\end{itemize}
A generalized $n$-manifold is called {\em resolvable} if it is a proper cell-like upper semi-continuous decomposition of an $n$-manifold.
\end{definition}

\begin{theorem}\label{upperthm3}
Let $T$ be a self-affine $\mathbb{Z}^n$-tile which admits a monotone model $\start$. Assume that all but finitely many boundary stars of $\start$ are cellular and $\partial M$ is a manifold. Then $\partial T$ is a generalized $(n-1)$-manifold with $\canonical|_{\partial \start}:\partial \start\to \partial T$ a cellular quotient map from the manifold $ \partial \start$. In other words, $\partial \start$ is a cell-like resolution of the generalized manifold $\partial T$.
\end{theorem}

\begin{proof}
By assumption, $\start$ satisfies the conditions of Proposition~\ref{cellprop} and thus $\canonical|_{\partial \start}:\partial \start\to \partial T$ is a cellular quotient map. Lacher~\cite[(11.2)~Corollary]{Lacher:77} implies that a cell-like image of a compact manifold is an ENR, and \cite[Proposition~8.5.1]{Daverman-Venema:09} states that every $n$-dimensional resolvable space is an $n$-dimensional homology manifold. This implies the result.
\end{proof}

\begin{remark}\label{remboundaryfinite}
Let $(\start,F)$ be a monotone model for a self-affine $\zn$-tile $T=T(A,\digits)$. Although $(F^{-1})^{(k)}([S]_\start)$, $S\subset \mathbb{Z}^n$, do not necessarily form a basis for the topology of $\mathbb{R}^n$, by the same arguments as in Section~\ref{sec:ball}, a finite in-out graph $\geometric{\mathcal{I}}_\start$ can be constructed also for $M$. Since two boundary stars $B_1$ and $B_2$ are homeomorphic ({\it i.e.}, equivalent) if $A^{k_1}B_1 = A^{k_2}B_2 + u$, it suffices to check cellularity of boundary stars only for one representative of each equivalence class. The finiteness of $\geometric{\mathcal{I}}_\start$ immediately implies that there are only finitely many such equivalence classes to check. To verify cellularity of a given boundary star, the methods from Section~\ref{sec:ballapprox} can be used.
\end{remark}

To make the step from a generalized manifold to a topological manifold, we need the well-known {\it disjoint disks property} ({\it cf.} \cite{Cannon:78}).

\begin{definition}[Disjoint disks property]\label{d:dpp}
A metric space $(X,d)$ has the {\em disjoint disks property} if for every pair of maps $g_1,g_2:\mathbb{D}^2\to X$ and every $\varepsilon > 0$ there exist maps $g_1',g_2':\mathbb{D}^2\to X$ such that $\max\{d(g_1,g_1'),d(g_2,g_2')\}<\varepsilon$ and $g'_1(\mathbb{D}^2) \cap g'_2(\mathbb{D}^2)=\emptyset$.
\end{definition}

Generalizing Theorem~\ref{upperthm5} this allows us to state a result on self-affine $\zn$-tiles of dimension $n\ge 6$ whose boundary is a manifold.

\begin{theorem}\label{upperthm4}
For $n\ge 6$, let $T$ be a self-affine $\mathbb{Z}^n$-tile which admits a monotone model $\start$. Assume that all but finitely many boundary stars of $\start$ are cellular and $\partial M$ is a manifold. If $\partial T$ satisfies the disjoint disks property then $\partial T$ is an $(n-1)$-manifold.
\end{theorem}

\begin{proof}
Theorem~\ref{upperthm3} yields that $\partial T$ is a resolvable generalized $(n-1)$-manifold. As we assume that $\partial T$ satisfies the disjoint disks property, Edwards' Cell-like Approximation Theorem ({\it cf.}~\cite{Edwards:80}) implies that $\partial T$ is a manifold (see also \cite{Daverman:07} for $n-1>5$ and \cite{DH:07} for $n-1=5$).
\end{proof}

Combining this theorem with the algorithm Theorem~\ref{thm:ball-algorithm} allows us to algorithmically recognize which self-affine $\zn$-tiles are $n$-balls subject to the disjoint disks property.

For dimensions less than 5 the disjoint disks property is not suited to detect manifolds, so we cannot use it for boundaries of self-affine $\zn$-tiles of dimension $n < 6$. In Daverman and Repov\v{s}~\cite{DR:92} alternatives for the disjoint disks property for 3-manifolds are proposed; for 4-manifolds no such alternatives seem to be known so far. 

Given that a self-affine manifold tiles itself by arbitrarily small copies of itself, it seems that its topology cannot be very complicated.  To be more precise, we offer the following conjecture:

\begin{conjecture}
\label{II}\label{ndimconj}
Every self-affine $n$-manifold is homeomorphic to an $n$-dimensional handlebody.
\end{conjecture}

\medskip

\noindent
{\bf Acknowledgement.} We thank the referee for carefully reading the manuscript and valuable comments that helped to considerably improve the readability of this paper.

\bibliographystyle{abbrv}
\bibliography{BallLike}

\begin{thebibliography}{10}

\bibitem{AL:11}
S.~Akiyama and B.~Loridant.
\newblock Boundary parametrization of self-affine tiles.
\newblock {\em J. Math. Soc. Japan}, 63(2):525--579, 2011.

\bibitem{Akiyama-Thuswaldner:05}
S.~Akiyama and J.~M. Thuswaldner.
\newblock The topological structure of fractal tilings generated by quadratic
  number systems.
\newblock {\em Comput. Math. Appl.}, 49(9-10):1439--1485, 2005.

\bibitem{Bandt:12}
C.~Bandt.
\newblock Combinatorial topology of three-dimensional self-affine tiles.
\newblock preprint, available under \url{http://arxiv.org/pdf/1002.0710.pdf}.

\bibitem{BG:94}
C.~Bandt and G.~Gelbrich.
\newblock Classification of self-affine lattice tilings.
\newblock {\em J. London Math. Soc. (2)}, 50(3):581--593, 1994.

\bibitem{BM:09}
C.~Bandt and M.~Mesing.
\newblock Self-affine fractals of finite type.
\newblock In {\em Convex and fractal geometry}, volume~84 of {\em Banach Center
  Publ.}, pages 131--148. Polish Acad. Sci. Inst. Math., Warsaw, 2009.

\bibitem{BW:01}
C.~Bandt and Y.~Wang.
\newblock Disk-like self-affine tiles in {$\Bbb R^2$}.
\newblock {\em Discrete Comput. Geom.}, 26, 2001.

\bibitem{BBLT:06}
G.~Barat, V.~Berth{\'e}, P.~Liardet, and J.~Thuswaldner.
\newblock Dynamical directions in numeration.
\newblock {\em Ann. Inst. Fourier (Grenoble)}, 56(7):1987--2092, 2006.
\newblock Num{\'e}ration, pavages, substitutions.

\bibitem{Bing:61}
R.~H. Bing.
\newblock A surface is tame if its complement is {$1$}-{ULC}.
\newblock {\em Trans. Amer. Math. Soc.}, 101:294--305, 1961.

\bibitem{Cannon:73}
J.~W. Cannon.
\newblock {${\rm ULC}$} properties in neighbourhoods of embedded surfaces and
  curves in {$E^{3}$}.
\newblock {\em Canad. J. Math.}, 25:31--73, 1973.

\bibitem{Cannon:78}
J.~W. Cannon.
\newblock Shrinking cell-like decompositions of manifolds. {C}odimension three.
\newblock {\em Ann. of Math. (2)}, 110(1):83--112, 1979.

\bibitem{Cannon:80}
J.~W. Cannon.
\newblock The characterization of topological manifolds of dimension {$n\geq
  5$}.
\newblock In {\em Proceedings of the {I}nternational {C}ongress of
  {M}athematicians ({H}elsinki, 1978)}, pages 449--454, Helsinki, 1980. Acad.
  Sci. Fennica.

\bibitem{Cao-Zhu:06}
H.-D. Cao and X.-P. Zhu.
\newblock A complete proof of the {P}oincar\'e and geometrization
  conjectures---application of the {H}amilton-{P}erelman theory of the {R}icci
  flow.
\newblock {\em Asian J. Math.}, 10(2):165--492, 2006.

\bibitem{Curry:06}
E.~Curry.
\newblock Radix representations, self-affine tiles, and multivariable wavelets.
\newblock {\em Proc. Amer. Math. Soc.}, 134(8):2411--2418 (electronic), 2006.

\bibitem{Daverman:07}
R.~J. Daverman.
\newblock {\em Decompositions of manifolds}.
\newblock AMS Chelsea Publishing, Providence, RI, 2007.
\newblock Reprint of the 1986 original.

\bibitem{DH:07}
R.~J. Daverman and D.~M. Halverson.
\newblock The cell-like approximation theorem in dimension 5.
\newblock {\em Fund. Math.}, 197:81--121, 2007.

\bibitem{DR:92}
R.~J. Daverman and D.~Repov{\v{s}}.
\newblock General position properties that characterize {$3$}-manifolds.
\newblock {\em Canad. J. Math.}, 44(2):234--251, 1992.

\bibitem{Daverman-Venema:09}
R.~J. Daverman and G.~A. Venema.
\newblock {\em Embeddings in manifolds}, volume 106 of {\em Graduate Studies in
  Mathematics}.
\newblock American Mathematical Society, Providence, RI, 2009.

\bibitem{DJN:12}
D.-W. Deng, T.~Jiang, and S.-M. Ngai.
\newblock Structure of planar integral self-affine tilings.
\newblock {\em Math. Nachr.}, 285(4):447--475, 2012.

\bibitem{Dold:72}
A.~Dold.
\newblock {\em Lectures on algebraic topology}.
\newblock Springer-Verlag, New York, 1972.
\newblock Die Grundlehren der mathematischen Wissenschaften, Band 200.

\bibitem{Edwards:80}
R.~D. Edwards.
\newblock The topology of manifolds and cell-like maps.
\newblock In {\em Proceedings of the {I}nternational {C}ongress of
  {M}athematicians ({H}elsinki, 1978)}, pages 111--127, Helsinki, 1980. Acad.
  Sci. Fennica.

\bibitem{Falconer:97}
K.~J. Falconer.
\newblock {\em Techniques in Fractal Geometry}.
\newblock John Wiley and Sons, Chichester, New York, Weinheim, Brisbane,
  Singapore, Toronto, 1997.

\bibitem{Freedman:82}
M.~H. Freedman.
\newblock The topology of four-dimensional manifolds.
\newblock {\em J. Differential Geom.}, 17(3):357--453, 1982.

\bibitem{FQ:90}
M.~H. Freedman and F.~Quinn.
\newblock {\em Topology of 4-manifolds}, volume~39 of {\em Princeton
  Mathematical Series}.
\newblock Princeton University Press, Princeton, NJ, 1990.

\bibitem{GY:06}
J.-P. Gabardo and X.~Yu.
\newblock Natural tiling, lattice tiling and {L}ebesgue measure of integral
  self-affine tiles.
\newblock {\em J. London Math. Soc. (2)}, 74(1):184--204, 2006.

\bibitem{Gelbrich:96}
G.~Gelbrich.
\newblock Self-affine lattice reptiles with two pieces in {${\bf R}^n$}.
\newblock {\em Math. Nachr.}, 178:129--134, 1996.

\bibitem{GH:94}
K.~Gr{\"o}chenig and A.~Haas.
\newblock Self-similar lattice tilings.
\newblock {\em J. Fourier Anal. Appl.}, 1(2):131--170, 1994.

\bibitem{GM:92}
K.~Gr{\"o}chenig and W.~R. Madych.
\newblock Multiresolution analysis, {H}aar bases, and self-similar tilings of
  {${\bf R}^n$}.
\newblock {\em IEEE Trans. Inform. Theory}, 38(2, part 2):556--568, 1992.

\bibitem{HSV:94}
D.~Hacon, N.~C. Saldanha, and J.~J.~P. Veerman.
\newblock Remarks on self-affine tilings.
\newblock {\em Experiment. Math.}, 3(4):317--327, 1994.

\bibitem{Hatcher:02}
A.~Hatcher.
\newblock {\em Algebraic topology}.
\newblock Cambridge University Press, Cambridge, 2002.

\bibitem{Hutchinson:81}
J.~E. Hutchinson.
\newblock Fractals and self-similarity.
\newblock {\em Indiana Univ. Math. J.}, 30(5):713--747, 1981.

\bibitem{Kenyon:92}
R.~Kenyon.
\newblock Self-replicating tilings.
\newblock In {\em Symbolic dynamics and its applications ({N}ew {H}aven, {CT},
  1991)}, volume 135 of {\em Contemp. Math.}, pages 239--263. Amer. Math. Soc.,
  Providence, RI, 1992.

\bibitem{KL:00}
I.~Kirat and K.-S. Lau.
\newblock On the connectedness of self-affine tiles.
\newblock {\em J. London Math. Soc. (2)}, 62(1):291--304, 2000.

\bibitem{Knuth:98}
D.~E. Knuth.
\newblock {\em The art of computer programming. {V}ol. 2}.
\newblock Addison-Wesley, Reading, MA, 1998.
\newblock Seminumerical algorithms, Third edition [of MR0286318].

\bibitem{Kuratowski:68}
K.~Kuratowski.
\newblock {\em Topology. {V}ol. {II}}.
\newblock New edition, revised and augmented. Translated from the French by A.
  Kirkor. Academic Press, New York, 1968.

\bibitem{Lacher:77}
R.~C. Lacher.
\newblock Cell-like mappings and their generalizations.
\newblock {\em Bull. Amer. Math. Soc.}, 83(4):495--552, 1977.

\bibitem{LW:96}
J.~C. Lagarias and Y.~Wang.
\newblock Integral self-affine tiles in {$\bold R^n$}. {I}. {S}tandard and
  nonstandard digit sets.
\newblock {\em J. London Math. Soc. (2)}, 54(1):161--179, 1996.

\bibitem{LW:96a}
J.~C. Lagarias and Y.~Wang.
\newblock Self-affine tiles in {${\bf R}^n$}.
\newblock {\em Adv. Math.}, 121(1):21--49, 1996.

\bibitem{LW:97}
J.~C. Lagarias and Y.~Wang.
\newblock Integral self-affine tiles in {${\bf R}^n$}. {II}. {L}attice tilings.
\newblock {\em J. Fourier Anal. Appl.}, 3(1):83--102, 1997.

\bibitem{Lai-Lau-Rao:10}
C.-K. Lai, K.-S. Lau, and H.~Rao.
\newblock Spectral structure of digit sets of self-similar tiles on {${\Bbb
  R}^1$}.
\newblock {\em Trans. Amer. Math. Soc.}, 365(7):3831--3850, 2013.

\bibitem{LL:07}
K.-S. Leung and K.-S. Lau.
\newblock Disklikeness of planar self-affine tiles.
\newblock {\em Trans. Amer. Math. Soc.}, 359(7):3337--3355, 2007.

\bibitem{LRT:02}
J.~Luo, H.~Rao, and B.~Tan.
\newblock Topological structure of self-similar sets.
\newblock {\em Fractals}, 10(2):223--227, 2002.

\bibitem{Luo-Zhou:04}
J.~Luo and Z.~L. Zhou.
\newblock Disk-like tiles derived from complex bases.
\newblock {\em Acta Math. Sin. (Engl. Ser.)}, 20(4):731--738, 2004.

\bibitem{Malone:00}
D.~Malone.
\newblock {\em Solutions to dilation equations}.
\newblock Ph.{D}. {T}hesis, University of Dublin, 2000.

\bibitem{Massey:78}
W.~S. Massey.
\newblock {\em Homology and cohomology theory}.
\newblock Marcel Dekker Inc., New York, 1978.
\newblock An approach based on Alexander-Spanier cochains, Monographs and
  Textbooks in Pure and Applied Mathematics, Vol. 46.

\bibitem{Moore:25}
R.~L. Moore.
\newblock Concerning upper semi-continuous collections of continua.
\newblock {\em Trans. Amer. Math. Soc.}, 27(4):416--428, 1925.

\bibitem{Roberts-Steenrod:38}
J.~H. Roberts and N.~E. Steenrod.
\newblock Monotone transformations of two-dimensional manifolds.
\newblock {\em Ann. of Math. (2)}, 39(4):851--862, 1938.

\bibitem{Rudnik:92}
K.~Rudnik.
\newblock Self-similar metric inverse limits of invariant geometric inverse
  sequences.
\newblock {\em Topology Appl.}, 48(1):1--17, 1992.

\bibitem{ST:03}
K.~Scheicher and J.~M. Thuswaldner.
\newblock Neighbours of self-affine tiles in lattice tilings.
\newblock In {\em Fractals in {G}raz 2001}, Trends Math., pages 241--262.
  Birkh\"auser, Basel, 2003.

\bibitem{Smale:61}
S.~Smale.
\newblock Generalized {P}oincar\'e's conjecture in dimensions greater than
  four.
\newblock {\em Ann. of Math. (2)}, 74, 1961.

\bibitem{SW:99}
R.~S. Strichartz and Y.~Wang.
\newblock Geometry of self-affine tiles. {I}.
\newblock {\em Indiana Univ. Math. J.}, 48(1):1--23, 1999.

\bibitem{Tang:04}
T.-M. Tang.
\newblock Crumpled cube and solid horned sphere space fillers.
\newblock {\em Discrete Comput. Geom.}, 31(3):421--433, 2004.

\bibitem{Thurston:89}
W.~Thurston.
\newblock Groups, tilings and finite state automata.
\newblock in: AMS Colloquium Lecture, 1989.

\bibitem{Vince:00}
A.~Vince.
\newblock Digit tiling of {E}uclidean space.
\newblock In {\em Directions in mathematical quasicrystals}, volume~13 of {\em
  CRM Monogr. Ser.}, pages 329--370. Amer. Math. Soc., Providence, RI, 2000.

\bibitem{Wang:99}
Y.~Wang.
\newblock Self-affine tiles.
\newblock In {\em Advances in wavelets ({H}ong {K}ong, 1997)}, pages 261--282.
  Springer, Singapore, 1999.

\end{thebibliography}
\end{document}